\documentclass[11pt,a4paper]{amsart}
\usepackage{amsmath,caption,booktabs,lipsum}
\numberwithin{equation}{section}
\usepackage{amsthm}
\usepackage{pdfpages}
\usepackage{amsfonts}
\usepackage{amssymb}
\usepackage{mathrsfs}
\usepackage{tikz}
\usepackage{caption}
\usepackage{environ}
\usepackage{elocalloc}
\usepackage{url}
\usepackage{caption}
\usepackage{graphicx}
\usepackage[backend = bibtex, style=numeric-comp,doi=false,isbn=false,url=false]{biblatex}
\usepackage{CJKutf8}
\usepackage[normalem]{ulem}
\usepackage{xcolor}
\usepackage{etoolbox}
\usepackage{caption}
\usepackage{mathtools}
\usepackage{dcpic}
\usepackage{tikz-cd}
\usepackage[bottom]{footmisc}
\usepackage{hyperref}
\usepackage{enumitem}
\usepackage{stmaryrd}
\usetikzlibrary{positioning}
\usetikzlibrary{arrows,chains,positioning,scopes,quotes}
\usetikzlibrary{decorations.markings}
\newcommand{\hgd}{h_{\operatorname{glued}}}
\newcommand{\add}{\operatorname{Ad}}
\newcommand{\sfc}{\si^{d-3}(C)}
\newcommand{\kk}{{\mathbf{K}}}
\newcommand{\supp}{\operatorname{Supp}\,}
\newcommand{\vrpt}{{\mathbb{VRP}^2}}
\newcommand{\reg}{{\operatorname{Reg}}\,}
\newcommand{\ai}{\alpha}
\newcommand{\dimh}{\dim_{\mathcal{H}}}

\newcommand{\sing}{{\operatorname{Sing}}\,}

\newcommand{\sus}{{\operatorname{SU}(3)}}

\newcommand{\sut}{{\mathfrak{su}(3)}}

\newcommand{\ee}{{{\boldsymbol{\e}}_{\operatorname{Allard}}}}
\newcommand{\es}{\emptyset}
\newcommand{\be}{\beta}
\newcommand{\nn}{{\mathbf{n}}}
\newcommand{\Ga}{\Gamma}
\newcommand{\ga}{\gamma}

\newcommand{\de}{\delta}
\newcommand{\ww}{{\mathbf{w}}}

\newcommand{\e}{\epsilon}
\newcommand{\hm}{\mathcal{H}}
\newcommand{\lam}{\lambda}
\newcommand{\dhm}{{\textnormal{d}\hm}}

\newcommand{\Om}{\Omega}
\newcommand{\si}{\sigma}
\newcommand{\mms}{{\mathbf{MOD}}}
\newcommand{\Si}{\Sigma}

\newcommand{\rh}{\rho}
\newcommand{\ue}{U_{\e'}(N\#\ssc)}
\newcommand{\uee}{U_{\e''}(N\#\ssc)}
\newcommand{\ta}{\theta}

\newcommand{\vv}{{\mathbf{v}}}
\newcommand{\VV}{{\mathbf{V}}}
\newcommand{\soss}{\mathfrak{so}(3)}

\newcommand{\ms}{\mathbf{M}}

\newcommand{\T}{\mathbf{T}}

\newcommand{\ug}{U_{\operatorname{gluing}}}
\newcommand{\s}{\subset}

\newcommand{\cp}{^\complement}

\newcommand{\cpt}{\mathbb{CP}^2}

\newcommand{\jac}{{\operatorname{Jac}}}
\newcommand{\kn}{\ker^\perp \ed\pi}

\newcommand{\rr}{r}
\newcommand{\ssf}{{S^2_{\sqrt[4]{3}}(0)}}

\newcommand{\la}{\langle}
\newcommand{\ra}{\rangle}
\newcommand{\ov}[1]{\overline{#1}}
\newcommand{\no}[1]{\left\lVert#1\right\rVert}

\DeclarePairedDelimiter{\ri}{\la}{\ra}

\newcommand{\ed}{{\operatorname{d}}}
\newcommand{\du}{^\ast}

\newcommand{\ka}{\kappa}
\newcommand{\m}{^{-1}}
\newcommand{\ts}{\otimes}

\newcommand{\pd}{\partial}

\newcommand{\sot}{\mathfrak{so}(3)}

\newcommand{\rpt}{\mathbb{RP}^2}

\newcommand{\N}{\mathbb{N}}
\newcommand{\R}{\mathbb{R}}
\newcommand{\Z}{\mathbb{Z}}

\newcommand{\C}{\mathbb{C}}

\newcommand{\uu}{{\mathbf{u}}}

\newcommand{\injrad}{\operatorname{InjRad}}

\newcommand{\sos}{\operatorname{SO}(3)}

\newcommand{\pp}{{\mathbf{p}}}
\newcommand{\tr}{\operatorname{tr}}
\newcommand{\id}{\operatorname{id}}

\newcommand{\ssc}{\si^s(C)}

\makeatletter
\def\thm@space@setup{%
	\thm@preskip=0.2cm plus 0cm minus 0cm
	\thm@postskip=\thm@preskip % or whatever, if you don't want them to be equal
}
\makeatother
\theoremstyle{plain}
\newtheorem{thm}{Theorem}
\newtheorem{exam}{Example}[subsection]
\newtheorem{lem}[exam]{Lemma}

\newtheorem{fact}[exam]{Fact}
\newtheorem{claim}[exam]{Claim}
\newtheorem{consq}[exam]{Consequence}
\newtheorem{assump}[exam]{Assumption}

\newtheorem{defn}[exam]{Definition}
\newtheorem{rest}{Result}

\theoremstyle{definition}

\newtheorem{ques}{Question}

\title[Persistent singular subsets II]{Area-minimizing submanifolds are not generically smooth, except for geodesics, minimal surfaces and minimal hypersurfaces}
\author{Zhenhua Liu}
\dedicatory{Dedicated to Xunjing Wei}
\begin{document}
	\setlength{\abovedisplayskip}{5pt}
	\setlength{\belowdisplayskip}{5pt}
	\setlength{\abovedisplayshortskip}{5pt}
	\setlength{\belowdisplayshortskip}{5pt}
	\maketitle\vspace{-3em}
	\begin{abstract}
We prove that area-minimizing submanifolds in mod $2$ homology are not generically smooth, except in the case of geodesics, minimal surfaces and minimal hypersurfaces. This answers questions of White and Yau that ask the generic smoothness of area-minimizing submanifolds. We furthermore establish a lower bound on the Hausdorff dimension of the singular sets of area-minimizing submanifolds with respect to  open sets of Riemannian metrics. The lower bound is ${(d-3)}$ where $d$ denotes the dimension of the submanifold. As a crucial step, we prove that the cone over the Veronese minimal embedding of $\rpt$ is mod $2$ area-minimizing, settling another long-standing open problem.
	\end{abstract}
	\section{Introduction}
	The general problem of finding area-minimizing representatives of homology classes is solved by Federer and Fleming \cite{FF,WFfc}:
	\begin{quote}
		\emph{Every integral or mod $2$ homology class on a compact Riemannian manifold admits an area-minimizing representative.}
	\end{quote}
	In other words we can always find a representative of a homology class that has least area among all representatives of the same homology class. The representative in the above result is found in the category of integral currents or mod $2$ currents, which roughly speaking is the closure of the set of algebraic topological polyhedron chains under Whitney flat topology \cite{FF,WFfc} in $\Z$ or $\Z/2\Z$ coefficient, respectively. 

The crucial difference between integral and mod $2$ homology is whether we allow nonorientable representatives or not. For instance, we can never represent an integral homology class by $\rpt,$ while we can do so for mod $2$ homology classes. To signify the distinction:
\begin{defn}\label{defnno}
We will reserve the terminology unoriented to denote both orientable and non-orientable objects while oriented will be reserved for orientable objects.
\end{defn}
For instance, the least possible area of a representative of a homology class equals to area infimum of simplicial representatives, which are oriented for integral homology and nonoriented for mod $2$ homology \cite{FF,WFfc}.
	
	Calling the area-minimizing representatives submanifolds was justified due to the following theorem:
	\begin{quote}
		\emph{An area-minimizing representative of an integral or mod $2$ homology class is a smooth minimal submanifold counted with integer multiplicities or $\Z/2\Z$ multiplicities outside of a singular set that is codimension 2
countably rectifiable with respect to the submanifolds.}
	\end{quote}
	Minus the conclusion of rectifiability, the above result was proven by Federer \cite{HFts} for mod $2$ homology and by  Almgren \cite{FA} for integral homology. For mod $2$ homology, the above masterpiece theorem is due to Leon Simon \cite{LScy}. For integral homology, the above theorem is proven by De Lellis and his collaborators \cite{DS1,DS2,DS3,DMS1,DMS2,DMS3} and independently by Krummel and Wickramasekera \cite{BK1,BK2,BK3}.   Thus, we will use interchangebly the terms area-minimizing representatives of homology classes and area-minimizing submanifolds.
	
	The above results on the singular sets of area-minimizing submanifolds is sharp, as singular sets do appear in very natural settings. For integral homology, by \cite[Section 4]{MR0168727},
 	\begin{quote}
		\emph{Complex algebraic subvarieties, including singular subvarieties, are area-minimizing in the Fubini-Study metric on complex projective spaces.}
	\end{quote}
For mod $2$ homology, by \cite{FMcalv}, the chain sum of two flat subtori of a flat tori is area-minimizing mod $2$, provided they intersect orthogonally along a subtorus of codimension at least $2$ with respect to the chain.

The presence of singular sets severely restrict the geometric and analytic tools available to obtain meaningful geometric or analytic statements \cite{BZpm,BWpm}. 

Thus, Yau has asked several times the natural question of whether one can obtain everywhere smooth area-minimizing submanifolds in generic metrics or singular sets are stable, first in \cite[Problem 108]{SY} for hypersurfaces and later in \cite[Problem 9 at the end]{SYrg,SYrga} for general cases,
	\begin{ques}\label{cy}
	\emph{"In
		particular, what is the structure of the singular set of an area-minimizing variety? Is
		such a singular set stable under perturbations of the ambient metric? "}
\end{ques}
White has also raised a similar question  in \cite[Problem 5.16]{GMT}:
	\begin{ques}\label{cw}
		\emph{"One can also ask whether singularities in
			a homologically area minimizing cycle in a Riemannian manifold disappear after generic perturbations of the metric."}
	\end{ques}
	Area-minimizing $1$-dimensional submanifolds are embedded geodesics \cite{HFts}, so White's and Yau's questions have trivial answers. Besides the $1$-dimensional case, there are only two kinds of progress towards the above question. White answered the above two questions in the case of $2$-dimensional submanifolds of arbitrary codimension \cite{BWgt} based on \cite{HFts,SC,DSS1,DSS2,DSS3}:
	\begin{quote}
		\emph{In generic metrics, $2$-dimensional integral or mod $2$ area-minimizing submanifolds of arbitrary codimensions are smooth.}
	\end{quote} Here by generic metrics, we mean residual sets of smooth metrics with respect to the Whitney topology. In a similar vein, the masterful efforts of Hardt-Simon, Smale, Li-Wang, Chodosh-Mantoulidis-Schulze-Wang solved the questions in the case of low dimensional hypersurfaces \cite{HS,NS,LW,CMS1,CMS2,ZW11}
	\begin{quote}
		\emph{In generic metrics,  area-minimizing hypersurfaces of dimension at most $10$ are smooth.}
	\end{quote}
	Technically speaking, written results in the literature only show that mod $2$ area-minimizing hypersurfaces are smooth in generic metrics up to dimension $7.$ Wang \cite{ZWpc} has informed the author that results in \cite{ZW11} should also hold for mod $2$ cases in principle. In general it is a well-known conjecture that area-minimizing hypersurfaces are generically smooth in all dimensions. See \cite{ZWgfa} for progress.	

	Our discussions indicate that both the dimension and the codimension of the homology class with respect to the ambient manifold play an important role in determining the generic smoothness of area-minimizing representatives. To this end, we adopt the following notation convention.
\begin{defn}
	Let integer $d$ denote the dimension of an integral or mod $2$ homology class, and integer $c$ denote the codimension of the class with respect to the ambient manifold.
\end{defn}
For instance, $d=2$ corresponds to $2$-dimensional homology classes, and $c=1$ corresponds to hypersurface homology classes. Following this convention, the dimension and codimension of an area-minimizing representative/submanifold will also be denoted by $d$ and $c$, respectively.
	
On the other hand, negative evidences pointing toward persistence of singularities also exist. Morgan \cite{FMeu,FMcalv} first pointed out that singular set of almost orthogonal transverse intersections with $d=c=3$ should be persistent under perturbations, with later progresses by \cite{GL,ZLns} for $d,c\ge 4$. Another piece of evidence comes from area-minimizing cones with $d\ge 5,c\ge 3$, whose links are non-bounding manifolds like $\sus/\sos$ \cite{HP} and $\cpt$ \cite{ZLns}.
	
In the integral homology case, we \cite{ZLns} systematically exploit the above two kind of singularities and prove that area-minimizing submanifolds are not generically smooth when $d\ge 5,c\ge 4$ ro $3\le c\le d\le 4,$ based on Zhang's gluing of calibrations \cite{YZa,YZj}. Unfortunately, that did not quite cover all complementary cases to the status of the art generic regularity results.

Our main result in mod $2$ homology will be sharp and complement nicely the masterful generic smoothness results.	\begin{rest}\label{mts}Area-minimizing unoriented submanifolds are not generically smooth, provided $d\ge 3, c\ge 2$.
	\end{rest}
	In other words outside the case of geodesics, minimal surfaces and minimal hypersurfaces, unoriented area-minimizing submanifolds are not generically smooth.
	
For a more precise statement of our main result, let us first lay out a basic assumption that we will use throughout this manuscript.	\begin{assump}
				Assume that $M^{d+c}$ is a compact smooth unoriented $(d+c)$-dimensional manifold and $[\Si]\in H_d(M,\Z/2\Z)$ is a $d$-dimensional mod $2$ homology class of codimension $c$. 		
	\end{assump}
From now on every mentioning of homology means mod $2$ homology.	We adopt the following definition.
	\begin{defn}\label{defnm}
		For a smooth Riemannian metric $g$ on $M^{d+c}$, define $\inf \dimh\sing([\Si],g)$ to be the infimum of Hausdorff dimensions of singular sets of all area-minimizing submanifolds $T$ in $[\Si]$ with respect to the metric $g:$ 
		\begin{align*}
			\inf \dimh\sing([\Si],g)=\inf_{\{T|T\operatorname{ is area-minimizing in }[\Si]\operatorname{ in }g\}}\dim_{\mathcal{H}}\sing T,
		\end{align*} where $\sing T$ denotes the singular set of $T$, $\dimh$ is the Hausdorff dimension. Set
		\begin{align*}
		\inf	\dim_{\mathcal{H}}\sing([\Si],g)=-1,
		\end{align*}if one area-minimizing submanifold $T$ in $[\Si]$ is smooth.
	\end{defn}
	The number $\inf\dim_{\mathcal{H}}\sing([\Si],g)$ is non-negative if and only if every area-minimizing representative of $[\Si]$ is singular in the ambient metric $g$.
	%\begin{defn}	Define a subset $\sing_{[\Si]}$ in the space of smooth Riemannian metrics to be	\begin{align*}		\sing_{[\Si]}=\{\operatorname{smooth metric }g|\dim_{\mathcal{H}}\sing([\Si],g)\ge 0\}.	\end{align*}\end{defn}
	
	Our main theorem is as follows.
	\begin{thm}\label{thmm}
	For any $d$-dimensional non-zero mod $2$ class $[\Si]$ on $M^{d+c},$ there is a non-empty open subset $\Om_{[\Si]}$ of the space of smooth Riemannian metrics on $M^{d+c}$ such that for all $g\in\Om_{[\Si]}$, we have
		\begin{align}
			\inf\dimh\sing([\Si],g)\ge 0,
		\end{align}provided $d\ge3,c\ge 2.$
Furthermore if $[\Si]$ admits a smoothly embedded representative, then
\begin{align}
	\inf\dimh\sing([\Si],g)\ge d-3.
	\end{align}	\end{thm}
 For a zero homology class, any area-minimizing representative is necessarily zero. For a non-zero homology class $[\Si]$, Thom's classical result \cite{RT} shows $[\Si]$ might have no smooth topological representatives at all. For such $[\Si]$, finding a smooth area-minimizing representative is impossible due to topological obstructions alone. In this case, $\inf\dim_{\mathcal{H}}\sing([\Si],g)\ge 0$ holds for all metrics $g,$ i.e., $\Om_{[\Si]}$ equal to the space of Riemannian metrics. 

Thus, the actual content of the above theorem is for non-zero classes that do  have smooth representatives. In this case, the above theorem says that there exists an open set $\Om_{[\Si]}$ in the space of smooth Riemannian metrics, such that every area-minimizing mod $2$ current in $[\Si]$ with respect to an ambient metric $g\in \Om_{[\Si]}$  has a non-empty singular set of Hausdorff dimension at least $(d-3)$, provided $d\ge 3,c\ge 2,$ thus showing Result \ref{mts}.
	
	As we have discussed above, among the three families $c=1,$ $d=1,d=2,$ that are not dealt with in our Theorem \ref{thmm},  area-minimizing submanifolds with $d=1,2$ have been proven to be smooth in generic metrics (\cite{HFts,BWgt}). The case of $c=1,$ i.e., hypersurfaces, is widely believed to be generically smooth as well, with the best result obtained by Chodosh-Mantoulidis-Schulze-Wang \cite{ZW11}. Thus, Theorem \ref{thmm} is indeed sharp if we consider any fixed $d$ or fixed $c$.

	Interestingly, by Thom's classical work \cite{RT}, persistent singular sets in topological representatives start appearing from $d=7,$ while in our Theorems \ref{mts} and \ref{thmm}, persistent singular sets in area-minimizing representatives  start appearing from $d=3.$ See  \cite{GCmod2,ABCD} for more on regularity of topological representatives and \cite{CLW} for progress on generic smoothness of stationary varifolds.

It is an open problem since the 1980s \cite{GL} whether the cone over the  Veronese minimal embedding of $\rpt$ into the unit sphere $S^4_1(0)$ of $\R^5$ is area-minimizing or not. In the positive directions, Gary Lawlor and Timothy Murdoch  \cite{LMvc,TMtc} have proven that $C$ is area-minimizing among homologous competitors satisfying a generalized orientability assumption. In the negative direction, Gary Lawlor \cite{GL} showed that the maximum of the curvature of $C$ is too large to use Lawlor's vanishing retractions, which is asymptotically sharp in many cases \cite{ZLns}. 

As a crucial step of our proof we prove that
\begin{thm}
\label{thmc}
The cone over the Veronese minimal $\rpt$ is an area-minimizing mod $2$ current.
\end{thm}
In the proof of Theorem \ref{thmc}, we devise a new general method to prove area-minimizing, inspired by Frank Morgan's work \cite{FMcalv}. We call this new method moderations, which is of independent interest and might help solve several other longstanding problem about area-minimizing cones.

We also devise a new way of gluing area-minimizing currents and prove a powerful realization theorem inspired by Zhang's work \cite{YZa,YZj} on gluing calibrations.
\begin{defn}
	We say a $d$-dimensional area-minimizing mod $2$ cone $K^d$ in $\R^{d+c}$ is a regular cylindrical cone if $K$ is not a linear subspace and
	\begin{itemize}
		\item  either $K$ has smooth link,
		\item or $K=\R^{s}\times C^{d-s}\s \R^{s}\times \R^{d-s+c},
		$ where $1\le s\le d-2$ is an integer and 
		$C$ is a mod $2$ area-minimizing cone with smooth link in $\R^{d-s+c}.$ 
	\end{itemize}
In the first case we say that $K$ has spine dimension $0$  and in the second case we say that $K$ has spine dimension $s.$ \end{defn}
Here the link of a cone is its intersection with the unit sphere in its ambient space. Note that by regularity of mod $2$ area-minimizing currents \cite{HFts}, $s$ cannot be larger than $(d-2).$
\begin{thm}\label{thmzhang}
	If $K$ is a $d$-dimensional regular cylindrical mod $2$ area-minimizing cone in $\R^{d+c}$ with spine dimension $s$, and $[\Si]\in H_d(M,\Z/2\Z)$ can be represented by a smoothly embedded manifold, there exists a smooth metric $g$ on $M,$ so that 
	\begin{itemize}
		\item the unique mod $2$ area-minimizing current $T$ in $[\Si]$ has a singular set diffeomorphic to the standard $s$-dimensional sphere $S^s,$
		\item in a neighborhood of the singular set of $T$ equals $C$ truncated times $S^2,$ thus at each point of $\sing T$, the tangent cone to $T$ is $K.$
	\end{itemize}
\end{thm}
To the author's knowledge, Theorem \ref{thmzhang} gives the first realization theorem of area-minimizing tangent cones without imposing calibrations.
	\subsection{Plan of our proof}\label{planpf}
	We now outline the plan of our proof for Theorems \ref{thmm} and \ref{mts}. 
	
	Theorem \ref{mts} is a direct corollary of Theorem \ref{thmm}, by adding to the conclusion homology classes with no smooth representatives at all due to topological reasons. 
	
	From now on, we will focus on proving Theorem \ref{thmm}.
	
	Recall that our goal is to find open subsets of the space of smooth Riemannian metrics, in which the singular sets of area-minimizing submanifolds are of dimension at least  $(d-3)$ We adopt the following definition about persistent singular sets.
	\begin{defn}\label{psing}
		We say a subset $\kk$ of the singular sets of an area-minimizing submanifold $T$ in a Riemannian metric $g$ is a persistent singular set, if there exists an open subset $\Om_T$ containing $g$ of the space of smooth Riemannian metrics, so that for all $h\in\Om_T$
		\begin{align*}
			\inf\dimh\sing([T],h)\ge \dimh \kk.
		\end{align*}
	\end{defn}
	Then our goal is to find an area-minimizing submanifold with persistent singular sets of dimension $(d-3)$ in a smooth Riemannian metric.
	
	Using connected sums of submanifolds, we can assume that $[\Si]$ can be represented by a connected smooth submanifold $N.$ 
	
	The topological representative $N$ have no singular sets to start with. Thus, to achieve our goal of having persistent singular sets, we must {add} {singular} {sets} {manually} to alter the topological representative $N$ while achieving the following four features:
	\begin{enumerate}
		\item  The added singular sets give an altered topological representative that stays in the class $[\Si]$.\label{ft0}
		\item The added singular sets should have dimension $(d-3).$\label{ft1}
		\item We can find a smooth Riemannian metric $h$, in which the altered topological representative is {area-minimizing}.\label{ft2}	
		\item The added singular sets are {persistent singular sets} for the altered representative area-minimizing in $h$.\label{ft3}
	\end{enumerate}
	Features (\ref{ft0}) and (\ref{ft1}) are only about the altered topological representative, while in Features (\ref{ft2}) and (\ref{ft3}) the altered topological representative is area-minimizing in our newly found metric $g.$
	
Features (\ref{ft0}), (\ref{ft1}) and (\ref{ft2}) together gives Theorem \ref{thmzhang}. Feature (\ref{ft0}) uses a classical construction due to Smale \cite{NS}, from which Feature (\ref{ft1}) follows naturally. Feature (\ref{ft2}) uses a metric gluing argument inspired by Zhang's work of gluing calibrations \cite{YZa,YZj,YZt}. 

Feature (\ref{ft3}) is achieved as follows.
With Theorem \ref{thmc} in hand, let $C$ denote the cone over the Veronese minimal embedding of $\rpt$, i.e., the union
of all rays starting from the origin and passing through a point in the Veronese embedding of $\rpt$.  the product $C\times S^{d-3}$ is area-minimizing in the product $\R^5\times S^{d-3}$ with product metric, where $S^{d-3}$ is the standard $(d-3)$-dimensional sphere. Then we have the following.
	\begin{fact}\label{fctcb}
		The singular set $\{0\}\times S^{d-3}$ of the $d$-dimensional area-minimizing submanifold $C\times S^{d-3}$ is a persistent singular set.
	\end{fact}
	To see the ideas behind the above fact, let us think of the simplest case where $d=3,$ i.e., we have two disjoint copies of $\R^5,$ each copy with an area-minimizing $C$ inside. Note that topologically $\rpt$ as a manifold cannot be the boundary of any $3$-dimensional manifold. Thus any reasonable perturbation of the Veronese embedding of $\rpt$ will not bound a smooth area-minimizing submanifold. The above Fact \ref{fctcb} is heuristically due to this non-bounding property of $\rpt$. The analogues of the above Fact \ref{fctcb} has also been used in similar ways in \cite{HP}.
	
	Note that the singular set in Fact \ref{fctcb} has dimension $(d-3)$, which explains our lower bound on the singular sets of area-minimizing submanifolds in Theorem \ref{thmm}.

	With Fact \ref{fctcb} in hand, our plan is to manually introduce $C\times S^{d-3}$ into our topological representative of $[\Si]$, while achieving Features \ref{ft0}, \ref{ft1} and \ref{ft2}. Then Feature \ref{ft3} is a direct consequence of Fact \ref{fctcb}, which finishes the proof.
	\subsection{Overview of the paper}
	Our paper will be structured as follows. 
	
	In Section \ref{bsdefn}, we will give some basic definitions.
	
	In Section \ref{secmod}, we will propose a new method of proving area-minimization, which we call moderations.
	
	In Section \ref{secthmc}, we prove Theorem \ref{thmc} using the notion of moderations.
	
	In Section \ref{preprep}, we will prepare a topological representative of $[\Si]$. In Section \ref{modsing}, we will alter the topological representative of $[\Si]$ to add our desired singular sets, while establishing Features \ref{ft0} and \ref{ft1} of our plan. In Section \ref{tbssc}, we establish some technical topological lemmas that are needed for Section \ref{zhang}.
	
	In Section \ref{zhang}, we will find a smooth Riemannian metric in which the altered representative produced in Section \ref{modsing} is area-minimizing, thus achieving Feature \ref{ft2} and finish the proof of Theorem \ref{thmzhang}.
		
	In Section \ref{seccpt}, we will prove Fact \ref{fctcb} in the case of our altered topological representative and warp up the proof.
	
	In Section \ref{secdis}, we will give some discussions about our results.
	\section*{Acknowledgements}
The author thanks Professor Camillo De Lellis for bringing White’s question to his attention and wishes him a happy 50th birthday. Sincere thanks go to Professor Robert Bryant, who told the author about the Veronese cone and its persistent singularity when the author was an undergraduate student at Duke, which ultimately leads to this manuscript.  Deep thanks go to Professor Frank Morgan, whose works \cite{FMcalv,FMeu} directly inspired the notion of moderation introduced in Definition \ref{defnmod}.  Special thanks go to Professor Junyi Zhang for a heuristic observation, from a physicist's perspective, concerning the Bloch coherent states of representations of $\sos$ and $\operatorname{SU}(2)$, which leads to the proof of the moderation inequality in Lemma \ref{lemcsc}.  The author also thanks Professors Yongsheng Zhang and Zhihan Wang, whose work, insights, and helpful conversations substantially influenced the development of this manuscript. Finally, the author warmly thanks Professors William Allard, Hubert Bray, Simon Brendle, whose steadfast encouragement and support were indispensable to the completion of this manuscript.
	\section{Basic definitions}\label{bsdefn}
	We will give some basic definitions in this section. The contents are more or less standard and the experienced reader can skip this section.
	\subsection{Manifolds}\label{bdmn}
	As we deal with mod $2$ homology in our manuscript, unless otherwise stated, every manifold and submanifold we mention in this manuscript is assumed to be unoriented (Definition \ref{defnno}).  
	
	In general,  the symbol $M$ will be reserved for a $(d+c)$-dimensional closed ambient manifold and the symbol $[\Si]$ will be reserved for a $d$-dimensional homology class. When we do not mention the ambient manifold explicitly, it is understood that the ambient manifold is $M.$ 
	
	We will often speak of smooth neighborhoods or smooth open sets, by which we mean an open set with smooth boundary.
	
	We will also make frequent use of transversality \cite{MH}. 
	\subsection{Riemannian geometry}
	When a submanifold manifold $N$ is adopted with the induced Riemannian metric, we will use $\operatorname{dist}_{(M,g)}(p,q)$ to denote the Riemannian distance between points $(p,q)$ in Riemannian manifold $(M,g)$. And $\operatorname{dist}_{(M,g)}(p,N)$ will be used to denote the Riemannian distance of $p$ to a closed subset $N$ on $(M,g).$ We will drop the ambient manifold $M$ when it is clear from context.
	
	We will reserve the symbol $\no{\cdot}$ to denote the Riemannian length of vectors, forms, etc. Whenever we use the symbol $|\cdot|,$ it is understood that we are either taking absolute value or taking the square root of the sum of squares of components in coordinates, thus different from $\no{\cdot}$ on general manifolds.
	
	The symbol $\T$ will be reserved to denote tangent bundles and the symbol $\T^\perp$ will be reserved to denote normal bundles.%	\subsection{Unique continuation of minimal submanifolds}	In this manuscript, we will frequently use the following classical result:	\begin{lem}\cite[Lemma 2.4.1]{ZLns}\label{suct}(Strong unique continuation) If two properly embedded $d$-dimensional minimal submanifolds $F,G$ in a $(d+c)$-dimensional Riemannian manifold are tangent to each other of infinite order at a point $p\in F\cap G$, then in a neighborhood of $p$ they coincide.	\end{lem}
	\subsection{Mod $2$ currents}
	The right category to discuss area-minimizing representatives of mod $2$ homology classes is the category of mod $2$ currents. We will use Federer's definitive monograph \cite{HF}, Simon's classical lecture notes \cite{LS} and White's foundational works \cite{BWdt,BWrc} as basic references. 
	
	For our purposes, the reader can just regard mod $2$ currents as integer coefficient chains of simplicial complexes in algebraic topology. Note that mod $2$ currents are allowed to have finite area boundaries.
	
	Techincally speaking, non-compact cones are not mod $2$ currents as they have infinite area. However, we follow the literature convention and call them mod $2$ currents, as we will only deal with compact truncations of cones.

For a mod $2$ current $T$, we will abuse notations and use $T$ to denote both the current and its underlying rectifiable set. Restrictions of $T$ onto smooth open sets $U$ will be denoted by $T\cap U$ or $T|_U$ depending on which notation is simpler in the context. Pushforward of $T$ under smooth maps $f$ will be denoted by $f(T).$
		\subsection{Three facts about mod $2$ currents}
We will use the following facts several times
\begin{fact}\label{const}\cite[6.1 Theorem]{MSfc}
	Let $N$ be a $d$-dimensional open submanifold in a complete $(d+c)$-dimensional manifold $M$. Then an $d$-dimensional mod $2$ current $T$ with $\pd T\cap N=\es$ restricted to $N$ equals $aN$ for some $a\in \Z/2\Z.$
\end{fact}
\begin{fact}\label{fctconc}\cite[Theorem 3.1]{BWdt}
	A $d$-dimensional mod $2$ flat chain supported on a $d$-dimensional Hausdorff measure $0$ set is the $0$ current.
\end{fact}
\begin{fact}\label{fctmap}\cite[Lemma 3.4]{LFZg}
	Assume that $f:M\to M$ is a $C^1$ map. Then for any mod $2$ current $T$, the mass of $f(T)$ is less than or equal to the mapping area of $T$ under $f,$
	\begin{align*}
		\ms(f(T))\le\int_T \jac(f|_{\T_pT})\dhm^d(p).
	\end{align*}
\end{fact}
\subsection{Definition of mod $2$ area-minimizing}
	We say an mod $2$ current $T$ is area-minimizing, if $T$ has the least area among all mod $2$ currents homologous to $T$:
\begin{defn}\label{defnam}
	A $d$-dimensional mod $2$ current $T$ is area-minimizing if 
	\begin{align*}
		\ms(T)\le \ms(T+\pd V),
	\end{align*}for all  $(d+1)$-dimensional mod $2$ currents $V$. 
\end{defn}
Here $\ms$ is the mass of the mod $2$ current, which in our context, is just the area of the underlying set. For instance, $\ms( S^1)=2\pi$ for the unit circle $S^1$ in $\R^2.$

It is also customary in geometric measure theory to speak of non-compact area-minimizing mod $2$ currents $T$. By this, we mean that the restriction of $T$ to compact subsets are mod $2$ area-minimizing.	\subsection{Definition of singular sets}
	\begin{defn}\label{defnsm}
		We say a mod $2$ current $T$ is smooth at a point $p$ in the support of $T$ if there exists an open set $U$ containing $p$ on $M$, such that  $T$ restricted to $U$ equals a smooth submanifold $N.$ The definitions of regular sets and singular sets of $T$ are as follows:
		\begin{itemize}
			\item The singular set of $T,$ $\sing T,$ is defined as set of points in the support of $T$ where $T$ is not smooth.		
			\item 
			The regular set of $T,$ $\operatorname{Reg}T$, is defined as the set of the points in the support of $T$ where $T$ is smooth. 
		\end{itemize}
	\end{defn} 
	Here support means the underlying set of a mod $2$ current. From now on, we will also use the symbol $\supp T$ to mean the support of $T.$%, and unfortunately is the standard term in geometric measure theory. The reader should not confuse it with the notion of supports in simplicial complexes.
	%In other words a current $T$ is smooth at a point $p$ in the underlying set of $T$, if the current $T$ restricted to a neighborhood of $p$ equals $kN$ for an integer $k\in N$ and a smooth submanifold $N$. The singular set of $T$ is defined as the complement of the smooth points of $T$.In our manuscript, all mod $2$ currents we construct will be chains of simplicial complexes with explicit descriptions. It will always be clear from the explicit descriptions alone what are the singular set and the regular set of the currents we  construct, so we will not spend time on formally proving what subsets are the singular sets.
%	\subsection{Multiplicity $1$ mod $2$ currents}Many powerful theorems from geometric measure theory deal with a special class of mod $2$ currents called    currents. We give a formal definition as follows.
%	\begin{defn}\label{defnmo}
%		We say an mod $2$ current $T$ is of multiplicity $1,$ if $T$ restricted to a neighborhood of each smooth point $p$ equals a smooth submanifold $N$.
%	\end{defn}
%	For example, a figure $8$ in $\R^2$ is a    mod $2$ current.
	\subsection{Connected sum of    mod $2$ currents}
	Many times in this manuscript, we need to use  a connected sum to connect a representative of $[\Si]$ with multiple components. We give a formal construction as follows.
\begin{assump}\label{assumpst}
	We assume that
\begin{itemize}
\item $T$ and $V$ are two $d$-dimensional mod $2$ currents of dimension $d$ and codimension ${c}$ {both at least} ${2}$.
\item	$T$ and $V$ have disjoint support. \end{itemize}	
\end{assump}
	Our goal is to construct a connected sum $T\# V$ under the above assumption \begin{lem}\label{lemcs}
	Under Assumption \ref{assumpst}, there is a mod $2$ current $T\# V$, called the connected sum of $T$ and $V,$ so that 
		\begin{itemize}
			\item $\sing (T\# V)=\sing (T+V)$.
			\item Restricted to a neighborhood of $\sing (T\# V),$ the mod $2$ current $T\#V$ equals $T+V$.
			\item If $\operatorname{Reg}T$ and $\operatorname{Reg}V$ are both connected, then $\operatorname{Reg}T\# V$ is also connected.
			\item $[T\# V]=[T]+[V]$ as homology classes.
		\end{itemize}
	\end{lem}
	For instance, the connected sum of two disjoint embedded tori in $\R^3$ is a genus $2$ surface. Here the symbol $+$ in $T+V$ means the sum as mod $2$ currents. The reader can just assume that $+$ means the sum of chains of simplicial complexes.
	\begin{proof}
\cite[Lemma 2.8.2]{ZLns} states the corresponding result for multiplicity $1$ integral currents and the same proof applies. For the reader's convenience, we reproduce the entire argument.		In case $T$ and $V$ are disjoint submanifolds, the construction is classical and can be found in \cite{CL}.

In our case, $T$ and $V$ are general mod $2$ currents that may have large singular sets. However, the idea is similar. We need to remove a $d$-dimensional ball from $T$ and $V,$ respectively, then use a tube diffeomorphic to $S^{d-1}\times[0,1]$ to connect $T$ to $V.$ 	The precise construction is as follows.

Take a smooth point $p$ of $T$, and a smooth point $q$ of $V,$ with $p\not=q.$ First, let us show that there is an open set $U$ containing both $p$ and $q$, such that we have a coordinate $$(x_1,\dots,x_d,y_1,\dots,y_c),$$
on $U$, with
\begin{align*}
	p=(\overbrace{0,\dots,0}^{d},\overbrace{0,0,\dots,0}^{c}),q=(\overbrace{0,\dots,0}^{d},1,\overbrace{0,\dots,0}^{c-1})
\end{align*}Moreover,  the mod $2$ current $T$ restricted to $U$ is parameterized by
\begin{align*}
	(x_1,\dots,x_d,0,0,\dots,0),	\end{align*}with each $x_j\in\R,$ and the mod $2$ current $V$ restricted to $U$ is parameterized by
\begin{align}\label{xdminus}
	(x_1,x_2,\dots,x_d,1,0,\dots,0),
\end{align}with each $x_j\in\R.$

To construct the desired coordinate chart $U$ as above, we need to use transversality. Let $\ga$ be a smooth curve from $p$ to $q$. By transversality \cite{MH} and $c\ge 2$, i.e., $\dim T+\dim\ga=\dim V+\dim\ga=d+1<d+c=\dim M,$ a generic smooth perturbation of $\ga$ from $p$ to $q$ satisfies
\begin{itemize}
	\item $\ga$ intersect $T$ and $V$ only at end points $p$ and $q$, 
	\item $\ga$ is not tangent to $\reg T$ and $\reg V.$ 
\end{itemize}
These two implies that in a tubular neighborhood $B(\ga)$ of $\ga$ we have $B(\ga)\cap T$ and $B(\ga)\cap V$ are two embedded standard $d$-dimensional balls. As normal bundles over curve segments are trivial, e.g., via parallel transport, we can adopt a Fermi coordinate chart $U(\ga):(x_1,\cdots,x_d,y_1,\cdots,y_c)$ around $\ga,$ in which $\ga$ is the line segment from $	p=(\overbrace{0,\dots,0}^{d},\overbrace{0,0,\dots,0}^{c})$ to $q=(\overbrace{0,\dots,0}^{d},1,\overbrace{0,\dots,0}^{c-1})$ and $T,V$ restricted to $U(\ga)$ is are two smoothly embedded $d$-dimensional balls passing through $p,q$ and intersecting $\ga$ transversely. Then the desired coordinate follows by adapting the coordinate $U(\ga)$ to $T,V$ at $p,q$ respectively.

Now remove the $d$-dimensional ball centered at $p$ $$B_1^d(0)\times\{0\}\times\{0\}^{c-1},$$ from $T$ and remove the $d$-dimensional ball centered at $q$ $$B_1^d(0)\times\{1\}\times\{0\}^{c-1},$$ from $V$, and  glue the tube
\begin{align*}
	S_1^{d-1}(0)\times [0,1]\times \{0\}^{c-1},
\end{align*}smoothly onto $T+V$ with the two $d$-dimensional balls removed,
we obtain $T\# V.$ Here by gluing smoothly we mean smoothing the corner introduced by the boundary of the tube. Since our constructions only alter $\reg T$ and $\reg V,$ we deduce that the first two bullets in Lemma \ref{lemcs} hold. Since $T$ and $V$ are of dimension at least $2$, we deduce the third bullet in Lemma \ref{lemcs}. For the last bullet in Lemma \ref{lemcs}, since
\begin{align*}
	&	S_1^{d-1}(0)\times [0,1]\times \{0\}^{c-1}\\=&\pd\bigg(B_1^d(0)\times [0,1]\times\{0\}^{c-1}\bigg)	+B_1^d(0)\times\{0\}\times\{0\}^{c-1}+B_1^d(0)\times\{1\}\times\{0\}^{c-1},
\end{align*}as mod $2$ currents, we deduce that $[T\# V]=[T]+[V]$.
	\end{proof}
\section{Moderation}\label{secmod}
In this section, we propose a criterion for proving area-minimization inspired by \cite{FMcalv,FMeu}, which to the author's knowledge has not been known before. Inspired by calibrations, we will call this criterion moderation.
\begin{defn}\label{defnmod}
Let $T$ be a $d$-dimensional mod $2$ current compactly supported in $\R^{d+c}$. 

 We say $T$ is a moderation if $T$ satisfies the following conditions:
\begin{enumerate}
	\item For Hausdorff $d$-dimensional almost every tangent space $E$ to $T$, we have 
	\begin{align}\label{eqmoddef}
		\int_{T}|\ri{E,\T_p T}|\dhm^d(p)=\sup_{W^d\textnormal{ a linear subspace of }\R^{d+c}}		\int_{T}|\ri{W,\T_p T}|\dhm^d(p)
	\end{align}
%	\item For Hausdorff $d$-dimensional almost every tangent space $E$ to $T$, the pushforward current	\begin{align*}		\pi_E(T)	\end{align*} is a mod $2$ area-minimizing current in $E.$
	\item  For Hausdorff $d$-dimensional almost every tangent space $E$ to $T$, we have\begin{align}\label{eqmodms}
		\ms(\pi_E(T))=	\int_{T}|\ri{E,\T_p T}|\dhm^d(p)>0.
	\end{align}
	\end{enumerate}
By (\ref{eqmoddef}), the left hand side of (\ref{eqmodms}) is independent of $E$, which we will call the moderated mass of $T,$ denoted by
\begin{align*}
	\mms(T)=\ms(\pi_E(T)).
\end{align*}
\end{defn}
Here $W^d$ denotes a $d$-dimensional linear subspace of $\R^{d+C}$ and $\pi_W$ means the orthogonal projection to the subspace $W$. The inner product $|\ri{W,\T_pT}|$ means the absolute value of the inner product of simple $d$-vectors representing $W$ and $\T_pT$ induced by the inner product $\ri{\cdot,\cdot}$ on $\R^{d+c}$. 

The definition of moderation immediately gives the following theorem
\begin{thm}\label{thmmod}
	A moderation $T$ is an area-minimizing mod $2$ current.
\end{thm}
\begin{proof}
Use $W$ to denote an arbitrary $d$-dimensional linear subspace of $\R^{d+c}.$ Abusing notations, the differential of the map $\pi_W$ equals $\pi_W.$ By definition of inner product on the exterior algebra, we have
	\begin{align*}
		\jac(\pi_W|_V)=|\ri{V,W}|.
	\end{align*} 
	Let $S$ be a mod $2$ current with
	\begin{align*}
		S-T=\pd Q,
	\end{align*}with $Q$ a $(d+1)$-dimensional mod $2$ current. 
	
We have
\begin{align*}
	\pi_W(S)-\pi_W(T)=\pd(\pi_W(Q)).
\end{align*}	However, $\pi_W$ is a $(d+1)$-dimensional mod $2$ current supported on a $d$-dimensional subspace. By Fact \ref{fctconc}, we deduce that $\pi_W(Q)=0.$ Thus, we have
\begin{align*}
	\pi_W(S)=\pi_W(T).
\end{align*}
Consequently, for Hausdorff $d$-dimensional every tangent plane $E=\T_pT$ to $T$, we have
\begin{align}\label{ineqq1}
		\mms(T)=\ms(\pi_E(T))=\ms(\pi_E(S))
\end{align}
	On the other hand, by Fact \ref{fctmap}, we have
	\begin{align}\label{ineqq2}
	\ms(\pi_E(S))\le\int_{S}\jac(\pi_E|_{\T_q S})\dhm^d(q)=\int_S|\ri{E,\T_qS}|.
	\end{align}	Combining (\ref{ineqq1}) and (\ref{ineqq2}), for Hausdorff $d$-dimensional every tangent plane $\T_pT$ to $T$, we have
	\begin{align}\label{eqmod1}
	\mms(T)\le\int_{S}|\ri{\T_pT,\T_qS}|\dhm^d(q).
	\end{align}Note that the right-hand side of (\ref{eqmod1}) is less than or equal to the mass of $S$, so is finite and thus $|\ri{\T_pS,\T_qS}|$ is integrable on $S$. Integrate both sides by mass measure of $T$ and use Fubini's theorem  we have
\begin{align}\label{eqint}
	&\mms(T)\ms(T)=\int_{T}\mms(T)\dhm^d(p)\\\le&\int_{T}\int_{S}|\ri{\T_pT,\T_qS}|\dhm^d(q)\dhm^d(p)\\=&\int_{S}\int_{T}|\ri{\T_qS,\T_pT}|\dhm^d(p)\dhm^d(q).
	\end{align}
By Definition \ref{defnmod}, we have for Hausdorff $d$-dimensional almost every $q\in S,$
\begin{align}\label{eqint2}
	\int_{T}|\ri{\T_qS,\T_pT}|\dhm^d(p)\le\mms(T).
\end{align}
Thus, from (\ref{eqint}) to (\ref{eqint2}), we have
\begin{align*}
	\mms(T)\ms(T)\le\mms(T)\ms(S).
\end{align*}	
By the second bullet in Definition \ref{defnmod} we can divide both sides by $\mms(T)$ and we are done.
\end{proof}
\section{Proof of Theorem \ref{thmc}}\label{secthmc}
Our goal in this section is to prove Theorem \ref{thmc}. The proof relies on showing that the Veronese cone over $\rpt$ is a moderation in the sense of Definition \ref{defnmod} and then Theorem \ref{thmc} follows from Theorem \ref{thmmod}. 

Roughly speaking, we relate the integrals in Definition \ref{defnmod} to the sharp inequalities in representation theory \cite{LPcs,RFsi,KNOT} and verify that 	 the Veronese cone over $\rpt$ is a moderation.

We will first define the Veronese embedding of $\rpt$ via $\sos$-actions on $\R^5$ and collect several well-known facts. These facts are standard and usually stated without proof in the literature, e.g., \cite{TMtc,CDKms}. Unfortunately different authors use different conventions of normalizations, so for the reader's convenience, we provide self-contained proof of these facts.

Then we delve deeper and gather several lemmas required for applying the representation theory results  \cite{LPcs,RFsi,KNOT}.

Finally, we prove that the Veronese cone over $\rpt$ is a moderation and finish by Theorem \ref{thmmod}.
\subsection{Definition of the Veronese cone and standard facts}Let us first define the ambient space $\R^5.$
\begin{defn}
	\label{defnvv}
	Let $\VV$ be the vector space of symmetric traceless $3\times 3$ real matrices, i.e.,
	\begin{align*}
		\VV=\{\vv|\vv\textnormal{ is a }3\times 3\textnormal{ real matrix},\vv^T=\vv,\tr\vv=0\}.
	\end{align*}Endow $\VV$ with the bilinear form $\ka$:
	\begin{align*}
		\ka(\vv,\mathbf{w})=\frac{1}{2}\tr(\mathbf{v}\mathbf{w})
	\end{align*}
\end{defn}
As all elements of $\VV$ are symmetric, $\ka$ equals half of the standard Frobenius inner product of matrices. We have
\begin{fact}
	$\ka$ is a positive definite inner product on $\VV$.
\end{fact}
The normalizing factor $\frac{1}{2}$ here are introduced to simplify the calculations. Actually $i\VV$ is the orthogonal complement of $\sot$ in $\sut$, but we will not use that fact here. From now on, we will reserve boldface letters to denote vectors in $\VV.$

	Let us first provide an orthonormal basis of $\VV$. 
\begin{defn}Define $\vv_1,\dots,\vv_5$ as follows
	\begin{align*}
		&\vv_1=\frac{1}{\sqrt{3}}\begin{bmatrix}
			1&0&0\\
			0&1&0\\
			0&0&-2\\
		\end{bmatrix},
		\vv_2=\begin{bmatrix}
			0&0&1\\
			0&0&0\\
			1&0&0\\
		\end{bmatrix},
		\vv_3=\begin{bmatrix}
			0&0&0\\
			0&0&1\\
			0&1&0\\
		\end{bmatrix},	\vv_4=\begin{bmatrix}
		0&1&0\\
		1&0&0\\
		0&0&0\\
		\end{bmatrix}
		,	\vv_5=\begin{bmatrix}
			1&0&0\\
			0&-1&0\\
			0&0&0\\\end{bmatrix}.
	\end{align*}
\end{defn}
Direct calculations can show that $\vv_1,\dots,\vv_5$ form an orthonormal basis of $\VV$ with respect ot $\ka$. We will also provide Mathematica verifications at the end of the paper. This gives
\begin{fact}
	$\VV$ is a $5$-dimensional real vector space. 
\end{fact}
Now we are ready to introduce the Veronese embedding of $\rpt$.
\begin{fact}\label{fctorbv}
Let $\sos$ acts on the space of $3\times 3$ real matrices via conjugation, 
	\begin{enumerate}
		\item The linear subspace $\VV$ is invariant under the conjugation action of $\sos.$
		\item $\sos$ conjugations preserve the inner product $\ka.$
		\item The stablizers of $\vv_1$ under the conjugation action of $\sos$ is the group of block diagonal matrices $\operatorname{S}(\operatorname{O}(2)\times \operatorname{O}(1))$.
		\item The orbit of $\vv_1$ under the conjugation action of $\sos$ is an embedding of $\rpt$ into the unit sphere of $\VV$ in metric $\ka.$
		\item The action of $\sos$ on $\VV$ is real irreducible.
	\end{enumerate}	
\end{fact}\begin{proof}
Conjugation by $\sos$ elements preserves trace and transpose symmetry, which gives the first bullet.

For the second bullet, note that for $\vv\in\VV$, $g\in\sos,$ we have
\begin{align*}
	\ka(g\vv g\m,g\vv g\m)=\frac{1}{2}\tr(g\vv g g\m\vv g\m)=\frac{1}{2}\tr(g\m g\vv\vv)=\ka(\vv,\vv).
\end{align*}
For the third and the fourth bullet, recall the manifold orbit-stabilizer theorem \cite[Proposition 4.2]{AAl} that the orbit of a compact Lie group $G$ action at a point $p$ with stabilizer subgroup $H,$  equals the coset manifold $G/H.$ 	To determine the orbit of $\vv_1$ under $\add$-action, let us first calculate the stabilizer subgroup of $\vv_1$ under the conjugation action of $\sos$.
	
	Let $\zeta$ be an arbitrary $3\times 3$ matrix. We can write $\zeta$ as a block matrix with $2\times 2,2\times 1,1\times 2,1\times 1$ blocks $E,F,G,H,$  i.e.,
	\begin{align*}
		\zeta=\begin{bmatrix}
			E&F\\G&H
		\end{bmatrix}.
	\end{align*}
	If $\zeta \vv_1\zeta\m=\vv_1,$ then we have
	\begin{align*}
		0=&[\zeta,\vv_1]=\begin{bmatrix}
			E&F\\G&H
		\end{bmatrix}\begin{bmatrix}
			I_2&0\\0&-2
		\end{bmatrix}-\begin{bmatrix}
			I_2&0\\0&-2
		\end{bmatrix}\begin{bmatrix}
			E&F\\G&H
		\end{bmatrix}\\
		=&\begin{bmatrix}
			0&-3F\\3G&0
		\end{bmatrix}.
	\end{align*}
	This implies that $F=0,G=0,$ i.e., $\zeta$ is block diagonal with $2\times 2,1\times 1$ blocks. 
	
	Thus, any stabilizer $\zeta$ of $\vv_1$ under the conjugation action of $\sos$ consists of special orthogonal matrices that are $2\times 2,1\times 1$ block diagonal of the form
	\begin{align}\label{eqcv1}
		\zeta=	\begin{bmatrix}
			E&0\\0&H
		\end{bmatrix}.
	\end{align}
	Since $\zeta\zeta^\dagger=\id,$ we deduce that $E\in \operatorname{O}(2)$ and $H\in \operatorname{O}(1).$ Since $\det \zeta=1,$ we deduce that the stabilizer of $\vv_1$ is a subgroup $$H=\operatorname{S}(\operatorname{O}(2)\times \operatorname{O}(1)).$$
	
	However, it is well known that $\rpt$ as a homogenous space is isometric to
	\begin{align*}
		\sos/H.
	\end{align*}For instance, consider the standard representation of $\sos$ on $\R^3,$ which induces a transitive action of $\sos$ on $\rpt.$ The stabilizer of point $[0:0:1]$ in homogeneous coordinate is precisely $H.$
	
	Thus the orbit of $\vv_1$ under the conjugation action of $\sos$ is indeed an embedded $\rpt$ in the unit sphere of $V$ by the manifold orbit-stabilizer theorem and we are done.
	
The last bullet is \cite[Proposition 3.1.1]{TMtct}
\end{proof}
The orbit of $\vv_1$ under the conjugation action of $\sos$ is the so-called Veronese minimal embedding of $\rpt$ into the unit sphere of $\R^5.$
%We will also call $\vrpt$ the Veronese $\rpt$.
\begin{defn}\label{defnvc}
Use the symbol $\vrpt$ to denote the orbit of $\vv_1$ under the adjoint action of $\sos$. 
Define the cone $C(\vrpt)$ over $\vrpt$ to be the union of orbits $t\vv_1$ under the conjugation action of $\sos$ with $t\in[0,\infty)$, i.e., 
\begin{align*}
C(\vrpt)=\cup_{t\in[0,\infty)}t\vrpt.
\end{align*}
(IN THIS SECTION ONLY) To simplify notations, we will use $C$ to denote $C(\vrpt).$

Define the truncated cone $C_1$ to be the restriction of $C$ to the unit ball of $\VV,$ i.e.,
\begin{align*}
	C_1=\cup_{t\in[0,1]}t\vrpt.
\end{align*}
\end{defn}
In other words, $C$ is that $C$ is formed by rays starting from the origin of $\VV$ passing through a point in $\vrpt$ and $C_1$ is the truncation of the union of rays at length $1.$ 	

By definition, we have
\begin{fact}
	$C$ and $C_1$ are mod $2$ currents with
	\begin{align*}
		\pd C=&0,\\\pd C_1=&\vrpt.
	\end{align*}
\end{fact}
We want to remark here again that we are using the geometric measure theory literature convention of regarding cones as currents while only their restrictions to compact sets are.

We need another parametrization of $\vrpt.$
\begin{defn}
For any point $\tau\in\R^3$, define a map $F$ from $\R^3$ to $3\times 3$ matrices by 
\begin{align*}
	F(\tau)=-\tau\tau^T+\frac{|\tau|^2}{3}I.
\end{align*}
Here $I$ is the $3\times 3$ identity matrix and we regard $\tau$ as a column vector.
\end{defn}
\begin{fact}\label{fctmet}
	$F$ is a double covering from $S^2_{\sqrt[4]{3}}(0)\s \R^3$ to $\rpt.$ And we have
	\begin{align*}
		F(0,0,	\sqrt[4]{3})=&\vv_1,\\
		F(\tau)=&F(-\tau),\\
		F\du\kappa=&\sqrt{3}\de,
	\end{align*}
\end{fact}
Here $S^2_{\sqrt[4]{3}}(0)$ is the sphere of radius ${\sqrt[4]{3}}$ centered at $0$ in $\R^3,$  $\ka$ is the canonical inner product on $\VV$ (Definition \ref{defnvv}) and $\de$ is the standard flat metric on $\R^3$.
\begin{proof}
	Let us first calculate the length of $F$. We have
\begin{align*}
	\ka(F(\tau),F(\tau))=&\frac{1}{2}\tr\left(\tau\tau^T\tau\tau^T-\frac{2|\tau|^2}{3}\tau\tau^T+\frac{|\tau|^4}{9}I\right)\\
	=&\frac{1}{2}\tr\left(\tau^T\tau\tau^T\tau\right)-\frac{1}{2}\tr\left(\frac{2|\tau|^2}{3}\tau^T\tau\right)+\frac{1}{2}\tr\left(\frac{|\tau|^4}{9}I\right)\\
	=&\frac{1}{2}|\tau|^4-\frac{1}{3}|\tau|^4+\frac{1}{6}|\tau|^4\\
	=&\frac{1}{3}|\tau|^4.
	\end{align*}
Thus, $F$ maps 	$S^2_{\sqrt[4]{3}}(0)$ to the unit sphere in $V.$

Now let $\sos$ acts in the standard way on $\R^3$ via matrix multiplications. Then for $g\in \sos,$ we have
\begin{align*}
	F(g\tau)=-g\tau\tau^Tg\m+\frac{|\tau|^2}{3}gg\m=g F(\tau)g\m.
\end{align*}
In other words, $F$ is an $\sos$-equivariant map. Since $S^2_{\sqrt[4]{3}}(0)$ is an orbit of $\sos$ in $\R^3$, $F(S^2_{\sqrt[4]{3}}(0))$ must the orbit of $F(\tau)$ under the conjugation action of $\sos$ for any $\tau\in S^2_{\sqrt[4]{3}}(0)$. Take $\tau=(0,0,\sqrt[4]{3}).$
\begin{align*}
	F\left(\begin{bmatrix}
		0 \\
		0 \\
		\sqrt[4]{3}
	\end{bmatrix}\right)=&-\begin{bmatrix}
	0 \\
	0 \\
	\sqrt[4]{3}
	\end{bmatrix}\begin{bmatrix}
	0 \\
	0
\\	\sqrt[4]{3} 
	\end{bmatrix}^T+\frac{1}{\sqrt{3}}I\\
	=&-\begin{bmatrix}
		0 & 0 & 0 \\
		0 & 0 & 0 \\
		0 & 0 & \sqrt{3}
	\end{bmatrix}+\frac{1}{\sqrt{3}}I\\
	=&\frac{1}{\sqrt{3}}\begin{bmatrix}
		1 & 0 & 0 \\
		0 & 1& 0 \\
		0 & 0 &-2	\end{bmatrix}\\
		=&\vv_1.
\end{align*}
Thus, we have \begin{align*}
	F(S^2_{\sqrt[4]{3}}(0))=\vrpt.
\end{align*}

Next let us show that $F$ is a $2$ to $1$ map from $S^2_{\sqrt[4]{3}}(0)$ to $\vrpt$.

Suppose that $F(\tau)=F(\tau')$ for $\tau,\tau'\in S^2_{\sqrt[4]{3}}(0)$, then there exists $g,h\in \sos$ so that $$\tau'=g\tau,\tau=h\begin{bmatrix}
	0 \\
	0 \\
	\sqrt[4]{3}
\end{bmatrix}$$ Thus, we have
\begin{align*}
h\vv_1h\m=F\left(h\begin{bmatrix}
	0 \\
	0 \\
	\sqrt[4]{3}
\end{bmatrix}\right)=F(\tau)=F(\tau')=F(g\tau)=gF(\tau)g\m=gh\vv_1h\m g\m,
\end{align*}which gives
\begin{align*}
	\vv_1=h\m g h\m\vv_1 h\m g\m h=(h\m g h)\vv_1 (h\m g h)\m.
\end{align*}
Thus, $h\m gh$ is a stabilizer of $\vv_1,$ under the conjugation action of $\sos.$ By the third bullet of Fact \ref{fctorbv}, we have
\begin{align*}
	g=h \begin{bmatrix}
		E&0\\0&H
	\end{bmatrix}h\m,
\end{align*}for $E\in O(2),H=\pm 1$ and $H\det E=1$
This gives
\begin{align*}
	\tau'=g\tau=gh\begin{bmatrix}
		0 \\
		0 \\
		\sqrt[4]{3}
	\end{bmatrix}=h \begin{bmatrix}
	E&0\\0&H
	\end{bmatrix}h\m h\begin{bmatrix}
	0 \\
	0 \\
	\sqrt[4]{3}
	\end{bmatrix}=Hh\begin{bmatrix}
	0 \\
	0 \\
	\sqrt[4]{3}
	\end{bmatrix}=H\tau=\pm \tau.
\end{align*}
As $F(\tau)=F(-\tau),$ we deduce that $F$ is a $2$ to $1$ map.

	Now let us calculate the differential of $F$. We have
	\begin{align*}
		dF_\tau(\xi)=-\xi\tau^T-\tau\xi^T+\frac{2\de(\tau,\xi)}{3}I.
	\end{align*}
For $\tau\in S^2_{\sqrt[4]{3}}(0)$	and $\xi$ tangent to  $S^2_{\sqrt[4]{3}}(0)$ at $\tau$, we have $\tau\perp\xi$ and thus
\begin{align*}
	dF_\tau(\xi)=-\xi\tau^T-\tau\xi^T,
\end{align*}and
\begin{align*}
	F\du \ka (\xi,\xi)=&\frac{1}{2}\tr\left(dF_\tau(\xi)dF_\tau(\xi)\right)\\
	=&\frac{1}{2}\tr(\xi\tau^T\xi\tau^T+\xi\tau^T\tau\xi^T+\tau\xi^T\xi\tau^T+\tau\xi^T\tau\xi^T)\\
	=&\frac{1}{2}\tr\left(\xi\de(\tau,\xi)\tau^T+\xi\de(\tau,\tau)\xi^T+\tau\de(\xi,\xi)\tau^T+\tau\de(\xi,\tau)\xi^T\right)\\
	=&0+\frac{1}{2}\tr\left(\xi^T\xi\de(\tau,\tau)\right)+\frac{1}{2}\left(\tau^T\tau\de(\xi,\xi)\right)+0\\
	=&\de(\tau,\tau)\de(\xi,\xi)\\
	=&\sqrt{3}\de(\xi,\xi).
\end{align*}
This implies that $dF$ is bijection from the tangent space to $S^2_{\sqrt[4]{3}}(0)$ at $\tau$ to the tangent space to $\rpt$ at $F(\tau).$ Thus, $F$ is a $2$ to $1$ immersion, i.e., a double cover.
\end{proof}
The key connection between the parametrization $F$ and the linear space $\VV$ is as follows.
\begin{fact}\label{fctcon}
A vector $\vv\in\VV$ is perpendicular to $\T_{F(\tau)}C$ for some point $\tau\in S^2_{\sqrt[4]{3}}(0)$ if and only if $\vv\tau=0$.
\end{fact}
\begin{proof}
		For any $\vv\in \VV,$ $\tau\in S^2_{\sqrt[4]{3}}(0)$ and $\xi\in\T_\tau S^2_{\sqrt[4]{3}}(0)$, we have
	\begin{align*}
		\ka(F(\tau),\vv)=&\frac{1}{2}\tr\left(\vv\left(-\tau\tau^T-\frac{1}{3}I\right)\right)=-\frac{1}{2}\tr\left(\tau^T\vv\tau\right)+0=-\frac{1}{2}\de(\tau,\vv\tau),\\
		\ka(dF_\tau(\xi),\vv)=&\frac{1}{2}\tr\left(\vv\left(-\xi\tau^T-\tau\xi^T\right)\right)=-\frac{1}{2}\tr\left(\tau^T \vv\xi+\xi^T\vv\tau \right)=-\de(\xi,\vv\tau),
	\end{align*}
	where we have used $\vv^T=\vv$ at the last step. We are done by Fact \ref{fctmet}.
\end{proof}
Next let us calculate the normal space to $F(\tau)$ of $C.$ 
\begin{fact}
		For any point $\tau\in S^2_{\sqrt[4]{3}}(0)\setminus\{(0,0,\sqrt[4]{3}),(0,0,-\sqrt[4]{3})\}$, the following vectors $\{\xi_0(\tau),\xi_1(\tau),\xi_2(\tau)\}$ is a positively oriented orthonormal basis of $\T_\tau\R^3$:	\begin{align*}
			\xi_0(\tau)=&\frac{1}{\sqrt[4]{3}}\tau=(x,y,z),\\
			\xi_1(\tau)=&\left(\frac{-y}{\sqrt{1-z^2}},\frac{x}{\sqrt{1-z^2}},0\right),\\
			\xi_2(\tau)=&\left(\frac{-xz}{\sqrt{1-z^2}},\frac{-yz}{\sqrt{1-z^2}},\sqrt{1-z^2}\right).
		\end{align*}		
\end{fact}
		Here we use $(x,y,z)$ to denote the coordinate components of $\frac{1}{\sqrt[4]{3}}\tau$. In the following, we will drop the variable $\tau$  in order to simplify notations and we will reserve the symbol $(x,y,z)$ to denote points in the unit sphere of $\R^3$. Positively oriented means $\det(\xi_0,\xi_1,\xi_2)=1.$
\begin{proof}
We have
\begin{align*}
	&\de(\xi_1,\xi_1)=\frac{y^2+x^2}{1-z^2}=1,\de(\xi_2,\xi_2)=\frac{z^2(x^2+y^2)}{1-z^2}+1-z^2=1,\\
	&\de(\xi_1,\xi_2)=\frac{yxz-xyz}{\sqrt{1-z^2}}=0,\de(\xi_1,\xi_0)=\frac{-yx+xy}{\sqrt{1-z^2}}=0,\de(\xi_2,\xi_0)=\frac{-z(x^2+y^2)}{\sqrt{1-z^2}}+z\sqrt{1-z^2}=0.
\end{align*}
When $(x,y,z)=(1,0,0)$ we have  $\det(\xi_0,\xi_1,\xi_2)=1$. As $\xi_0,\xi_1,\xi_2$ are continuous orthonormal frames on the connected open set $S^2_1(0)\setminus\{(0,0,1),(0,0,-1)\}$ of the sphere $S^2_1(0)$, we deduce that $\det(\xi_0,\xi_1,\xi_2)=1$ everywhere on $S^2_1(0)\setminus\{(0,0,1),(0,0,-1)\}$ .   
\end{proof}
\begin{fact}\label{fctnmc}
	For any point $\tau=\sqrt[4]{3}\xi_0,$  the vectors following $\mathbf{n}_1,\mathbf{n}_2$ form an orthonormal basis normal space $\T^\perp_{F(\tau)}C,$
	\begin{align*}
		\mathbf{n}_1(\tau)=\xi_1(\tau)(\xi_1(\tau))^T-\xi_2(\tau)(\xi_2(\tau))^T,\\
		\mathbf{n}_2(\tau)=\xi_1(\tau)(\xi_2(\tau)^T)+\xi_2(\tau)(\xi_1(\tau)^T).
	\end{align*}
\end{fact}
\begin{proof}
Again we will omit the variable $\tau$ to simplify notations. 

We have
\begin{align*}
	\mathbf{n}_1\tau=&\xi_1\xi_1^T\tau-\xi_2\xi_2^T\tau=0,\\
	\mathbf{n}_2\tau=&\xi_1\xi_2^T\tau+\xi_2\xi_1^T\tau=0.
\end{align*}
Thus, by Fact \ref{fctcon}, $\mathbf{n}_1,\mathbf{n}_2\in \T_{F(\tau)}^\perp C$. Next let us show that $\mathbf{n}_1\perp \mathbf{n}_2,$
\begin{align*}
	\ka(\mathbf{n}_1,\mathbf{n}_2)=&\frac{1}{2}\tr\left(\left(\xi_1\xi_1^T-\xi_2\xi_2^T\right)\left(\xi_1\xi_2^T+\xi_2\xi_1^T\right)\right)\\
	=&\frac{1}{2}\tr(\xi_1\xi_1^T\xi_1\xi_2^T+\xi_1\xi_1^T\xi_2\xi_1^T-\xi_2\xi_2^T\xi_1\xi_2^T-\xi_2\xi_2^T\xi_2\xi_1^T)\\
	=&\frac{1}{2}\tr(\xi_2^T\xi_1\xi_1^T\xi_1)+\frac{1}{2}\tr(\xi_1^T\xi_1\xi_1^T\xi_2)-\frac{1}{2}\tr(\xi_2^T\xi_2\xi_2^T\xi_1)-\frac{1}{2}\tr(\xi_1^T\xi_2\xi_2^T\xi_2)\\
	=&\frac{1}{2}\de(\xi_2,\xi_1)\de(\xi_1,\xi_1)+\frac{1}{2}\de(\xi_1,\xi_1)\de(\xi_1,\xi_2)-\frac{1}{2}\de(\xi_2,\xi_2)\de(\xi_2,\xi_1)-\frac{1}{2}\de(\xi_1,\xi_2)\de(\xi_2,\xi_2)\\
	=&0.
\end{align*}
Moreover, we have
\begin{align*}
	\ka(\mathbf{n}_1,\mathbf{n}_1)=&\frac{1}{2}\tr\left(\left(\xi_1\xi_1^T-\xi_2\xi_2^T\right)\left(\xi_1\xi_1^T+\xi_2\xi_2^T\right)\right)\\
	=&\frac{1}{2}\tr(\xi_1\xi_1^T\xi_1\xi_1^T-\xi_1\xi_1^T\xi_2\xi_2^T-\xi_2\xi_2^T\xi_1\xi_1^T+\xi_2\xi_2^T\xi_2\xi_2^T)\\
	=&\frac{1}{2}\tr(\xi_1^T\xi_1\xi_1^T\xi_1)-\frac{1}{2}\tr(\xi_2^T\xi_1\xi_1^T\xi_2)-\frac{1}{2}\tr(\xi_1^T\xi_2\xi_2^T\xi_1)+\frac{1}{2}\tr(\xi_2^T\xi_2\xi_2^T\xi_2)\\
	=&\frac{1}{2}\de(\xi_1,\xi_1)\de(\xi_1,\xi_1)-\frac{1}{2}\de(\xi_2,\xi_1)\de(\xi_1,\xi_2)-\frac{1}{2}\de(\xi_1,\xi_2)\de(\xi_2,\xi_1)+\frac{1}{2}\de(\xi_2,\xi_2)\de(\xi_2,\xi_2)\\
	=&1.
\end{align*}
and
\begin{align*}
	\ka(\mathbf{n}_2,\mathbf{n}_2)=&\frac{1}{2}\tr\left(\left(\xi_1\xi_2^T+\xi_2\xi_1^T\right)\left(\xi_1\xi_2^T+\xi_2\xi_1^T\right)\right)\\
	=&\frac{1}{2}\tr(\xi_1\xi_2^T\xi_1\xi_2^T+\xi_1\xi_2^T\xi_2\xi_1^T+\xi_2\xi_1^T\xi_1\xi_2^T+\xi_2\xi_1^T\xi_2\xi_1^T)\\
	=&\frac{1}{2}\tr(\xi_2^T\xi_1\xi_2^T\xi_1)+\frac{1}{2}\tr(\xi_1^T\xi_1\xi_2^T\xi_2)+\frac{1}{2}\tr(\xi_2^T\xi_2\xi_1^T\xi_1)+\frac{1}{2}\tr(\xi_1^T\xi_2\xi_1^T\xi_2)\\
	=&\frac{1}{2}\de(\xi_2,\xi_1)\de(\xi_2,\xi_1)+\frac{1}{2}\de(\xi_1,\xi_1)\de(\xi_2,\xi_2)+\frac{1}{2}\de(\xi_2,\xi_2)\de(\xi_1,\xi_1)+\frac{1}{2}\de(\xi_1,\xi_2)\de(\xi_1,\xi_2)\\
	=&1.
\end{align*}
\end{proof}
\subsection{Complexification of $\VV$}
We need to use complexify $\VV$ in order to use deeper results from representation theory \cite{RFsi,LPcs,KNOT}. To this end, we need to complexify $\VV.$
\begin{defn}
	Define $$\VV_\C=\VV\ts \C=\VV\oplus i\VV.$$ Extend the conjugation action of $\sos$ on $\VV$ complex linearly to $\VV_\C$, i.e., acting as conjugation on $3\times 3$ complex matrices. Extend $\ka_\C$ to a Hermitian inner product $\ka_\C$ on $\VV_\C$ by setting
	\begin{align*}
		\ka(\vv_\C,\uu_\C)=\frac{1}{2}\tr(\vv_\C\ov{\uu_\C}),
		\end{align*}with $\vv_\C,\uu_\C\in \VV_\C$ and the symbol bar meaning the complex conjugation of matrices.\end{defn}
Here symbol bar means complex conjugation of matrices.	Note $\ka_\C$ is still invariant under conjugation actions of $\sos.$
\begin{fact}\label{fctirred}
	The conjugation action of $\sos$ on $\VV_\C$ is a complex irreducible representation of $\sos.$ 
\end{fact}
\begin{proof}
	Recall the last bullet of Fact \ref{fctorbv}, i.e., that the conjugation action of $\sos$ is irreducible on the $5$-real-dimensional vectors space $\VV.$ By \cite[Chapter II, (6.3) Theorem and (6.6) Proposition (i) to (iii)]{BTrcp}, the action of $\sos$ on $\VV_\C$ is either irreducible or splits into two complex subrepresentations of the same dimension. In the latter case, we deduce that $\VV_\C$ has even complex dimension, which is impossible. We are done.
\end{proof}
The complexification $\VV_\C$ provides alternative ways to describe $2$-vectors.
\begin{defn}
	For $S^2_{\sqrt[4]{3}}(0)\setminus\{(0,0,\sqrt[4]{3}),(0,0,-\sqrt[4]{3})\}$, define
	\begin{align*}
		\nn(\tau)&=\nn_1\wedge\nn_2,\\
		\nn_\C(\tau)&=\frac{1}{\sqrt{2}}\nn_1+\frac{1}{\sqrt{2}}i\nn_2=\frac{1}{\sqrt{2}}(\xi_1+\xi_2 i)(\xi_1+\xi_2 i)^T.
	\end{align*}
\end{defn}
\begin{lem}\label{leminp}
		For any simple $2$-vector $\eta=\ww_1\wedge\ww_2$ with $\ww_1,\ww_2$ an orthonormal basis of $\eta,$ define
	\begin{align*}
		\eta_\C=\frac{1}{\sqrt{2}}\ww_1+\frac{1}{\sqrt{2}}\ww_2 i.
	\end{align*}
	We have
	\begin{align}\label{eqconj}\nn_\C(\tau)=&\ov{\nn_\C(-\tau)},\\
|\ka_\C(\nn_\C(\tau),\eta_\C)|^2-		|\ka_\C(\nn_\C(-\tau),\eta_\C)|^2=&\ka(\nn_1(\tau)\wedge \nn_2(\tau),\ww_1\wedge\ww_2).
	\end{align}
\end{lem}
\begin{proof}
	First observe that
	\begin{align*}
		\xi_1(\tau)=-\xi_1(-\tau),\xi_2(\tau)=\xi_2(-\tau),
	\end{align*}where again symbol bar means complex conjugation.
	
	We have
	\begin{align*}
		\ka_\C(\nn_\C(\tau),\eta_\C)=&\frac{1}{4}\tr((\nn_1(\tau)+\nn_2(\tau)i)(\ww_1-\ww_2 i))\\
		=&\frac{1}{4}\left(\tr(\nn_1(\tau) \ww_1)+\tr(\nn_2(\tau)\ww_2)-\tr(\nn_1(\tau)\ww_2)i+\tr(\nn_2(\tau)\ww_1)\right)i.
	\end{align*}
	And similarly
	\begin{align*}
	\ka_\C(\nn_\C(-\tau),\eta_\C)=&\frac{1}{4}\tr((\nn_1(\tau)-\nn_2(\tau)i)(\ww_1-\ww_2 i))\\
	=&\frac{1}{4}\left(\tr(\nn_1(\tau) \ww_1)-\tr(\nn_2(\tau)\ww_2)-\tr(\nn_1(\tau)\ww_2)i-\tr(\nn_2(\tau)\ww_1)\right)i.
\end{align*}
Thus, we have
\begin{align*}
&|\ka_\C(\nn_\C(\tau),\eta_\C)|^2-|\ka_\C(\nn_\C(-\tau),\eta_\C)|^2\\
=&\frac{1}{16}\bigg((\tr(\nn_1(\tau) \ww_1)+\tr(\nn_2(\tau)\ww_2))^2+(-\tr(\nn_1(\tau)\ww_2)+\tr(\nn_2(\tau)\ww_1))^2\\&
-(\tr(\nn_1(\tau) \ww_1)-\tr(\nn_2(\tau)\ww_2))^2-(-\tr(\nn_1(\tau)\ww_2)-\tr(\nn_2(\tau)\ww_1))^2\bigg)\\
=&\frac{1}{16}\bigg(4\tr(\nn_1(\tau) \ww_1)\tr(\nn_2(\tau)\ww_2))-4\tr(\nn_1(\tau)\ww_2)\tr(\nn_2(\tau)\ww_1)\bigg)\\=&\ka(\nn_1(\tau),\ww_1)\ka(\nn_2(\tau),\ww_2)-\ka(\nn_1(\tau),\ww_2)\ka(\nn_2(\tau),\ww_1)\\=&\ka(\nn_1(\tau)\wedge\nn_2(\tau),\ww_1\wedge \ww_2).
\end{align*}
\end{proof}
\subsection{Lie algbera $\soss$ and eigenvectors}
	Let us first provide an orthonormal basis of $\soss$. 
\begin{defn}Define $\si_1,\si_2,\si_3$ as follows
	\begin{align*}\si_1=\begin{bmatrix}
			0&0&0\\
			0&0&-1\\
			0&1&0\\
		\end{bmatrix},\si_2=\begin{bmatrix}
		0&0&1\\
		0&0&0\\
		-1&0&0\\
		\end{bmatrix},
		\si_3=\begin{bmatrix}
		0&-1&0\\
		1&0&0\\
		0&0&0\\
		\end{bmatrix}.
		\end{align*}
\end{defn}
\begin{fact}\label{fctgen}\cite[Example 3.10]{AKil} We have 
		\begin{align*}
		[\si_1,\si_2]=\si_3,[\si_2,\si_3]=\si_1,[\si_3,\si_1]=\si_2.
	\end{align*}
\end{fact}
\begin{fact}\label{fcteigen}
	We have\begin{align*}
		[x\si_1+y\si_2+z\si_3,\nn_\C(\tau)]
		&=-2i\nn_\C(\tau),\\
		[x\si_1+y\si_2+z\si_3,\nn_\C(-\tau)]
		&=2i\nn_\C(-\tau).
	\end{align*}
\end{fact}
\begin{proof}
First recall the well-known Rodrigues’ formula for rotations \cite[discussions around (1.2.6) (1.2.7) and (1.4.18)]{DKlg} the matrix exponential $\exp(\theta(x\si_1+y\si_2+z\si_3))$ is the rotation around the axis $(x,y,z)$ by angle $\theta$ so that $(x,y,z),v,\exp(\theta(x\si_1+y\si_2+z\si_3))v$ is positively oriented for $v$ normal to $(x,y,z).$ As $\xi_0,\xi_1,\xi_2$ is positively oriented and $\xi_1,\xi_2$ form an orthonormal basis of the orthogonal subspace to $(x,y,z)$, we deduce that 
\begin{align*}
	&\exp(\theta(x\si_1+y\si_2+z\si_3))(\xi_1+\xi_2i)\\=&(\cos\ta \xi_1+\sin\ta\xi_2)+(-\sin\ta\xi_1+\cos\ta \xi_2)i\\=&\cos\ta(\xi_1+\xi_2i)-\sin\ta(\xi_1+\xi_2i)i\\=&e^{-\theta i}(\xi_1+i\xi_2).
\end{align*} We have
\begin{align*}
	&\exp(\theta(x\si_1+y\si_2+z\si_3))\nn_\C(\tau)\exp(\theta(x\si_1+y\si_2+z\si_3))\m\\
	=&\exp(\theta(x\si_1+y\si_2+z\si_3))\frac{1}{\sqrt{2}}(\xi_1+\xi_2 i)(\xi_1+\xi_2 i)^T\exp(\theta(x\si_1+y\si_2+z\si_3))^T\\
	=&\frac{1}{\sqrt{2}}(\exp(\theta(x\si_1+y\si_2+z\si_3))(\xi_1+\xi_2 i))(\exp(\theta(x\si_1+y\si_2+z\si_3))(\xi_1+\xi_2 i))^T\\
	=&\frac{1}{\sqrt{2}}e^{-2\theta i}(\xi_1+\xi_2i)(\xi_1+\xi_2i)^T\\
	=&e^{-2\theta i}\nn_\C(\tau).
\end{align*}
Now differentiate at $\theta=0,$ we have
\begin{align*}
	[x\si_1+y\si_2+z\si_3,\nn_\C(\tau)]	= -2i \nn_\C(\tau).
\end{align*}
Apply complex conjugation and use (\ref{eqconj}). We are done.
\end{proof}
\subsection{Inequality from representation theory}
\begin{fact}\label{fctcs}
	For any convex function $\Phi:[0,1]\to \R$, we have
	\begin{align}\label{eqcs1}
		&\sup_{\substack{\psi\in\VV_\C\\\ka_\C(\psi,\psi)=1}}\int_{(x,y,z)\in S^2}\Phi\left(\left|\ka\left(\nn_\C(\sqrt[4]{3}(x,y,z)),\psi\right)\right|^2\right)\dhm^2(x,y,z)\\\label{eqcs2}=&\pi\int_0^1\Phi(s)s^{-\frac{3}{4}}ds\\
		\label{eqcs3}=&	\sup_{\substack{\psi\in\VV_\C\\\ka_\C(\psi,\psi)=1}}\int_{(x,y,z)\in S^2}\Phi\left(\left|\ka\left(\nn_\C(-\sqrt[4]{3}(x,y,z)),\psi\right)\right|^2\right)\dhm^2(x,y,z).
	\end{align}and the supremum is attained for $\psi=\nn_\C(\tau)$ for any $\tau\in S^2_{\sqrt[4]{3}}(0)\setminus\{(0,0,\sqrt[4]{3}),(0,0,-\sqrt[4]{3})\}.$
\end{fact}
Note that the above integrals are well-defined as the integrands are not defined only on two poles of $S^2_{\sqrt[4]{3}}(0),$ a Hausdorff $2$-dimensional measure $0$ set.
\begin{proof}
	This is the coherent state inequality obtained first by \cite{LPcs} with equality case determined by \cite{RFsi,KNOT}. Let us verify the required conditions. We will use the version stated as \cite[Theorem 4]{RFsi}.
	
	The paragraph before \cite[Theorem 4]{RFsi} explains the terminologies, which uses irreducible representations of $\operatorname{SU}(2)$ and its Lie algebra $\mathfrak{su}(2).$ We have actually verified all required conditions on $\sos$. Thus, conditions mentioned in the paragraph just before \cite[Theorem 4]{RFsi} follows by passing to the double cover $\operatorname{SU}(2)$ of $\sos$. Let us get into the details.
	
	First, by Fact \ref{fctirred}, the conjugation action of $\sos$ on $\VV_\C$ is complex irreducible. As lifting to coverings does not change invariant subspaces, the lifting of the conjugation action of $\sos$ on $\VV_\C$ to $\operatorname{SU}(2)$ is an irreducible representaion. As $\dim_\C\VV_\C=5,$ we have $J=2$ in the notations of \cite{RFsi}.
	
	Note that \cite{RFsi} is using physicists' convention of multiplying skew-hermitian Lie algebras by $i$ \cite[Sectioon (2.21)]{BTrcp}. Thus the operators satisfying $[S_1,S_2]=iS_3$ in the notation of \cite{RFsi} and cyclically, translates to operators satisfying $[iS_1,iS_2]=i(iS_3)$  and cyclically, i.e., $[S_1,S_2]=S_3$ and cyclically.
	
	The conjugation action of $\sos$ on $\VV_\C$ naturally induces a representation of the Lie algebra $\soss$ by differentiation and acts via commutators. In other words, for $\si\in\soss$ and $\vv_\C\in\VV_\C$, the action of $\si$ on $\vv_\C$ is
	\begin{align*}
		[\si,\vv_\C].
	\end{align*} Recall Fact \ref{fctgen}. By definition of representation of Lie algebra, we deduce that 
	\begin{align*}
		S_1=[\si_1,\cdot],
		S_2=[\si_2,\cdot],
		S_3=[\si_3,\cdot],
	\end{align*}satisfy $[S_1,S_2]=S_3$ and cyclically.
	
By Fact \ref{fcteigen}, $\nn_\C(-\tau)$ is an eigenfunction of $\omega.S$ in notations of \cite{RFsi} with eigenvalue $2i,$ which by physicists convention of multiplying $i,$ is an eigenfunction with eigenvalue $-2=-J.$ Thus, in notations of \cite{RFsi},  $\psi_\omega=\nn_\C(-\tau)$. Now apply \cite[Theorem 4]{RFsi} and we get the equality from (\ref{eqcs2}) to (\ref{eqcs3}). The  equality (\ref{eqcs1}) to (\ref{eqcs2}) is just a change of variables.
\end{proof}
\begin{lem}\label{lemcsc}
	For any $3$-dimensional subspace $W\s\VV,$ we have
	\begin{align*}
	\int_{C_1}|\ka({W,\T_pC_1})|\dhm^3(p)\le\frac{3\pi}{4}
	\end{align*}
\end{lem}
\begin{proof}
By coarea formula \cite[equation 10.3]{LS} and homogeneity of $C_1$, we have
\begin{align}
	&\int_{C_1}|\ka({W,\T_pC_1})|\\=&\int_{0}^1\int_{t\vrpt}|\ka({W,\T_pC_1})|\dhm^2(p)dt\\
	=&\int_0^1t^2\int_{\vrpt}|\ka({W,\T_pC})|\dhm^2(p)dt\\
	=&\frac{1}{3}\int_{\vrpt}|\ka({W,\T_pC})|\dhm^2(p)\\
	=&\frac{1}{3}\times\frac{1}{2}\int_{S^2_{\sqrt[4]{3}}(0)}|\ka(W,\T_{F(\tau)C})|\sqrt{3}{\dhm^2(\tau)},
\end{align}where we have used Fact \ref{fctmet} at the last step.

Let $\ww_1,\ww_2$ be an orthonormal frame of the orthogonal complement of $W$, i.e., $$\VV=W\oplus \R \ww_1\oplus \R \ww_2.$$ Set
\begin{align*}
	W_\C=\frac{1}{\sqrt{2}}\ww_1+\frac{1}{\sqrt{2}}\ww_2i.
\end{align*}
Recall Fact \ref{fctnmc} and Lemma \ref{leminp}. By using Hodge dual on the exterior algebra,
\begin{align}
\label{eqint1}	&\int_{S^2_{\sqrt[4]{3}}(0)}|\ka(W,\T_{F(\tau)C})|\sqrt{3}\dhm^2(\tau)\\
	=&\int_{S^2_{\sqrt[4]{3}}(0)}|\ka(\ww_1\wedge \ww_2,\nn_1(\tau)\wedge\nn_2(\tau))|\sqrt{3}\dhm^2(\tau)\\
	=&\int_{S^2_{\sqrt[4]{3}}(0)}\left||\ka_\C(\nn_\C(-\tau),W_\C)|^2-|\ka_\C(\nn_\C(\tau),W_\C)|^2\right|\sqrt{3}\dhm^2(\tau)\\
	=&3\int_{S^2_{1}(0)}\left||\ka_\C(\nn_\C(-\sqrt[4]{3}(x,y,z)),W_\C)|^2-|\ka_\C(\nn_\C(\sqrt[4]{3}(x,y,z)),W_\C)|^2\right|{\dhm^2(x,y,z)}\\
	\le&3\label{eqis2}\int_{S^2_{1}(0)}\left||\ka_\C(\nn_\C(\sqrt[4]{3}(x,y,z)),W_\C)|^2-\frac{1}{16}\right|{\dhm^2(x,y,z)}\\&+3\int_{S^2_{1}(0)}\left||\ka_\C(\nn_\C(-\sqrt[4]{3}(x,y,z)),W_\C)|^2-\frac{1}{16}\right|{\dhm^2(x,y,z)}\\
	\label{eqint3}\le&6\pi\int_0^1\left|s-\frac{1}{16}\right|s^{-\frac{3}{4}}ds,
\end{align} 
where we have used triangle inequality at (\ref{eqint2}) and used Fact \ref{fctcs} at (\ref{eqint3}). On the other hand we have
\begin{align}
&\int_0^1\left|s-\frac{1}{16}\right|s^{-\frac{3}{4}}ds=\int_0^{\frac{1}{16}}\left(\frac{1}{16}-s\right)s^{-\frac{3}{4}}ds+\int_{\frac{1}{16}}^1\left(s-\frac{1}{16}\right)s^{-\frac{3}{4}}ds\\=&\frac{1}{16}4s^{\frac{1}{4}}\bigg|^{\frac{1}{16}}_0-\frac{4}{5}s^{\frac{5}{4}}\bigg|_0^{\frac{1}{16}}+\frac{4}{5}s^{\frac{5}{4}}\bigg|^1_{\frac{1}{16}}-\frac{1}{16}4s^{\frac{1}{4}}\bigg|_{\frac{1}{16}}^1=\frac{1}{8}-\frac{1}{40}+\frac{4}{5}-\frac{1}{40}-\frac{1}{4}+\frac{1}{8}\\=&\frac{3}{4}\label{eqint4}
\end{align}Combining (\ref{eqint1}) to (\ref{eqint4}) yields
\begin{align*}
	\int_{C_1}|\ka({W,\T_pC_1})|\dhm^3(p)\le \frac{3\pi}{4}
\end{align*}
\end{proof}
\subsection{Determining the pushforward $\pi_{\T_{\vv_1}C}(C_1)$ and its area}
With Lemma \ref{lemcsc} in mind, to prove use moderation Theorem \ref{thmmod}, we need to determine $\ms(\pi_{\T_{\mathbf{p}}C}(C_1))$ with $\pp$ a point in $C_1$ and $\pi_{\T_{\mathbf{p}}C}$ the orthogonal projection to $\T_{\mathbf{p}}C_1$. 

Let us first reduce by symmetry
\begin{fact}\label{fcthom}
For $\pp\in C_1$, we have
$$\ms(\pi_{\T_{\mathbf{p}}C}(C_1))=\ms(\pi_{\T_{\vv_1}C}(C_1)).$$
\end{fact}
\begin{proof}
First, $C$ is a homogeneous cone, so we deduce that $$\T_{\pp}C=\T_{\frac{\pp}{\sqrt{\ka(\pp,\pp)}}}C.$$ Since $\vrpt$ is an $\sos$ orbit, there exists $g\in\sos$ so that $$\vrpt\ni\frac{\pp}{\sqrt{\ka(\pp,\pp)}}=g\vv_1g\m.$$ 

Next, as $C$ is $\sos$-invariant, we deduce that $\T_{g\vv_1g\m}C$ is the image of $T_{\vv_1}C$ under $g$ conjugation. Thus, we have
\begin{align*}
\ms(\pi_{\T_{\mathbf{p}}C}(C_1))=\ms(\pi_{\T_{g\vv_1g\m}C}(C_1))=\ms(\pi_{g\T_{\vv_1}Cg\m}(C_1))=\ms(\pi_{\T_{\vv_1}C}(g\m C_1g))=\ms(\pi_{\T_{\vv_1}C}(C_1)).
\end{align*}
\end{proof}
Next, let us determine $\T_{\vv_1}C.$
\begin{fact}We have
		\begin{align}\label{cuvw}
		\T_{\vv_1}C=\R\vv_1\oplus\R\vv_2\oplus\R\vv_3=\{u\vv_1+v\vv_2+w\vv_3|u,v,w\in\R\}.
	\end{align} 
\end{fact}\begin{proof}
We have	\begin{align*}
	dF_{(0,0,\sqrt[4]{3})}(1,0,0)=-\begin{bmatrix}
		0 \\
		0 \\
		\sqrt[4]{3}
	\end{bmatrix}\begin{bmatrix}
		1 \\
		0
		\\	0 
	\end{bmatrix}^T-\begin{bmatrix}
		1 \\
		0 \\
		0\end{bmatrix}\begin{bmatrix}
		0 \\
		0
		\\	\sqrt[4]{3} 
	\end{bmatrix}^T=-\sqrt[4]{3}\vv_2,\\
	dF_{(0,0,\sqrt[4]{3})}(0,1,0)=-\begin{bmatrix}
		0 \\
		0 \\
		\sqrt[4]{3}
	\end{bmatrix}\begin{bmatrix}
		0 \\
		1
		\\	0 
	\end{bmatrix}^T-\begin{bmatrix}
		0 \\
		1 \\
		0\end{bmatrix}\begin{bmatrix}
		0 \\
		0
		\\	\sqrt[4]{3} 
	\end{bmatrix}^T=-\sqrt[4]{3}\vv_3.
\end{align*}By Fact \ref{fctmet}, We deduce that $\T_{\vv_1}C$ is spanned by 
\begin{align*}
	\T_{\vv_1}C=\R\vv_1\oplus\R\vv_2\oplus\R\vv_3.
\end{align*} 
\end{proof}
From now on, we will use the coordinate $(u,v,w)$ on $\T_{\vv_1}C.$
\begin{lem}\label{lemproj}The mod $2$ current $\pi_{\T_{\vv_1}C_1}$ equals the ellipsoid
\begin{align*}
	\left\{(u,v,w)\bigg|\frac{u^2}{\frac{3}{4}}+\frac{v^2}{\frac{3}{4}}+\frac{\left(u-\frac{1}{4}\right)^2}{\frac{9}{16}}\le1\right\},
\end{align*}and thus we have
\begin{align*}
	\ms(\pi_{\T_{\vv_1}C}(C_1))=\frac{3}{4}\pi.
\end{align*}\end{lem}
\begin{proof}
	For $\tau=\sqrt[4]{3}(x,y,z)\in S^2_{\sqrt[4]{3}}(0)$, we have\begin{align*}
		F(\tau)=&-\sqrt{3}\begin{bmatrix}
			x \\
			y \\
			z
		\end{bmatrix}\begin{bmatrix}
			x \\
			y
			\\	z 
		\end{bmatrix}^T+\frac{1}{\sqrt{3}}I\\
		=&-\sqrt{3}\begin{bmatrix}
			x^2 & xy & xz \\
			xy & y^2 & yz \\
			xz & yz & z^2
		\end{bmatrix}+\frac{1}{\sqrt{3}}I\\
		=&\begin{bmatrix}
			\frac{1-3x^2}{\sqrt{3}} & -\sqrt{3}xy &- \sqrt{3}xz \\
			-\sqrt{3}xy & \frac{1-3y^2}{\sqrt{3}} &- \sqrt{3}yz \\
			-\sqrt{3}xz & -\sqrt{3}yz &\frac{1-3z^2}{\sqrt{3}} 	\end{bmatrix}\\
		=&\frac{1-\frac{3}{2}(x^2+y^2)}{\sqrt{3}}\vv_1-\sqrt{3}xz\vv_2-\sqrt{3}yz\vv_3-\sqrt{3}xy\vv_4+\frac{-\frac{3}{2}(x^2-y^2)}{\sqrt{3}}\vv_5\\
		=&\frac{3z^2-1}{2}\vv_1-\sqrt{3}xz\vv_2-\sqrt{3}yz\vv_3-\sqrt{3}xy\vv_4+\frac{\sqrt{3}(y^2-x^2)}{2}\vv_5.
	\end{align*}
		Thus, we have
	\begin{align}\label{eqpt}
		\pi_{\T_{\vv_1}\C}(F(\tau))=\frac{3z^2-1}{2}\vv_1-\sqrt{3}xz\vv_2-\sqrt{3}yz\vv_3.
	\end{align}The image (\ref{eqpt}) satisfy the quadratic equation in coordinate $(u,v,w)$ (\ref{cuvw}),
\begin{align}\label{eqep}
	&3v^2+3w^2=9(x^2+y^2)z^2=9(1-z^2)z^2\\
	=&9\left(1-\frac{2u+1}{3}\right)\frac{2u+1}{3}=(2-2u)(2u+1)\\
	=&-4u^2+2u+2\\
	=&-\left(2u-\frac{1}{2}\right)^2+\frac{9}{4}.
\end{align}
Rearranging (\ref{eqep}) gives
\begin{align}\label{eqesd}
\frac{v^2}{\frac{3}{4}}+\frac{w^2}{\frac{3}{4}}+\frac{\left(u-\frac{1}{4}\right)^2}{\frac{9}{16}}=1.
\end{align}
Thus, we have as sets\begin{align*}
	\pi_{\T_{\vv_1}C}(\vrpt)\s\left\{(u,v,w)\bigg|\frac{v^2}{\frac{3}{4}}+\frac{w^2}{\frac{3}{4}}+\frac{\left(u-\frac{1}{4}\right)^2}{\frac{9}{16}}=1\right\}.
\end{align*}
In other words, the mod $2$ current cycle $\pi_{\T_{\vv_1}C}(\vrpt)$ is supported on a subset of the closed submanifold $\left\{(u,v,w)\bigg|\frac{v^2}{\frac{3}{4}}+\frac{w^2}{\frac{3}{4}}+\frac{\left(u-\frac{1}{4}\right)^2}{\frac{9}{16}}=1\right\}$. By Fact \ref{const}, either we have as mod $2$ currents, \begin{align}\label{pivrpt}
	\pi_{\T_{\vv_1}C}(\vrpt)=\left\{(u,v,w)\bigg|\frac{v^2}{\frac{3}{4}}+\frac{w^2}{\frac{3}{4}}+\frac{\left(u-\frac{1}{4}\right)^2}{\frac{9}{16}}=1\right\},
\end{align}
or as mod $2$ currents,
\begin{align}\label{eqnull}
	\pi_{\T_{\vv_1}C}(\vrpt)=0.
\end{align} 
The latter cannot happen. To see this, 	let us calculate the differential of $\pi_{\T_{\vv_1}C}\circ F.$ We have
\begin{align*}
	d(\pi_{\T_{\vv_1}C}(F))_{\sqrt[4]{3}(x,y,z)}=\begin{pmatrix}
		0&0&{3}z\\
		-\sqrt{3}z&0&-\sqrt{3}x\\
		0&-\sqrt{3}z&-\sqrt{3}y.
	\end{pmatrix}
\end{align*}
This gives
\begin{align*}
	\det\left(d(\pi_{\T_{\vv_1}C}(F))_{\sqrt[4]{3}(x,y,z)}\right)=9z^3.	
\end{align*}
As $F$ is a double covering to $\vrpt$ Fact \ref{fctmet}, we deduce that $\pi_{\T_{\vv_1}C}$ restricted to $\vrpt$ is an immersion except for on the circle $F(\ssf\cap \{z=0\}).$ When $z=0,$ the circle $\ssf\cap\{z=0\}$ is mapped to one point $-\frac{1}{2}\vv_1.$ 

Next let us upgrade that immersion to an embedding. On the other hand, suppose for $zz'\not=0,$ and $(x,y,z)\not=(x',y',z'),$ we have
\begin{align*}
\frac{3z^2-1}{2}\vv_1-\sqrt{3}xz\vv_2-\sqrt{3}yz\vv_3=\frac{3(z')^2-1}{2}\vv_1-\sqrt{3}x'z'\vv_2-\sqrt{3}y'z'\vv_3,
\end{align*}
then we deduce that
\begin{align*}
	z'=\pm z,y'=\pm y,x'=\pm y.
\end{align*}
As $F$ is a double covering, we deduce that  $\pi_{\T_{\vv_1}C}$ restricted to $\vrpt$ is an embedding except for on the circle $F(\ssf\cap \{z=0\})$, which is $1$-dimensional and mapped to a point. Thus, Hausdorff $2$-dimensional almost every point in $\pi_{\T_{\vv_1}C}(\vrpt)$ has only one preimage and (\ref{eqnull}) cannot happen. In other words, (\ref{pivrpt}) holds.

Note that $(u,v,w)=(0,0,0)$ is enclosed by the ellipsoid surface (\ref{eqesd}). As $C_1$ is union of line segements from $0$ to points on $\vrpt$ and line segments maps to line segements under linear maps, we deduce that $\pi_{\T_{\vv_1}C}(C_1)$, projection of truncated cone, equals the truncated cone over $\pi_{\T_{\vv_1}C}(\vrpt)$ in $\T_{\vv_1}C$, i.e., as mod $2$ currents,
\begin{align*}
	\pi_{\T_{\vv_1}C}(C_1)=\left\{(u,v,w)\bigg|\frac{v^2}{\frac{3}{4}}+\frac{w^2}{\frac{3}{4}}+\frac{\left(u-\frac{1}{4}\right)^2}{\frac{9}{16}}\le1\right\},
\end{align*}which is a ellipsoid of volume 
\begin{align*}
\ms(\pi_{\T_{\vv_1}}C_1)=	\frac{4}{3}\pi\sqrt{\frac{3}{4}}\sqrt{\frac{3}{4}}\sqrt{\frac{9}{16}}=\frac{3}{4}\pi.
\end{align*}
\end{proof}	
\subsection{Wrapping up the proof of Theorem \ref{thmc}}
By Fact \ref{fcthom} and Lemma \ref{lemproj}, we deduce that
, we need to determine for any point $\pp\in C_1$, we have
\begin{align*}
	\ms(\pi_{\T_{\mathbf{p}}C}(C_1))=\frac{3}{4}\pi.
\end{align*} 
By  Lemma \ref{lemcsc} and Fact \ref{fctmap}, we deduce that
\begin{align*}
		\int_{C_1}|\ka({\T_{\pp}C_1,\T_{\ww}C_1})|\dhm^3(\ww)\le\frac{3\pi}{4}=\ms(\pi_{\T_{\mathbf{p}}C}(C_1))\le\int_{C_1}|\ka({\T_{\pp}C_1,\T_{\ww}C_1})|\dhm^3(\ww),
\end{align*}
i.e.,
\begin{align}\label{jac}
	\int_{C_1}|\ka({\T_{\pp}C_1,\T_{\ww}C_1})|\dhm^3(\ww)=\frac{3}{4}\pi.
\end{align}
By  (\ref{jac}), Lemma \ref{lemproj}, Lemma \ref{lemcsc}, we deduce that $C_1$ is a moderation in the sense of Definition \ref{defnmod}. Apply Theorem \ref{thmmod} and we deduce that $C_1$ is mod $2$ area-minimizing. Apply  homotheties $\vv\to t\vv$ to $C_1$, we deduce that $C$ restricted to balls of radius $t$ centered at $0$ in $\VV$ is mod $2$ area-minimizing for every $t\in(0,\infty)$. Thus $C$ is mod $2$ area-minimizing. We are done.\section{Preparing a representative of $[\Si]$}\label{preprep}
	By our assumptions on the homology class $[\Si]$ in Theorem \ref{thmm}, the class $[\Si]$ has a smooth representative finitely many connected components. In other words $[\Si]$ has an mod $2$ current representative as follows:
	\begin{align}\label{sigs}
		\Si=N_1+N_2+\dots+N_n,
	\end{align}where $\{N_1,\dots, N_n\}$ is a collection of pairwise disjoint $d$-dimensional smooth submanifolds.
	
	To make $\Si$ in (\ref{sigs}) connected, 
	we need to apply  the connected sum of    currents mod $2$ defined in Lemma \ref{lemcs} to $\Si$. In other words we obtain a representative $N$ of $\Si$, where
	\begin{align*}
		N=N_1\# \dots\#N_n.
	\end{align*}
	By Lemma \ref{lemcs}, $N$ is an embedded connected submanifold  representing the homology class $[\Si]$.
	
	To sum it up, the assumption on $[\Si]$ in the statement of Theorem \ref{thmm} can be reduced to the following assumption.
	\begin{assump}\label{simpa}
		The mod $2$ class $[\Si]$ is represented by a smoothly embedded connected submanifold $N$.
	\end{assump}\raggedbottom
	\section{Construction of singular sets f}\label{modsing}
	As outlined in the plan of our proof (Section \ref{planpf}),  we need to add singular sets manually to the topological representative $N$ of $[\Si]$. The way to do this is to construct a homologically trivial cycle $\si_s(C)$ on $M$  in Section \ref{secssc} with our desired singular set and using Lemma \ref{lemcs} to do a connected sum $N\#\si_s(C).$ This section will be devoted to the construction of the homologically trivial cycle $\si_s(C).$	
		\subsection{The $s$-th spherical link of a cone $C$}\label{secssc}
	Let $V$ be a $d-s-1$ dimensional not necessarily connected closed smooth submanifold of the unit sphere $S^{d-s+c-1}$of $\R^{d-s+c}$. Here $s$ is an integer at most $(d-2)$ that we will set to be the dimension of the singular set we want to add.
	\begin{defn}\label{defncv}
		The cone $C(V)$  over $V$ in $\R^{d-s+c}$ is defined to be the union of all rays from the origin passing through $V.$
	\end{defn}
	Here by a ray passing through a point $p$ we mean the half line parameterized by $tp$, with $t\in \R_{\ge 0}.$
	
	It is straightforward to verify that $C(V)$ is a boundaryless mod $2$ current. When $V$ is clear from the context, we will drop the symbol $V$ and write $C$ only. In case $C$ is not a flat subspace of $\R^{d-s+c}$, it is straightforward to see that
	\begin{align*}
		\sing C=\{0\}^{d-s+c}.
	\end{align*}
	From now on, we assume that $C$ is not a flat subspace of $\R^{d-s+c}$.	
	\begin{defn}
		Define the $s$-th spherical link of $C$ to be				\begin{align}\label{prodl}
			\si^s(C)=(C\times \R^{s+1})\cap{ S^{d+c}}.
		\end{align}
		The symbol $S^{d+c}$ denotes the $(d+c)$-dimensional unit sphere in $\R^{d+c+1}$ and $C\times \R^{s+1}$ is embedded naturally in the product $\R^{d-s+c}\times \R^{s+1}\cong \R^{d+c+1}.$ 		
	\end{defn}
	It is straightforward to verify that $\si^s(C)$ is a $d$-dimensional mod $2$ current on the $(d+c)$-dimensional sphere $S^{d+c}$. \begin{figure}
		\centering
		\def\svgwidth{0.9\paperwidth}
		%% Creator: Inkscape 1.3.2 (091e20e, 2023-11-25, custom), www.inkscape.org
%% Accompanies image file 'ssc.pdf' (pdf, eps, ps)
\begingroup%
  \makeatletter%
  \providecommand\color[2][]{%
    \errmessage{(Inkscape) Color is used for the text in Inkscape, but the package 'color.sty' is not loaded}%
    \renewcommand\color[2][]{}%
  }%
  \providecommand\transparent[1]{%
    \errmessage{(Inkscape) Transparency is used (non-zero) for the text in Inkscape, but the package 'transparent.sty' is not loaded}%
    \renewcommand\transparent[1]{}%
  }%
  \providecommand\rotatebox[2]{#2}%
  \newcommand*\fsize{\dimexpr\f@size pt\relax}%
  \newcommand*\lineheight[1]{\fontsize{\fsize}{#1\fsize}\selectfont}%
  \ifx\svgwidth\undefined%
    \setlength{\unitlength}{2160bp}%
    \ifx\svgscale\undefined%
      \relax%
    \else%
      \setlength{\unitlength}{\unitlength * \real{\svgscale}}%
    \fi%
  \else%
    \setlength{\unitlength}{\svgwidth}%
  \fi%
  \global\let\svgwidth\undefined%
  \global\let\svgscale\undefined%
  \makeatother%
  \begin{picture}(1,0.5)%
    \lineheight{1}%
    \setlength\tabcolsep{0pt}%
    \put(0,0){\includegraphics[width=\unitlength,page=1]{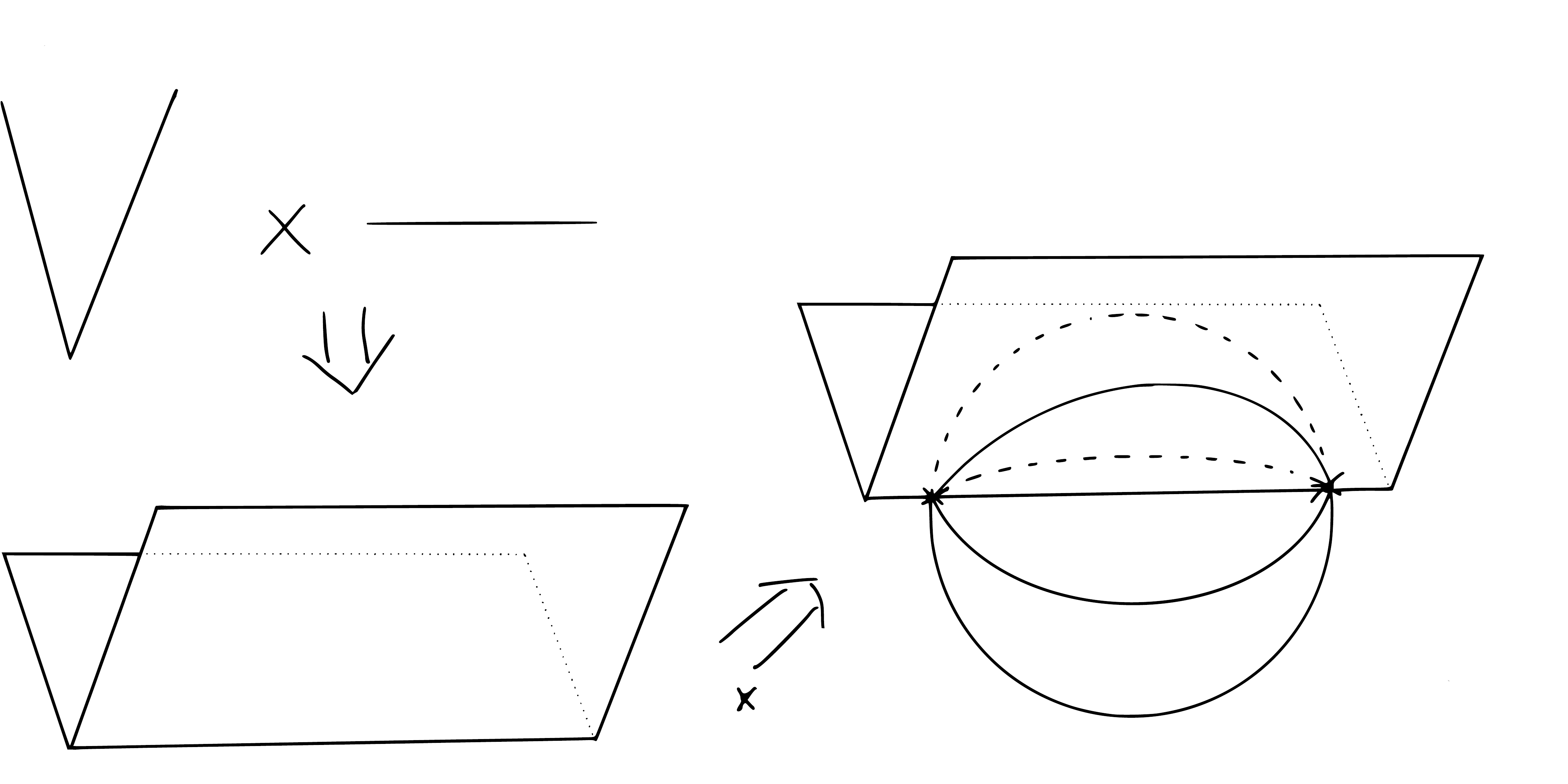}}%
    \put(0.00496863,0.45693687){\color[rgb]{0,0,0}\makebox(0,0)[lt]{\lineheight{1.25}\smash{\begin{tabular}[t]{l}$C^{d-s}\s \R^{d-s+c}$\end{tabular}}}}%
    \put(0.27824023,0.38284622){\color[rgb]{0,0,0}\makebox(0,0)[lt]{\lineheight{1.25}\smash{\begin{tabular}[t]{l}$\R^{s+1}$\end{tabular}}}}%
    \put(0.10669392,0.20236793){\color[rgb]{0,0,0}\makebox(0,0)[lt]{\lineheight{1.25}\smash{\begin{tabular}[t]{l}$C^{d-s}\times \R^{s+1}\s \R^{d+c+1}$\end{tabular}}}}%
    \put(0.42363669,0.02495004){\color[rgb]{0,0,0}\makebox(0,0)[lt]{\lineheight{1.25}\smash{\begin{tabular}[t]{l}$\sing \ssc =\{0\}^{d-s+c}\times S^s$\end{tabular}}}}%
    \put(0.56380286,0.3628129){\color[rgb]{0,0,0}\makebox(0,0)[lt]{\lineheight{1.25}\smash{\begin{tabular}[t]{l}$\ssc=(C^{d-s}\times \R^{s+1})\cap S^{d+c}$\end{tabular}}}}%
    \put(0.67572663,0.08771092){\color[rgb]{0,0,0}\makebox(0,0)[lt]{\lineheight{1.25}\smash{\begin{tabular}[t]{l}$S^{d+c}\s \R^{d+c+1}$\end{tabular}}}}%
  \end{picture}%
\endgroup%

		\caption{Construction of $\ssc$}
		\label{fig:ssc}
	\end{figure}
	\begin{lem}\label{singssc}
		There exists a neighborhood $U(\{0\}^{d-s+c}\times S^s)$ of $\{0\}^{d-s+c}\times S^s$ in $S^{d+c}$ diffeomorphic to the product of the unit ball $B_1^{d-s+c}$ in $\R^{d-s+c}$ with  $S^s$ by a diffeomorphism $\Ga$, i.e.,
		\begin{align*}
			U(\{0\}^{d-s+c}\times S^s)\overset{\Ga}{\cong}B_1^{d-s+c}\times S^s,
		\end{align*}such that $\ssc$ restricted to $U(\{0\}^{d-s+c}\times S^s)$ equals to
		\begin{align*}
			\big(C|_{B_1^{d-s+c}}\big)\times S^s,
		\end{align*}via $\Ga.$
		Moreover, the singular set of $\ssc$ is $\{0\}^{d-s+c}\times S^s$.
	\end{lem}
	Here $S^s$ denotes the unit sphere in $\R^{s+1} $ and we regard $S^0$ as two disjoint points. We use the symbol $T|_K$ to denote the restriction of an mod $2$ current to a set $K.$ 
	\begin{proof}
\cite[Lemma 4.1.3]{ZLns} states the result corresponding result with $C$ a multiplicity $1$ integral current cone. The same proof works. For the reader's convenience, we reproduce the entire proof.		Let us construct explicitly the diffeomorphism $\Ga$ and its inverse $\Ga\m.$

Adopt a coordinate system
\begin{align*}
	x=	(x_1,\dots,x_{d-s+c}),
\end{align*}on $\R^{d-s+c}$, and a coordinate system
\begin{align*}
	y=(y_1,\dots,y_{s+1}),
\end{align*}on $\R^{s+1}.$ The product structure $\R^{d+c+1}\cong \R^{d-s+c}\times \R^{s+1}$ gives a coordinate system
$		(x,y)
$		on $\R^{d+c+1}.$

Define a map $\Ga:\R^{d+c+1}\setminus\{|y|=0\}\to \R^{d+c+1}$ by 
\begin{align*}
	\Ga(x,y)=\bigg(\frac{x}{|y|},\frac{y}{|y|}\bigg).
\end{align*}
\begin{claim}\label{gabi}
	We claim that $\Ga$ is a smooth diffeomorphism from $S^{d+c}\setminus\{|y|=0\}$ onto $\R^{d-s+c}\times S^s.$
\end{claim}
To prove the above claim, define a smooth map $\Ga':\R^{d+c+1}\to\R^{d+c+1},$
\begin{align*}
	\Ga'(x,y)=\bigg(\frac{x}{\sqrt{1+|x|^2}},\frac{y}{\sqrt{1+|x|^2}}\bigg).
\end{align*}
The map	$\Ga'$ is smooth everywhere and we wish to show that
\begin{align*}
	\left(\Ga|_{S^{d+c}\setminus\{|y|=0\}}\right)\m=\Ga'|_{\R^{d-s+c}\times S^s}.
\end{align*}Direct calculation shows that the map $\Ga'$ restricted to $\R^{d-s+c}\times S^s,$ i.e., where $|y|=1,$ satisfies
\begin{align*}
	y\big(\Ga'|_{\R^{d-s+c}\times S^s}(x,y)\big)\not=&\,0,\\\big|\Ga'|_{\R^{d-s+c}\times S^s}(x,y)\big|=&\,1,\\\Ga|_{S^{d+c}\setminus\{|y|=0\}}\circ\Ga'|_{\R^{d-s+c}\times S^s}(x,y)=&\,(x,y).
\end{align*}
Here $y\big(\Ga'|_{\R^{d-s+c}\times S^s}(x,y)\big)$ means the $y$ coordinate of $\Ga'|_{\R^{d-s+c}\times S^s}(x,y).$ In other words $\Ga'|_{\R^{d-s+c}\times S^s}$ maps $\R^{d-s+c}\times S^s$ to $S^{d+c}\setminus\{|y|=0\}$
and $\Ga'|_{\R^{d-s+c}\times S^s}$ is a right inverse of $\Ga|_{S^{d+c}\setminus\{|y|=0\}}.$ Thus $\Ga'|_{\R^{d-s+c}\times S^s}$ must  be injective and $\Ga|_{S^{d+c}\setminus\{|y|=0\}}$ must be surjective onto $\R^{d-s+c}\times S^s$. 

On the other hand, direct calculations also show that $\Ga$ restricted to $S^{d+c}\setminus\{|y|=0\},$ i.e., where
$|x|^2+|y|^2=1,|y|\not=0,$ is a smooth map that satisfies
\begin{align*}
	\Big|y\big(\Ga|_{S^{d+c}\setminus\{|y|=0\}}(x,y)\big)\Big|=&\,1,\\\Ga'|_{\R^{d-s+c}\times S^s}\circ\Ga|_{S^{d+c}\setminus\{|y|=0\}}(x,y)=&\,(x,y).
\end{align*}
In other words $\Ga$ maps $S^{d+c}\setminus\{|y|=0\}$ to $\R^{d-s+c}\times S^s$ and $\Ga'|_{\R^{d-s+c}\times S^s}$ is also a left inverse of $\Ga|_{S^{d+c}\setminus\{|y|=0\}}.$ Thus $\Ga'|_{\R^{d-s+c}\times S^s}$ must be surjective onto $S^{d+c}\setminus\{|y|=0\}$ and $\Ga|_{S^{d+c}\setminus\{|y|=0\}}$ must be injective.

To sum it up, both $\Ga|_{S^{d+c}\setminus\{|y|=0\}}$ and $\Ga'|_{\R^{d-s+c}\times S^s}$ are smooth bijective maps and they are inverse maps of each other. This finishes the proof of Claim \ref{gabi}.

Now we are ready to find the neighborhood $U(\{0\}^{d-s+c}\times S^s)$ in the statement of Lemma \ref{singssc}. Since $\{0\}^{d-s+c}\times S^s$ is an orbit of the standard Lie group action $$\operatorname{O}(d-s+c)\times\operatorname{O}(s+1),$$ on $\R^{d-s+c}\times \R^{s+1},$ we deduce that the tubular neighborhoods $U_r$ of radius $r$ around $\{0\}^{d-s+c}\times S^s$ in $S^{d+c}$ are also invariant under the action of $O(d-s+c)\times O(s+1)$ as well. Using this, we can calculate that
\begin{align*}
	U_r=\{(x,y)\,|\,|x|^2+|y|^2=1,|x|\le\sin r,(x,y)\in \R^{d-s+c}\times \R^{s+1}\}.
\end{align*}
The same calculations as in the proof of Claim \ref{gabi} show that $U_{\frac{\pi}{4}}$ is mapped by $\Ga|_{S^{d+c}\setminus\{|y|=0\}}$ bijectively to
$		B_{1}^{d-s+c}\times S^s.$ 
\begin{defn}
	Define $$U(\{0\}^{d-s+c}\times S^s)=U_{\frac{\pi}{4}}.$$
\end{defn}
To simplify our notations, in this subsubsection we will still use $U_{\frac{\pi}{4}}$ instead of $U(\{0\}^{d-s+c}\times S^s).$ 
\begin{claim}\label{mapc}
	\begin{itemize}
		\item The map $\Ga|_{S^{d+c}\setminus\{|y|=0\}}$ sends $\si^s(C)|_{ U_{\frac{\pi}{4}}}$ to $C|_{B_1^{d-s+c}}\times S^s.$
		\item The map $\Ga'|_{\R^{d-s+c}\times S^s}$ sends $C|_{B_1^{d-s+c}}\times S^s$ to $\ssc|_{ U_{\frac{\pi}{4}}}.$
	\end{itemize}
\end{claim}
To verify the first bullet of Claim \ref{mapc},  we can parameterize $\ssc|_{ U_{\frac{\pi}{4}}}$ by $(v,y)$ with	\begin{align*}
	&v\in C,\\&|v|\le \sin\frac{\pi}{4},\\&|v|^2+|y|^2= 1.
\end{align*} Thus, on $\ssc|_{ U_{\frac{\pi}{4}}}$, we have
\begin{align*}
	\Ga\big|_{\ssc|_{ U_{\frac{\pi}{4}}}}(v,y)=\bigg(\frac{v}{|y|},\frac{y}{|y|}\bigg)\in C|_{ B_1^{d-s+c}}\times S^s,
\end{align*}where we have used the fact that $\lam v\in C$ for any $\lam\in\R_{\ge 0}.$ This finishes the first bullet of Claim \ref{mapc}

For the second bullet of Claim \ref{mapc}, we can parameterize $C|_{ B_1^{d-s+c}}\times S^s$ by $(v,y)$ with
\begin{align*}
	&v\in C,\\&|v|\le 1,\\&|y|=1.
\end{align*}
Thus, on $C|_{ B_1^{d-s+c}}\times S^s$, we have
\begin{align*}
	&\Ga'\big|_{C|_{ B_1^{d-s+c}\times S^s}}(v,y)=\bigg(\frac{v}{\sqrt{1+|v|^2}},\frac{y}{\sqrt{1+|v|^2}}\bigg)\\\in &(C\times \R^{s+1})\cap  U_{\frac{\pi}{4}}=\ssc|_{ U_{\frac{\pi}{4}}}.
\end{align*}
Again we have used the fact that $\lam v\in C$ for any $\lam \in \R_{\ge0}$. This finishes the second bullet of Claim \ref{mapc}.

Since $\Ga$ and $\Ga'$ are inverses of each other, by Claim \ref{mapc} we deduce that  $\ssc$ restricted to $U(\{0\}^{d-s+c}\times S^s)$ is equal to
\begin{align*}
	C|_{B_1^{d-s+c}}\times S^s,
\end{align*}via $\Ga.$

Now to finish the proof, we need to determine $\sing\ssc.$ Note that $C\times \R^{s+1}$ is invariant under uniform scalings of $\R^{d+c+1},$ i.e., scalar multiples of the identity matrix.  Thus, we deduce that the regular set $\reg (C\times \R^{s+1})$ intersects the unit sphere $S^{d+c}$ in $\R^{d+c+1}$ orthogonally. Orthogonality here gives transversality of $\reg (C\times \R^{s+1})$ with $S^{d+c}$, which implies that 
\begin{align*}
	\big(		\reg (C\times \R^{{s+1}})\big)\cap S^{d+c}\subset\reg\ssc.
\end{align*}
By Definition \ref{defnsm}, this implies that
\begin{align*}
	\big(		\sing (C\times \R^{{s+1}})\big)\cap S^{d+c}\supset\sing\ssc.	
\end{align*}The singular set of $C\times\R^{s+1}$ is $\{0\}^{d-s+c}\times\R^{s+1}.$ Consequently, we deduce that
\begin{align*}
	\sing \ssc\s\big(\{0\}^{d-s+c}\times\R^{s+1}\big)\cap S^{d+c}=\{0\}^{d-s+c}\times S^s.
\end{align*}
Now recall that $\ssc$ restricted to $U_{\frac{\pi}{4}}$  equals  $(C|_{B_1^{d-s+c}})\times S^s$ via $\Ga$.  The singular set of $C|_{B_1^{d-s+c}}\times S^s$ is  $\{0\}^{d-s+c}\times S^s.$ The inverse image of $\{0\}^{d-s+c}\times S^s$, i.e., $\Ga'(\{0\}^{d-s+c}\times S^s)$ is precisely $\{0\}^{d-s+c}\times S^s.$ We are done.
(We emphasize that the two $\{0\}^{d-s+c}\times S^s$ are in different coordinates but are the same set via $\Ga.$)	\end{proof}
	\subsection{Adding the singular sets to make an altered representative}\label{altered}
	Now we are ready to make the altered representative of $[\Si]$ with singular sets. 
	
	First recall that $[\Si]$ can be represented by an embedded connected submanifold $N$ (Assumption \ref{simpa}). 
	\begin{figure}[h]
	\centering
	\def\svgwidth{0.9\paperwidth}
	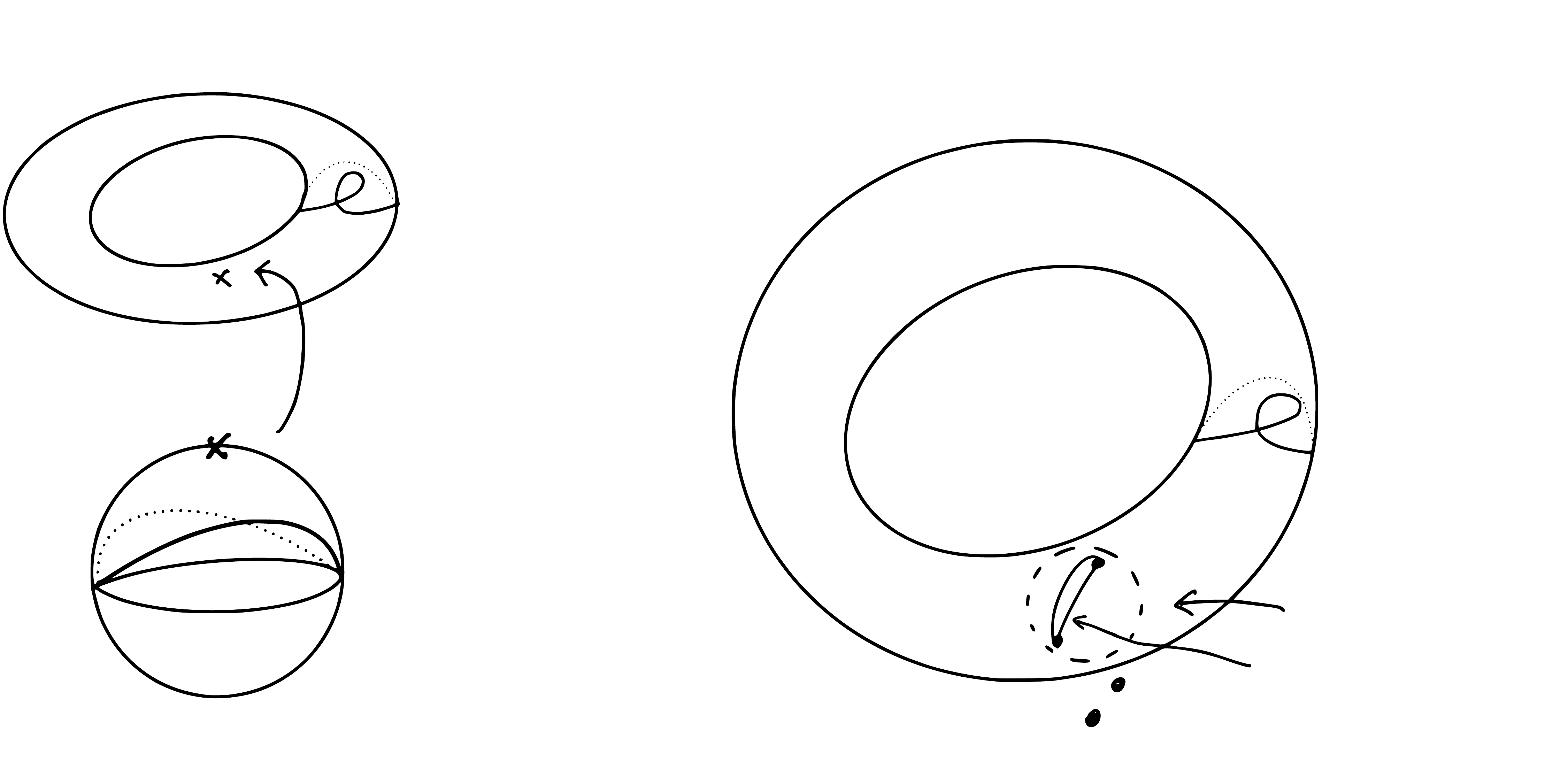
	\caption{Realizing $\ssc$ on $M$}
	\label{fig:ssccycle}
\end{figure}
	
	Now, let us realize $\ssc$ as  a homologically trivial cycle on $M.$ 
	Pick any point $p$ on $S^{d+c}$ with $p\not\in \supp\ssc.$ Then $S^{d+c}\setminus p$ is diffeomorphic to an open $(d+c)$-dimensional ball. Pick any point $q$ on $M,$ with $q\not\in\supp N$ , then we can embed $S^{d+c}\setminus p$ onto an open ball centered at $q$ on $M$, not intersecting $\supp N.$ This way we can regard $\ssc$ as a homologically trivial cycle on $M,$ in the embedding of $S^{d+c}\setminus p.$ 
	
	Now we can use the connected sum defined in Lemma \ref{lemcs} to construct our desired representative $$N\#\ssc$$ of $[\Si]$ with the following properties:
	\begin{assump}\label{assumpns}
		$N\#\ssc$ is a mod $2$ current with {connected} regular set, and 
		\begin{align*}
			N\#\ssc|_{  U(\sing\ssc)}&=\ssc|_{ U(\sing\ssc)},\\
			\sing N\#\ssc&=\sing \ssc.
		\end{align*}
	\end{assump}
	Here $U(\sing\ssc)$ is just the set $U(\{0\}^{d-s+c}\times S^s)$ in Lemma \ref{singssc}. The change of notation is
	for coordinate independent expressions.
%	\section{Deformation map and shorting the metric}	In this section we will construct a deformation map adapted to the cone $C$ that will allow us to compare more general competitors in an altered metric, thus strengtherning the area-minimization property of $C$. This will serve as the replacement of calibrations that is used in Zhang's arguments \cite{YZa,YZj}.	
	\section{Regular neighborhoods of $N\#\si^s(C)$}\label{tbssc}
In the next section, we will need the existence of regular neighborhoods of $N\#\ssc$ of arbitrarily small radii. A regular neighborhood is a classical concept in piecewise linear topology, and basically is the simplicial version of tubular neighborhoods around simplicial subcomplexes. Heuristically speaking, we need such regular neighborhoods to have a well-controlled local geometry around $N\#\ssc,$ which will enable us to make $N\#\ssc$ area-minimizing in a smooth metric.

Finding such arbitrarily small radii regular neighborhoods around smooth submanifolds is a classical fact, i.e., just taking tubular neighborhoods of arbitrarily small radii. Unfortunately, $N\#\ssc$ has fairly complicated singular set, so we need to go through a more complicated route, by realizing $N\#\ssc$ as a subcomplex of a triangulation and then using the classical regular neighborhood theorem from piecewise linear topology. 

\cite[Lemma 5.0.1]{ZLns} states a corresponding result when $N,\ssc$ are integral cycles. The proof relies on simplicial triangulations thus has no dependence on orientability so it also works here. However, for reader's convenience, we reproduce the entire proof.

We plan to achieve our goal of constructing regular neighborhoods in four steps.
\begin{enumerate}
	\item Prove that $\supp N\#\ssc$ is a Whitney stratified set. \label{stepw1}
	\item Whitney stratified sets are triangulable, so $\supp N\# \ssc$ can be realized as a subcomplex of a triangulation of $M.$\label{stepw2}
	\item By the classical simplicial regular neighborhood theorem, the $n$-th barycentric subdivision of the triangulation of $M$  with large enough $n$ will yield regular neighborhoods $U_{\e}$ that deformation retract onto $\supp N\#\ssc$ with $U_\e$ arbitrarily close in distance to $\supp N\#\ssc.$\label{stepw3}
	\item Use Hirsch's smooth regular neighborhood theorem to upgrade the simplicial regular neighborhoods in the previous step to smooth regular neighborhoods.\label{stepw4}
\end{enumerate}
Now equip $M$ with an arbitrary ambient metric $h$. We are ready to state the main lemma in this section.
\begin{lem}\label{tube}
	For any $\e>0,$ there exist  smooth open sets $U_\e(N\#\ssc)$ containing $\supp N\#\ssc$, such that
	\begin{enumerate}
		\item  	 $U_\e(N\#\ssc)$ deformation retracts onto $\supp N\#\ssc$ by a map $\pi_\e$,\label{homue0}
		\item $U_\e(N\#\ssc)$ is arbitrarily close to $\supp N\#\ssc,$ i.e.,\label{homue1}
		\begin{align}\label{distue}
			U_\e(N\#\ssc)\s \{p\,|\,\operatorname{dist}_h(p,\supp(N\# \ssc))<\e\}.
		\end{align}
		\item The $d$-th homology group of $U_\e(N\#\ssc)$ is generated by $N\#\ssc$, i.e.,\label{homue}
		\begin{align}\label{homns}
			H_d(U_\e(N\#\ssc),\Z/2\Z)=\Z/2\Z[N\#\ssc].
		\end{align}
	\end{enumerate}
	Further more for $0<\e''<\e'<\e<\injrad M$, we can choose $\uee,\ue,U_\e(N\#\ssc),$ so that
	\begin{align*}
				\uee\s \ue\s U_\e(N\#\ssc),
	\end{align*}
\end{lem}
This section will be devoted to proving Lemma \ref{tube}.
The proof will be carried out in the order of Step (\ref{stepw1}) and Step (\ref{stepw2}) together, Step (\ref{stepw3}) and Step (\ref{stepw4}) together. Lemma \ref{tube} will be obtained as a direct consequence of the above four steps.	
\subsection{$\supp N\#\ssc$ is a Whitney stratified set and triangulable}
Our plan in this subsection is to stratify $\supp N\#\ssc$ into disjoint unions of smooth submanifolds of different dimensions, and then prove that the stratification is a Whitney stratification. The Whitney stratification allows us to give a triangulation of $M$ with $N\#\ssc$ as a subcomplex. 

An excellent introduction to Whitney's stratification is available in \cite{DTs}, and the classical reference for triangulation of Whitney stratified sets is \cite{MG}. However, since we only use Whitney's stratification to get a triangulation, the reader can just regard Whitney's stratification as a black box. Intuitively, it is clear from the constructions alone that $\supp N\#\ssc$ can be triangulated.

A Whitney stratification of a closed set $K$ is a finite filtration of closed sets $K_0\s K_1\s K_2\s\dots\s K_J=K,$ such that (the possibly empty) $K_j\setminus K_{j-1}$ is a locally finite disjoint union of $j$-dimensional submanifolds, called $j$-th strata of $K,$ and each pair of strata satisfies Whitney's condition (b). Whitney's condition (b) on an ambient manifold $M$	(\cite[Section 1]{DTg}) is defined as follows:
\begin{defn}\label{whitb}
	For disjoint submanifolds $X,Y$ on $M$, let a point $y$ be a point in $Y\cap\ov{X}.$ We say $X$ is (b) regular over $Y$ at $y$, if given sequences $\{x_i\}$ on $X$, $\{y_i\}$ on $Y$, both converging to $y,$ such that the tangent space $T_{x_i}X$ converges to a plane $P$ in a coordinate chart, and the unit vector in the direction of $x_i$ to $y_i$ converges to $v,$ then $v\in P.$
\end{defn}  
The above definition of Whitney's stratification is taken from \cite[Section 2]{MG}. In some literature, authors also require strata in a Whitney stratification to satisfy Whitney's condition (a), which is long known to be implied by Whitney's condition (b), e.g., \cite[p. 454 last paragraph]{DTg}. 
\subsubsection{Strata of $\supp N\#\ssc$}
First we want to define a filtration $$K_0\subset K_1\subset K_2\subset\dots\subset K_d$$ for $\supp N\#\ssc$ with $K_j\setminus K_{j-1}$ a $j$-dimensional submanifold.

To do this, let us discuss the singularity structure of $\supp N\#\ssc.$ By Assumption \ref{assumpns}, $\sing\left( N\#\ssc\right)=\sing\ssc$.

Recall Lemma \ref{singssc}, $\sing\ssc$ is an embedded standard $s$-dimensional sphere $S^s$ with $s\ge 0.$ Restricted to a neighborhood $U(\sing\ssc)$ around $\sing\ssc$, $\ssc$ equals $C|_{B_1^{d-s+c}}\times S^s$ via a diffeomorphism $\Ga.$  Thus,  in our construction, we want to make sure $\sing\ssc$ is contained in  $$K_s\s K_{s+1}\s\dots\s K_d$$  and contributes to only the strata $K_s\setminus K_{s-1},$ where we regard $K_{-1}$ as an empty set.

Now we are ready to describe the filtration $K_0\s\dots\s K_d$ of $\supp N\#\ssc.$ By the above discussion, we have a decomposition of $\supp N\#\ssc$ into disjoint submanifolds of different dimensions as follows:
\begin{align}\label{sscdec}
	\supp N\#\ssc=\reg \left(N\#\ssc\right)\cup \sing\ssc.
\end{align}
If $s=0$, set $K_0=\sing\ssc.$ Otherwise, set $K_0=K_{-1}=\es.$ Suppose we have constructed a $K_k.$ If a factor in the right hand side of equation (\ref{sscdec}) has dimension $(k+1)$, then define $K_{k+1}$ as the union of $K_k$ with all the factors in the right hand side of (\ref{sscdec}) that has dimension $(k+1).$ Otherwise, set $K_{k+1}=K_k.$ Thus, inductively, we can construct a filtration 
\begin{align}\label{filterk}
	K_0\s\dots\s K_d
\end{align} of $\supp N\#\ssc,$ so that $K_j\setminus K_{j-1}$ is a (possibly empty)  smooth submanifold of dimension $j.$

In other words, when $s=0,$ we have
\begin{align*}
	\sing\ssc=K_0=\cdots=K_{d-1}\s N\#\ssc=K_d,
\end{align*}when $s\ge 1,$
we have
\begin{align*}
	\es=K_0=\cdots=K_{s-1}\s \sing\ssc=K_s=K_{s+1}\cdots\s =K_{d-1}\s N\#\ssc=K_d,
\end{align*}
\subsubsection{Verifying Whitney's condition (b)}
Now we have to verify that each pair of strata satisfy Whitney's condition (b) (Definition \ref{whitb}). 

First let us decide which strata $K_j\setminus K_{j-1}$ can limit to another strata $K_l\setminus K_{l-1}$.

Recall that all strata are the factors of the same dimension in the right hand side of (\ref{sscdec}). Let us consider factor by factor.

Since the regular set $\reg \left(N\#\ssc\right)$ is dense, $\reg \left(N\#\ssc\right)$ can limit to any other strata. Let $\{x_j\}$ be a sequence of points in $\reg \left(N\#\ssc\right)$ that converges to a point $x.$ 

If $x\in \sing\ssc,$ then let $\{y_i\}$ be a sequence of points in $\sing\ssc$ that also converges to $x.$ By Lemma \ref{singssc} and Assumption \ref{assumpns}, restricted to a neighborhood of $\sing\left( N\#\ssc\right)$, $\ssc$ equals $C|_{B_1^{d-s+c}}\times S^s\s B_1^{d-s+c}\times S^s.$ Thus, by choosing a coordinate system on $S^s,$ we can assume that there is a coordinate chart on $M$ containing $x,$ such that $N\#\ssc$ restricted to this chart equals\begin{align*}
	C\times \R^s\s \R^{d-s+c}\times \R^s
\end{align*}
Then the $\{y_j\}$ lies in $\{0\}^{d-s+c}\times \R^s$, and $\{x_j\}$ lies in $C\times \R^s\setminus(\{0\}^{d-s+c}\times \R^s).$

We claim that line segment $\overrightarrow{x_iy_i}$ lies entirely on $C\times \R^s,$ so by definition of tangent space as the set of tangent vectors to curves, $\overrightarrow{x_iy_i}\in T_{x_i}\reg \left(N\#\ssc\right)$. Thus, Whitney's condition (b) is satisfied for the strata $\reg \left(N\#\ssc\right)$ over $\sing\ssc$. To see this, note that any point $q$ on $\overrightarrow{x_iy_i}$ equals\begin{align*}
	q=t\overrightarrow{0x_i}+(1-t)\overrightarrow{0y_i},
\end{align*}for $t\in[0,1]$. Since $C\times \R^s$ is invariant under uniform scalings, $tx_i$ also lies in $C\times \R^s$. Since $C\times \R^s$ is invariant under translations generated by vectors in $\{0\}^{d-s+c}\times\R^s$, and since $(1-t)y_i\in \{0\}^{d-s+c}\times\R^s$, we see that $q\in C\times \R^s$ as claimed.

The strata $\sing\ssc$ is closed and thus not limiting to any other strata. 

To sum it up, $\supp N\#\ssc$ is a Whitney's stratified set.
\subsubsection{Realizing $N\#\ssc$ as a subcomplex of a triangulation of $M$}Since we want to use the fact that Whitney's stratified set can be triangulated, and since we want to realize $N\#\ssc$ as a subcomplex of $M,$ we need to include $M$ into the Whitney stratification as well. To do this, set $$ N\#\ssc=K_d=K_{d+1}=\dots=K_{d+c-1},K_{d+c}=M.$$ Since we only added the ambient manifold $M$ to the filtration (\ref{filterk}), it is straightforward to verify that
\begin{align*}
	K_0\s K_1\s\cdots\s K_{d+c},
\end{align*} gives a Whitney stratification of $M.$

It is a classical fact that Whitney stratified sets admit triangulations. By \cite[Theorem p. 196]{MG}, $M$ admits a finite triangulation into a simplicial complex, such that $N\#\ssc$ is a subcomplex. \cite[Theorem p. 196]{MG} is stated for Thom-Mather stratified sets, of which Whitney's stratification is a special case as mentioned in \cite[Section 1, first paragraph]{MG}. To sum it up, we have proven that
\begin{consq}\label{subcomplex}
	There is a finite triangulation of $M$ into a simplicial complex, such	$N\#\ssc$ is a subcomplex, and $\sing\left( N\#\ssc\right)$ is a subcomplex of $N\#\ssc$.
\end{consq}
We have finished Step (\ref{stepw1}) and (\ref{stepw2}) of our plan of this section. Now we are ready to deal with the next two steps.
\subsection{Smooth regular neighborhoods of $N\#\ssc$ of arbitrarily small radii}
In this section we will prove the first two bullets of Lemma \ref{tube}, i.e., constructing regular neighborhoods of arbitrarily small radii.
\subsubsection{Simplicial regular neighborhoods of arbitrarily small radii}
We need the following classical regular neighborhood theorem from piecewise linear topology:
\begin{fact}\label{regnthm}
	(\cite[Theorem 2.11]{JH})	The derived neighborhoods $O_n$ of a subcomplex $K$ in the $n$-th barycentric subdivision are regular neighborhoods of $K$, provided $n\ge 2,$
	and each $O_n$ deformation retracts onto $K.$ \end{fact}
Here the derived neighborhood means the union of all simplices that intersect our subcomplex in the $n$-th barycentric subdivision, or in simplicial jargon, union of all stars of $K$. The deformation retract of $O_n$ onto $K$ can be constructed by collapsing each simplex intersecting our subcomplex. 

We also need another classical fact about diameters of simplices in barycentric subdivisions:
\begin{fact}\label{radbd}
	The diameter of $j$-dimensional simplices resulting from the $n$-th barycentric division is at most $\big(\frac{j}{j+1}\big)^n$ times the diameter of the original $j$-dimensional simplex.
\end{fact}
This is proved in \cite[p. 120]{AH}.

By Consequence \ref{subcomplex}, $N\#\ssc$ can be realized as a subcomplex of a finite triangulation of $M.$ By Fact \ref{regnthm}, the $n$-th derived neighborhood $O_n$ of the simplicial realization of $N\#\ssc$ in $M$ deformation retracts onto $N\#\ssc.$ By  Fact \ref{radbd}, for any $\e>0,$ there exists $n\in\N$ large enough, such that
\begin{align}\label{simur}
	O_n\s	\{p\,|\,\operatorname{dist}(p,\supp N\#\ssc)<\e\}.
\end{align}
Though Fact \ref{radbd} is stated only for simplices in Euclidean space, we only have finitely many simplices in $M,$ and each simplex in $M$ is bi-Lipschitz equivalent to a simplex in $\R^{d+c}$. Thus, we can still conclude that the diameter on $M$ of simplices in the $n$-th barycentric subdivision will converge to $0$ as $n\to\infty.$
\subsubsection{Upgrade to smooth regular neighborhoods}
We have obtained simplicial regular neighborhoods of $N\#\ssc$ in (\ref{simur}). Now we want to upgrade it into a smooth one. This is precisely the conclusion in \cite[conlusion (1a') p. 525]{MHr}:
\begin{fact}
	For a simplicial regular neighborhood
	$L$ of a subcomplex $K,$ and any open set $U$ containing $L,$ there exists a smooth regular neighborhood of $K$ in $U$ that deformation retracts onto $K$.\end{fact}
Now applying the above fact to \begin{align*}
	&K=N\#\ssc,\\& L=O_n,\\&U=\{p|\operatorname{dist}(p,\supp N\#\ssc)<\e\},
\end{align*} we arrive at a smooth open set $U_\e(N\#\ssc)$ containing $\supp N\#\ssc$ that deformation retracts onto $\supp N\#\ssc$ by a map $\pi_\e$ and 
\begin{align*}
	U_\e(N\#\ssc)\s \{p|\operatorname{dist}(p,\supp N\#\ssc)<\e\}.
\end{align*}This finishes the proof of the first two bullets in Lemma \ref{tube}.
\subsection{Homology of $U_\e(N\#\ssc)$}
Now we are left to prove the last bullet in Lemma \ref{tube}. 
Since $U_\e(N\#\ssc)$ deformation retracts onto $\supp N\#\ssc,$ to show that $$H_d(U_\e(N\#\ssc),\Z/2\Z)=\Z/2\Z[N\#\ssc],$$ it suffices to prove that
\begin{align*}
	H_d(\supp N\#\ssc,\Z/2\Z)=\Z/2\Z[N\#\ssc].
\end{align*}
Recall the constancy theorem Fact \ref{const}. Let $W$ be a simplicial $\Z/2\Z$ cycle supported on $\supp N\#\ssc$. Then $W$ induces a mod $2$ current supported on $\supp N\#\ssc$. On the connected open manifold $\reg \left(N\#\ssc\right)$, we can use the Fact \ref{const} to deduce that $W$ restricted to $\reg \left(N\#\ssc\right)$ is a $\Z/2\Z$ multiple of $N\#\ssc$ restricted to $\reg \left(N\#\ssc\right)$ say $k\in\Z/2\Z.$ Thus, $W-kN\#\ssc$ is supported on a $s$-dimensional set $N\#\ssc$, by Fact \ref{fctconc}, we deduce that $W=k(N\#\ssc)$. This shows that $H_d( N\#\ssc,\Z/2\Z)$ is generated by $[N\#\ssc].$ 

Since $[N\#\ssc]$ is not $0$ as a homology class, we are done.

For choice of $\uee,\ue,U_\e(N\#\ssc),$ first make an arbitrary choice of $U_{\e}(N\#\ssc).$ Note that $\pd U_{\e}(N\#\ssc)$ is of positive distance $\ov{\e}$ away from $N\#\ssc$, i.e., points on $\pd U_{\e}(N\#\ssc)$ and $N\#\ssc$ are of distance at least $\ov{\e}>0.$ Then we choose one $U_{\ov{\e}/2}(N\#\ssc)$ to be $\ue$. Similarly $\pd \ue$ is of positive distance $\ov{\e'}$ away from $N\#\ssc$. Then we choose one $U_{\ov{\e'}/2}(N\#\ssc)$ to be $\uee$. 
	\section{Making the altered representative area-minimizing}\label{zhang}
	This section will be devoted to the proof of the following lemma, 
		\begin{lem}\label{nsmin}
			There is a smooth metric $h$ on $M,$ such that $ N\#\ssc$ is the unique area-minimizing representative of $[\Si]$ with respect to $h$.   Furthermore, we have
		\begin{itemize}
			\item 
			In a neighborhood $U_\rr(\sing\ssc)$ of $\sing \ssc$ diffeomorphic to $B_\rr^{d-s+c}\times S^s,$ via the diffeomorphism $\Ga$ defined in Lemma \ref{singssc},  the metric $h$ is equal to the pullback under $\Ga$ of the standard metric on $B_{\rr}^{d-s+c}\times S^s,$ and $N\#\ssc$ restricted to $U_\rr(\sing \ssc)$ equals to $C|_{B_\rr^{d-s+c}}\times S^s$ via $\Ga.$
		\end{itemize}
	\end{lem} 
	Here $B_\rr^{d-s+c}\times S^s$ is the product of a radius $\rr$ ball $B_\rr^{d-s+c}$ inside $\R^{d-s+c}$ and the standard $s$-dimensional sphere $S^s$ with $\rr>0$, and
	\begin{align*}
		U_\rr(\sing\ssc)\overset{\Ga}{\cong}B_\rr^{d-s+c}\times S^s.
	\end{align*}
	
	The proof relies on a new gluing scheme inspired by Zhang's gluing of calibrations in  \cite{YZa,YZj,YZt}. 
	To the author's knowledge this is the gluing method for area-minimizing currents that does not utilize calibrations.
	
	Roughly speaking, the construction of the metric $h$ involves five steps:
	\begin{enumerate}
		\item The way we construct $N\#\ssc$ provides a Riemannian metric $h_\sing$ so that $N\#\ssc$ is mod $2$ area-minimizing when restricted to $ U(\sing\ssc).$ Then we handcraft a Riemannian metric $h_\reg$ that makes the regular set of $N\#\ssc$ area-minimizing\label{z1}
		\item Gluing  the metrics $h_\reg$ to $h_\sing$ to construct a smooth metric $\hgd$, in which $N\#\ssc$ is area-minimizing mod $2$ in a smooth regular neighborhood $U_\e(N\#\ssc)$ constructed in Lemma \ref{tube}.\label{z2}% This implies that $N\#\ssc$ is area-minimizing in $U_\e(N\#\ssc)$ in the class $[N\#\ssc]\in H_d(U_\e(N\#\ssc),\Z)$.
		\item Multiplying the metric by a bump function to make $N\#\ssc$ the unique area-minimizing representative of $[N\#\ssc]$ in $U_\e(N\#\ssc)$. \label{z3}
		\item Making the metric in the complement of $U_\e(N\#\ssc)$ very large, so that any area-minimizing mod $2$ current $T$  in $[\Si]$ will stay in $U_\e(N\#\ssc).$\label{z4}
		\item Though homologous objects restricted to subsets of ambient space are not necessarily homologous, the equation (\ref{homns})  implies that $[T]= [N\#\ssc]$. Then Step (\ref{z3}) forces $T= N\#\ssc$ provided $T$ is area-minimizing in $[\Si].$
		\label{z5}
	\end{enumerate}
	After doing some basic preparational works, our proof will be carried in the order of the above five steps.
	
	First, let us make an assumption that we will use throughout this section.
	\begin{assump}\label{assumpc}(IN THIS SECTION ONLY)
		Assume $C$ is a mod $2$ area-minimizing  cone in $\R^{d-s+c}$ and $C_1$ is $C$ restricted to the unit ball of $\R^{d-s+c}.$
	\end{assump}
	\subsection{A metric shorting lemma}
	\begin{lem}\label{lemshort}
		Let $f:L\to L'$ a smooth map between two  Riemannian manifolds $(L,g)$ and $(L',g').$ Let $T$ be a $d$-dimensional mod $2$ current on $L.$ If 
		\begin{align*}
			f\du g'\le g,
		\end{align*}then we have
		\begin{align*}
			\ms_g(T)\ge \ms_{g'}(f(T)).
		\end{align*}
	\end{lem}
	\begin{proof}
		First, recall Fact \ref{fctmap},
		\begin{align*}
			\ms_{g'}(f(T))\le \int_T\jac(f)|_{\T_pT} \dhm^d_{g}(p)
		\end{align*}
		On the other hand, by assumption we have
		\begin{align*}
		\jac(f)|_{\T_pT}=\sqrt{g'(df(\T_pT),df(\T_pT))}=\sqrt{f\du g'(\T_pT,\T_pT)}\le\sqrt{g(\T_pT,\T_pT)}=1.
		\end{align*} We are done.
	\end{proof}
	\subsection{Preparing the neighborhoods}
	First of all, we need to prepare the neighborhoods $U_\e(N\#\ssc)$ and $U(\sing\ssc)$ for our constructions.
	
	In order to make space for gluing  in later subsections, we need to consider more refined neighborhoods $U_r(\sing\ssc).$
	\begin{defn}
		Define $$U_r(\sing\ssc)\overset{\Ga}{\cong}B_r^{d-s+c}\times S^s.$$
	\end{defn}
	\begin{figure}[h]
	\centering
	\def\svgwidth{0.9\paperwidth}
	%% Creator: Inkscape 1.3.2 (091e20e, 2023-11-25, custom), www.inkscape.org
%% Accompanies image file 'neighborhoods.pdf' (pdf, eps, ps)
\begingroup%
  \makeatletter%
  \providecommand\color[2][]{%
    \errmessage{(Inkscape) Color is used for the text in Inkscape, but the package 'color.sty' is not loaded}%
    \renewcommand\color[2][]{}%
  }%
  \providecommand\transparent[1]{%
    \errmessage{(Inkscape) Transparency is used (non-zero) for the text in Inkscape, but the package 'transparent.sty' is not loaded}%
    \renewcommand\transparent[1]{}%
  }%
  \providecommand\rotatebox[2]{#2}%
  \newcommand*\fsize{\dimexpr\f@size pt\relax}%
  \newcommand*\lineheight[1]{\fontsize{\fsize}{#1\fsize}\selectfont}%
  \ifx\svgwidth\undefined%
    \setlength{\unitlength}{2160bp}%
    \ifx\svgscale\undefined%
      \relax%
    \else%
      \setlength{\unitlength}{\unitlength * \real{\svgscale}}%
    \fi%
  \else%
    \setlength{\unitlength}{\svgwidth}%
  \fi%
  \global\let\svgwidth\undefined%
  \global\let\svgscale\undefined%
  \makeatother%
  \begin{picture}(1,0.5)%
    \lineheight{1}%
    \setlength\tabcolsep{0pt}%
    \put(0,0){\includegraphics[width=\unitlength,page=1]{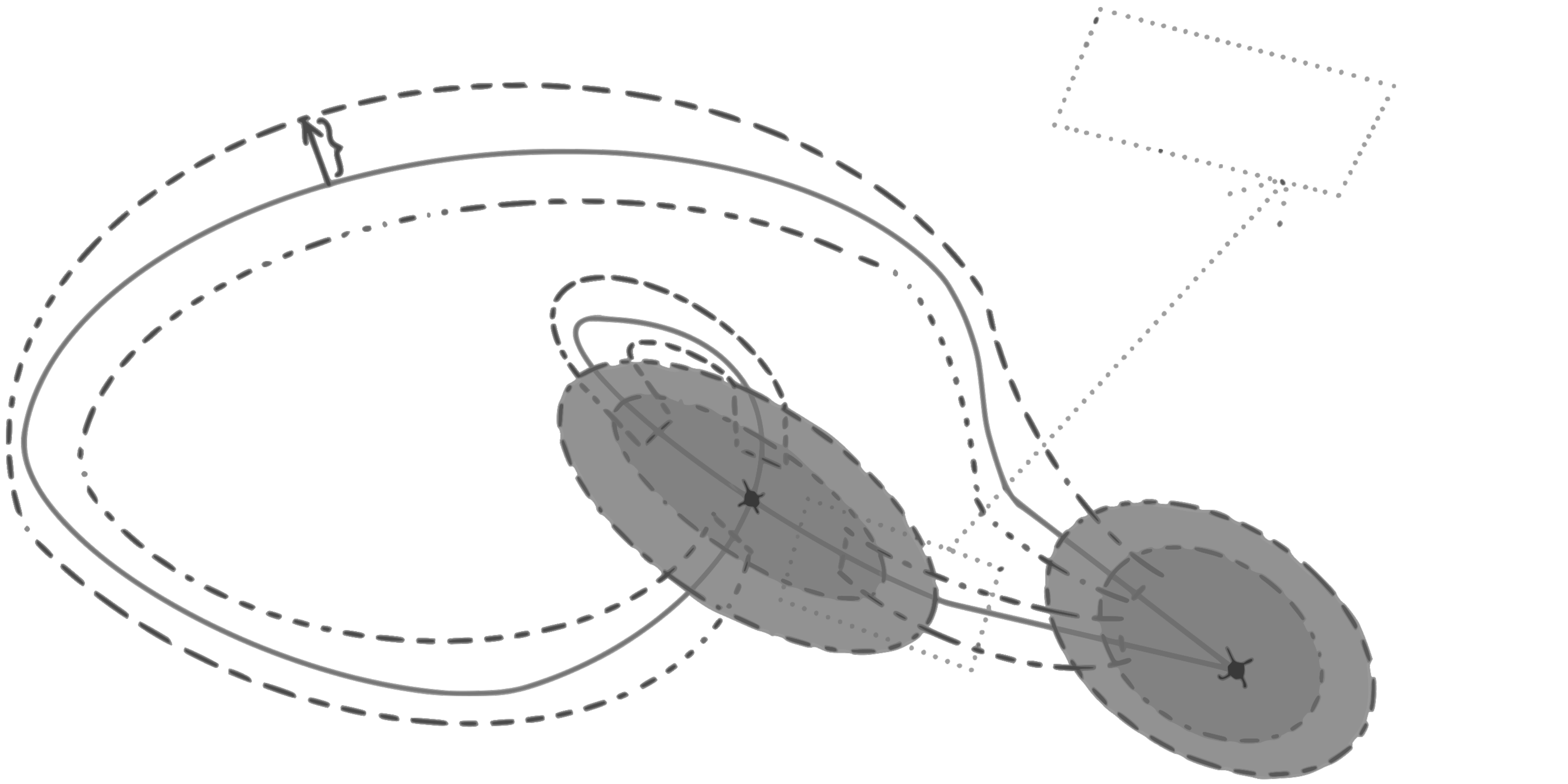}}%
    \put(0.22281534,0.40549558){\color[rgb]{0,0,0}\makebox(0,0)[lt]{\lineheight{1.25}\smash{\begin{tabular}[t]{l}$\nu$\end{tabular}}}}%
    \put(0.19838288,0.46729538){\color[rgb]{0,0,0}\makebox(0,0)[lt]{\lineheight{1.25}\smash{\begin{tabular}[t]{l}$B_\nu\left(\reg^0 N\#\ssc\right)$\end{tabular}}}}%
    \put(0.26449458,0.07062615){\color[rgb]{0,0,0}\makebox(0,0)[lt]{\lineheight{1.25}\smash{\begin{tabular}[t]{l}$N\#\ssc$\end{tabular}}}}%
    \put(0.71577794,0.18560241){\color[rgb]{0,0,0}\makebox(0,0)[lt]{\lineheight{1.25}\smash{\begin{tabular}[t]{l}$\s U\left(\sing N\#\ssc\right)$\end{tabular}}}}%
    \put(0.715778,0.43280237){\color[rgb]{0,0,0}\makebox(0,0)[lt]{\lineheight{1.25}\smash{\begin{tabular}[t]{l}\end{tabular}}}}%
    \put(0.69560911,0.21974426){\color[rgb]{0,0,0}\makebox(0,0)[lt]{\lineheight{1.25}\smash{\begin{tabular}[t]{l} \end{tabular}}}}%
    \put(0.73441374,0.21255822){\color[rgb]{0,0,0}\makebox(0,0)[lt]{\lineheight{1.25}\smash{\begin{tabular}[t]{l}$ {U}_{\frac{1}{10}}\left(\sing N\#\ssc\right)$\end{tabular}}}}%
  \end{picture}%
\endgroup%

	\caption{Schematic illustration of the neighborhoods\\We emphasize that $\reg \left(N\#\ssc\right)$ should be connected! Impossible to draw on a $2$-d plane.}
	\label{fignbhd}
\end{figure}
	The reader should now consult Figure \ref{fignbhd} for intuition.
	\subsection{Making $N\#\ssc$ mod $2$ area-minimizing when restricted to $  U(\sing \ssc)$}
	In this subsection, we will show how from our construction of $N\#\ssc$, we naturally get a smooth metric $h_\sing$ defined throughout $M$ such that $N\#\ssc$ is mod $2$ area-minimizing when restricted to $  U(\sing \ssc).$
		
	We need the following classical fact.
	\begin{fact}\label{calc}
		The mod $2$ current $C_1\times S^s$ is area-minimizing mod $2$ in $\R^{d-s+c}\times S^s$ with the product Riemannian metric of standard metrics on the two factors.
	\end{fact}
	\begin{proof}
Let $T$ be a mod $2$ current so that $T-C_1\times S^s=\pd Q$ for a compactly supported $(d+1)$-dimensional mod $2$ current $Q$. Use $\pi_{S^s}$ to denote the projection of $\R^{d-s+c}\times S^s$ and for $q\in S^s,$ use $T\cap \pi_{S^s}\m(q)$ to denote the slicing of $T$ by $\pi_S^s$ at $q.$ Note that the Riemannian product metric gives $\jac(\pi_{S^s}|_W)\le 1$ for any $d$-dimensional linear subspaces $W$ in the tangent spaces to $\R^{d-s+c}\times S^s.$ Thus, by the coarea formula coarea formula \cite[equation 10.3]{LS}, we have
\begin{align}\label{prod1}
	\ms(T)=\int_{S^s}\int_{T\cap\pi_{S^s}\m(q)}\frac{1}{\jac\left(\pi_{S^s}|_{\T_pT\cap\pi_{S^s}\m(q)}\right)}\dhm^{d-s}(p)\dhm^s(q)\ge \int_{S^s}\ms(T\cap\pi_{S^s}\m(q))\dhm^s(q).
\end{align}
On the other hand
\begin{align}
	(\pd Q)\cap\pi_{S^s}\m(q)=T\cap\pi_{S^s}\m(q)-(C_1\times S^s)\cap\pi_{S^s}\m(q)=T\cap\pi_{S^s}\m(q)-C_1.
\end{align}
As $C_1$ is mod $2$ area-minimizing, we deduce that
\begin{align}
	\ms(T\cap\pi_{S^s}\m(q))\ge\ms(C_1).
\end{align}
Riemannian product structure gives
\begin{align}\label{prod3}
	\int_{S^s}\ms(C_1)\dhm^s(q)=\ms(C_1\times S^1).
\end{align}
Combining (\ref{prod1}) through (\ref{prod3}), we are done.
\end{proof}	
	Now, recall Assumption \ref{assumpc} and Lemma \ref{singssc}. In $U(\sing\ssc)$ adopt the pullback of the standard product metric on $\R^{d-s+c}\times S^s$ under $\Ga$ and then extend this product metric on $U(\sing \ssc)$ to a smooth Riemannian metric $h_\sing$ defined throughout $M.$ 
	\begin{fact}\label{calsing}
		We have a smooth Riemannian metric $h_\sing$ on $M$ so that
		\begin{itemize}
			\item $h_\sing$ restricted to $U(\sing\ssc)$ is the pullback of the standard product metric on $\R^{d-s+c}\times S^s$ under $\Ga$,
			\item $N\#\ssc$ is mod $2$ area-minimizing in $h_\sing$ when restricted to $U(\sing\ssc)$
		\end{itemize}  
	\end{fact}
	\subsection{Metric near regular set of $N\#\ssc$}
	In this subsection, we want to make the regular set of $N\#\ssc$ in a smooth metric $h_\reg$. A natural idea is to take advantage of the exponential map of the normal bundle of $\reg N\#\ssc,$ since a classical fact is that $\reg N\#\ssc$ is area-minimizing in its normal bundle.  Unfortunately, $\reg N\#\ssc$ is an open manifold with limit points in $  \sing \ssc.$ Thus, we have to go through a more complicated route by restricting to a compact subset of $\reg N\#\ssc$ and taking into account that we have to make space for gluing of calibrations in the next subsection.
	
	Let us first discuss what compact subsets of $\reg N\#\ssc$ we want to deal with. Write
	\begin{align*}
		\reg^0(N\#\ssc)=(\reg N\#\ssc )\cap   U_{\frac{1}{10}}  (\sing \ssc)\cp.
	\end{align*}
	By construction, $\reg^0(N\#\ssc)$ is a smooth submanifold with boundary.
	
	Let $\exp^\perp$ be the exponential map from the normal bundle $\T^\perp\reg N\#\ssc$ of $\reg N\#\ssc$ to $M$ in the base metric $h_\sing.$ Here $\T^\perp$ means the normal bundle of a submanifold with respect to the ambient manifold $M.$ Without loss of generality, we can assume that 
	\begin{assump}\label{assumpexp}
		$\exp^\perp$ is a diffeomorphism when restricted to the following ball bundle \begin{align*}
			B^\perp_\nu \reg^0(N\#\ssc)=	\{(p,v)|p\in \reg^0(N\#\ssc),\operatorname{ }v\in T_p^\perp \reg N\#\ssc,\operatorname{ }h_\sing(v,v)\le \nu^2\},
		\end{align*} over the base $\reg^0(N\#\ssc)$, with $0<\nu<\frac{1}{10}.$
	\end{assump}
	The image of $B^\perp_\nu \reg^0(N\#\ssc)$ under $\exp^\perp$ is a set of points with distance at most $B_\nu^\perp$ to $\reg^0(N\#\ssc)$. Write
	\begin{align*}
		B_\nu(\reg^0(N\#\ssc))=\exp^\perp B^\perp_\nu \reg^0(N\#\ssc).
	\end{align*}
	Note that $B_\nu(\reg^0(N\#\ssc))$ is a manifold with boundary and corner as $\reg^0(N\#\ssc)$ has non-empty boundary.
	By (\ref{distue}), choosing $\e$ small, we can assume that
	\begin{align}\label{inc}
		  U_{\frac{1}{10}}  (\sing \ssc)\cp\cap U_\e(N\#\ssc)\s 	B_\nu(\reg^0(N\#\ssc)).
	\end{align}
	See Figure \ref{fignbhd} to get an intuition. This seemingly odd (\ref{inc}) is to prepare for the gluing of metrics in the next subsection.
	
	Let ${\pi}^0$ denote the projection onto base $\reg^0 (N\#\ssc)$ in the ball bundle $B_\nu^\perp\reg^0(N\#\ssc).$ Then on $B_\nu(\reg^0(N\#\ssc))$, define
	\begin{align*}
		{\pi}=\pi^0\circ(\exp^\perp)\m.
	\end{align*}
	In other words ${\pi}$ is the nearest distance projection of $B_\nu(\reg^0(N\#\ssc))$ onto $\reg^0(N\#\ssc).$ The projection ${\pi}$ provides a smooth orthogonal splitting of the tangent spaces to $M$ on the ball bundle $B_\nu(\reg^0(N\#\ssc))$ into
	\begin{align}\label{split}
		\ker \ed\pi\oplus \ker^\perp \ed\pi.
	\end{align}
	Here $\oplus$ is the Riemannian direct sum of vector spaces, and $\ker^\perp \ed\pi$ is the orthogonal complement of $\ker \ed\pi.$ By definition of $\pi,$ $\ker\ed\pi$ are tangent spaces to the level sets of $\pi,$ i.e., the submanifolds formed by geodesics orthogonal to $\reg N\#\ssc.$ Thus, $\kn$ restricted to $\reg N\#\ssc$
	is the tangent space to $\reg N\#\ssc$. Heuristically, one should think of $\kn$ as the natural extension of the tangent space of $\reg N\#\ssc$ into the ball bundle.

	This also provides a natural smooth splitting of the Riemannian metric $h_\sing$ into
	\begin{align}\label{hsingdec}
		h_\sing=(h_\sing)|_{\ker \ed\pi}\oplus(h_\sing)|_{\ker^\perp \ed\pi}.
	\end{align}Here $\oplus$ means the Riemannian sum of the metrics on orthogonal subspaces.

\begin{defn}\label{defnreg}
	In the ball bundle $B_\nu(\reg^0(N\#\ssc))$ define a  metric
$$h_\reg=(h_\sing)|_{\ker \ed\pi}+\pi\du h_\sing.$$	
\end{defn}
Note that by Assumption \ref{assumpexp}, the metric $h_\reg$ is a well-defined smooth Riemannian metric. By construction we have
\begin{fact}
	$h_\reg$ restricted to tangent spaces to $ N\#\ssc$ equals $h_\sing$.
\end{fact}%Furthermore, we have	\begin{fact}\label{calreg}	In $B_\nu(\reg^0(N\#\ssc))$ equipped with $h_\reg$, $ N\#\ssc$ restricted to	$B_\nu(\reg^0(N\#\ssc))$ is mod $2$ area-minimizing.	\end{fact}\begin{proof}	We have $h_\reg\ge \pi\du h_\sing.$ For any mod $2$ current $T$ that has the same boundary as $ N\#\ssc$ restricted to	$B_\nu(\reg^0(N\#\ssc))$, apply Lemma \ref{lemshort}:	\begin{align*}		\ms_{h_\reg}(T)\ge \ms_{h_\sing}(\pi(T))=\ms_{h_\reg}(\pi(T))=\ms_{h_\reg}(N\cap B_\nu(\reg^0(N\#\ssc))),\end{align*}where we have used constancy theorem Fact \ref{const} at the last step.\end{proof}	
\subsection{Making $N\#\ssc$ mod $2$ area-minimizing in $U_\e(N\#\ssc)$ by gluing metrics}
	Now we have come to the meat of the matter. We will glue the metrics we obtained so far to make $N\#\ssc$ mod $2$ area-minimizing in $U_\e(N\#\ssc).$
	
	By equation (\ref{inc}), Fact \ref{calsing} and Definition \ref{defnreg}, in the region
	\begin{align*}
		\big(  U(\sing \ssc)\setminus  U_{\frac{1}{10}}  (\sing \ssc)\big)\cap U_\e(N\#\ssc)\s 	B_\nu(\reg^0(N\#\ssc)),
	\end{align*}
	we have two Riemannian metrics $h_\reg,h_\sing$. Our goal is to glue them together and make $N\#\ssc$ is mod $2$ area-minimizing in $U_\e(N\#\ssc)$
	\subsubsection{Notation convention}
	From now on, we will restrict our attention to the region $\big(  U(\sing \ssc)\setminus  U_{\frac{1}{10}}  (\sing \ssc)\big)\cap U_\e(N\#\ssc)$. We adopt the notation:
	\begin{align*}
		U_{\operatorname{gluing}}=\big(  U(\sing \ssc)\setminus  U_{\frac{1}{10}}  (\sing \ssc)\big)\cap U_\e(N\#\ssc).
	\end{align*}
	Recall that $h_\sing$ is defined throughout $M$ and $h_\reg$ is obtained by modifying $h_\sing$ (Definition \ref{defnreg}). 
	
	Thus, from now on, we will use $h_\sing$ as our base metric.
	\subsubsection{The squeezing map $f$}\label{secsq}
	Let us first define a distance function that will serve as the scale of gluing. Recall equation (\ref{inc}).
	\begin{defn}\label{defnrh}
		For $p\in \ug$, define
		\begin{align*}
			\rh(p)=\operatorname{dist}_{(N\#\ssc,h_\sing)}(\pi(p),\sing\ssc).
		\end{align*}
	\end{defn}
Recall Fact \ref{calsing}. The Riemannian product structure of $h_\sing$ in $U(\sing\ssc)$ implies that $\rh(p)$ is a smooth map on $\ug$. Note that $\rh$ is always strictly larger than $0$ on $\ug.$
\begin{defn}\label{defnai}
	Let $\ai$ be a smooth monotonically decreasing function from $[0,1]\to[0,1]$, so that
	\begin{align*}
		\ai=\begin{cases}
			1,&\textnormal{ on }[0,\frac{1}{3}]\\
			\textnormal{strictly monotonically decreasing from }1\textnormal{ to }0,&\textnormal{ on }[\frac{1}{3},\frac{2}{3}]\\
			0,&\textnormal{ on }[\frac{2}{3},1]
		\end{cases}
	\end{align*}
\end{defn}
\begin{defn}
Define a map $f:\ug\to\ug$ as follows. Set $f(p)$ to be the unique point of distance $\ai(\rh(p))\operatorname{dist}_{h_\sing}(p,\pi(p))$ away from $\pi(p)$ on the geodesic from $p$ to $\pi(p)$
\end{defn}
Note that $f(p)$ is well-defined by (\ref{inc}).
\begin{fact}\label{fctf}
	The map $f$ is a smooth map on $\ug$ and 
	\begin{itemize}
		\item $f$ is identity on $\rh\m[0,\frac{1}{3}]$,
\item $f=\pi$ on $\rh\m[\frac{2}{3},1]$,
\item $f$ is identity when restricted to $N\#\ssc.$
\item on $\rh\m(\frac{1}{3},\frac{2}{3}),$ we have
\begin{align*}
	f\du h_\sing\ge \frac{1}{3}\pi\du h_\sing
\end{align*} if we take $\e$ in $U_\e(N\#\ssc)$ small.
	\end{itemize}
\end{fact}
\begin{proof}
The first three bullet follow directly from definition. We only need to verify the last bullet and smoothness of $f.$

For any point $q\in N\#\ssc\cap \ug$  adopt a normal coordinate chart on $N\#\ssc$ with respect to the intrinsic metric on $N\#\ssc$ and then extend to a Fermi coordinate chart adapted to $N\#\ssc$: $$U_q:(n_1,\cdots,n_d,l_1,\cdots,l_c).$$
Also set $n=(n_1,\cdots,n_d),l=(l_1,\cdots,l_c).$
	By properties of Fermi coordinate chart  \cite[Chapter 2]{AG}, we have that
	\begin{itemize}
		\item $N\#\ssc$ is the $n_1\cdots n_d$ coordinate plane,
		\item geodesics normal to $N\#\ssc$ are straight lines starting from the $n_1\cdots n_d$ plane,
		\item $\operatorname{dist}_{h_\sing}((n,d),N\#\ssc))=|l|,$
		\item $\rh$ depends only on $n$ not on $l.$
		\item we have 
		\begin{align}\label{metbd}
(1+A\e)		(|j|^2+|k|^2)\ge 	(h_\sing)_{(n,l)}((j,k),(j,k))\ge(1-A\e)(|j|^2+|k|^2),
		\end{align} with $A$ depending only on $N\#\ssc$ and absolute value means absolute value in coordinate system.
	\end{itemize}
	Thus, the map $f$ in the Fermi coordinate chart $U_q$ equals
	\begin{align}\label{defnfnl}
		f(n,l)=(n,\ai(\rh(n))l)
	\end{align} 
As $N\#\ssc\cap \ov{\ug}$ is compact, we can cover $\ug$ with finitely many such Fermi coordinate charts by (\ref{inc}). Smoothness of $f$ follows immediately.

On $\rh\m(\frac{1}{3},\frac{2}{3})$, $\ai(\rh)>0,$ so $f(n,l)$ is a one to one map and we have
\begin{align*}
	df_{(n,l)}(j,k)=\left(j,\ai'\left(\rh(n,l)\right)d\rh(j)l+\ai(\rh)k\right).
\end{align*}
Now choose $\e$ small enough so that $A\e<\frac{1}{2}$ for the finitely many Fermi coordinate chart in our covering of $\ug$. By (\ref{metbd}), we have
\begin{align*}
	&(f\du h_\sing)_{(n,l)}((j,k),(j,k))\\
	\ge& \frac{1}{2}|j|^2+\frac{1}{2}|\ai'\left(\rh(n,l)\right)d\rh(j)l+\ai(\rh)k|^2\\
	\ge&\frac{1}{2}|j|^2.
\end{align*}
On the other hand
\begin{align*}
	(\pi\du h_\sing)_{(n,l)}((j,k),(j,k))\le\frac{3}{2}|j^2|\le 3(f\du h_\sing)_{(n,l)}((j,k),(j,k)).
\end{align*}

If $df_{(n,l)}(j,k)=0$ then $j=0$ and consequently $k=0.$ Thus, $df_{(n,l)}$ is a linear isomorphism on $\rh\m(\frac{1}{3},\frac{2}{3})$.%Properness on $\rh\m(\frac{1}{3},\frac{19}{30})$ follows 
\end{proof}
Now extend the squeezing map $f$ by setting
\begin{align}f=
	\begin{cases}
	\textnormal{identity map},&\text{on }  U_{\frac{1}{10}}  (\sing \ssc),\\
		\pi,&\text{on } B_\nu(\reg^0(N\#\ssc))\setminus \ug.				
	\end{cases}
\end{align}
Then by Fact \ref{fctf}, $f$ is a well-defined smooth map when restricted to $U_\e(N\#\ssc).$ 

A special property of $f$ is as follows.
\begin{fact}\label{fctproj}
	If $T$ is a mod $2$ current supported in $U_\e(N\#\ssc)$ and $T-N\#\ssc=\pd Q$ with a $(d+1)$-dimensional mod $2$ current $Q$ supported in $U_\e(N\#\ssc)$, then $f(Q)$ is compactly supported in $U(\sing\ssc)$, i.e., \begin{itemize}
		\item $f(T)$ restricted to $U(\sing\ssc)\cp$ equals $N\#\ssc$ restricted to $U(\sing\ssc)\cp$ 
		\item  $f(T)$ restricted to  $U(\sing\ssc)$ is homologous to $N\#\ssc$ restricted to $U(\sing\ssc)$
	\end{itemize}
\end{fact}
\begin{proof}
	Suppose $T-N\#\ssc=\pd Q$ with $Q$ supported in $U_\e(N\#\ssc).$ Then
	\begin{align*}
		f(T)-N\#\ssc=f(T)-f(N\#\ssc)=\pd f(Q).
	\end{align*}
Recall that $f=\pi$ for $p\in U(\sing\ssc)\cp$ and for $p\in U(\sing\ssc)$ with $\rh(p)\in[\frac{2}{3},1].$ Thus $f(Q)$ restricted to $\left(U(\sing\ssc)\cap\rh\m[0,\frac{2}{3})\right)\cp$ is a $(d+1)$-dimensional mod $2$ current supported on a $d$-dimensional submanifold, which must be $0$ by Fact \ref{fctconc}. This implies that $f(Q)$ is supported on $U(\sing\ssc)\cap\rh\m[0,\frac{2}{3})$. We are done.
\end{proof}
\subsubsection{Gluing metrics}
Now with the squeezing map $f$ in hand, we are ready to define the gluing metric.
We further define
\begin{align}
h_{\operatorname{glued}}=
	\begin{cases}
	h_\sing=(h_\sing)|_{\ker \ed\pi}\oplus(h_\sing)|_{\ker^\perp \ed\pi},&\text{on }  U_{\frac{1}{10}}  (\sing \ssc),\\
\left((1-\ai(\rh))(h_\sing)|_{\ker \ed\pi}\right)+ f\du h_\sing&\textnormal{on }\ug\\h_\reg=(h_\sing)|_{\ker \ed\pi}+\pi\du h_\sing,&\text{on } B_\nu(\reg^0(N\#\ssc))\setminus \ug.				
	\end{cases}
\end{align}
We have
\begin{fact}\label{hgdc}
	The metric $\hgd$	is a smooth Riemannian metric when restricted to $U_\e(N\#\ssc)$ and we have
	\begin{align}\label{fshort}
		f\du h_\sing\le \hgd,
	\end{align}with
	\begin{align}\label{fconst}
f\du h_\sing=\hgd=h_\sing
	\end{align}when restricted to the tangent spaces to $N\#\ssc.$
\end{fact}
\begin{proof}
First, the three domains of definition are placed one by one according to distance from $\sing\ssc$. We only need to verify near the two interfaces.

Recall Definition \ref{defnai} and Definition \ref{defnrh}. Near the first interface between $U_{\frac{1}{10}}  (\sing \ssc)$ and $\ug$, we have $\hgd=h_\sing$ on $U_{\frac{1}{10}} $ and on $\ug\cap \rh\m([0,\frac{1}{3}])$ by	Fact \ref{fctf}, we have $$\hgd=0+\id\du h_\sing=h_\sing.$$ Smoothness across this first interface follows.
	
For the first interface between $\ug$ and $B_\nu(\reg^0(N\#\ssc))\setminus \ug$, on $\ug\cap\rh\m([\frac{2}{3},1])$ by	Fact \ref{fctf}, we have $$\hgd=(h_\sing)|_{\ker \ed\pi}+\pi\du h_\sing=h_\reg.$$ Smoothness across the second interface follows.

Finally, we have to verify that $\hgd$ is positive-definite. $h_\sing$ and $h_\reg$ are positive definite by construction, so we only need to consider $\hgd$ restricted to $\ug.$ On $\ug\cap\rh\m([0,\frac{1}{3}])$ we have $\hgd=h_\sing$ and $\ug\cap\rh\m([\frac{2}{3},1])$ we have $\hgd=h_\reg.$ 

On $\ug\cap\rh\m([\frac{1}{3},\frac{1}{2}]),$ when $\rh$ is close to $\frac{1}{3},$ $f$ is close to identity map so $\hgd\ge \frac{1}{2}h_\sing$ is positive definite. Thus, we are left with $\ug\cap\rh\m([\frac{1}{3}+a,\frac{1}{2}]),$ for small $a>0.$ Here by last bullet of Fact \ref{fctf}, we have
\begin{align*}
	\hgd\ge (1-\ai(a))h(h_\sing)|_{\ker \ed\pi}+\frac{1}{3}\pi\du h_\sing\ge\min\{1-\ai(a),\frac{1}{3}\}h_\reg,
\end{align*}which is positive definite.
\end{proof}The above fact immediately gives that
\begin{lem}
	$N\#\ssc$ is mod $2$ area-minimizing in $U_\e(N\#\ssc)$ with metric $\hgd.$
\end{lem}	
\begin{proof}
	Let $T$ be a mod $2$ current homologous to $N\#\ssc$ in $U_\e(N\#\ssc)$. By Lemma \ref{lemshort}, we have
	\begin{align}\label{eqmst1}
		\ms_{\hgd}(T)\ge \ms_{h_\sing}(f(T)).
	\end{align}
Recall Fact \ref{fctproj} and Fact \ref{calsing}. We have
\begin{align}
	&\ms_{ h_\sing}(f(T))\\=&	\ms_{ h_\sing}(f(T)\cap U(\sing\ssc))+	\ms_{h_\sing}(f(T)\cap (U(\sing\ssc))\cp)\\\ge&	\ms_{ h_\sing}(N\#\ssc\cap U(\sing\ssc))+	\ms_{h_\sing}(N\#\ssc\cap (U(\sing\ssc))\cp)\\\ge&\ms_{h_\sing}(N\#\ssc).]\label{eqmst2}
\end{align}Combining (\ref{eqmst1}) through (\ref{eqmst2}) and using (\ref{fconst}), we deduce that
\begin{align*}
	\ms_{\hgd}(T)\ge\ms_{\hgd}(N\#\ssc).
\end{align*}
We are done.
\end{proof}
	Now extend $\hgd$ restricted to  ${U_\e(N\#\ssc)}$, to a smooth Riemannian metric defined throughout $M.$ We will still use $\hgd$ to denote this extension of $\hgd$.
	\subsubsection{Adding a bump to make $N\#\ssc$ uniquely minimizing in $U_\e(N\#\ssc)$}
By \cite[Lemma 2.3.1]{ZLns}, there is a smooth function  $\be$ on $M$ with value between $[0,\frac{1}{2}]$ and \begin{align}\label{bez}
		\be\m(0)=\supp N\#\ssc\cup \ov{U\left(\sing\left( N\#\ssc\right)\right)}.
	\end{align}
	Define the new metric
	\begin{align*}
		h_\be=(1+\be)\hgd.
	\end{align*}
	By construction we have
	\begin{align*}
		h_\be\ge \hgd.
	\end{align*}
	Now suppose $T$ is a mod $2$ area-minimizing cycle in $U_\e(N\#\ssc)$ and homologous to $N\#\ssc$, which gives
	\begin{align*}
		\ms_{h_\be}(N\#\ssc)\le	\ms_{h_\be}(T).
	\end{align*}On the other hand, we have
	\begin{align*}
		\ms_{h_\be}(T)\ge\ms_{\hgd}(T)\ge \ms_{\hgd}(N\#\ssc)=\ms_{h_\be}(N\#\ssc), 
	\end{align*}	
	which implies
	\begin{align*}
		\ms_{h_\be}(T)=\ms_{h_\be}(N\#\ssc)=	\ms_{\hgd}(T)=\ms_{\hgd}(N\#\ssc).
	\end{align*}
		 Since on the complement of $\supp\left(N\#\ssc\right)\cup \ov{U\left(\sing\left( N\#\ssc\right)\right)}$ in $M,$ we have strict inequality $h_\be>\hgd.$ We deduce that $T$ must be supported in $\supp\left(N\#\ssc\right)\cup \ov{U\left(\sing\left( N\#\ssc\right)\right)}$.
	
	Note that $T$ cannot be supported in $\ov{U\left(\sing\left( N\#\ssc\right)\right)}$, since $$H_d(U\left(\sing\left( N\#\ssc\right)\right),\Z/2\Z)=0$$ yet $[T]\not=0\in H_d(M,\Z/2\Z).$
	
	Now by Fact \ref{const}, we deduce that $T$ restricted to the non-empty closed set $ N\#\ssc\cap \ov{U\left(\sing\left( N\#\ssc\right)\right)}\cp$ equals $N\#\ssc$ restricted to $ N\#\ssc\cap \ov{U\left(\sing\left( N\#\ssc\right)\right)}\cp.$ By \cite[Lemma 2.4.1]{ZLns} and the fact that having infinite order tangent to each other is a relatively closed condition, we deduce that $T= N\#\ssc.$
	
	To sum it up, we have proved 
	\begin{fact}\label{uni}
		The mod $2$ current $N\#\ssc$ is the unique area-minimizing mod $2$ current in its homology class in $U_\e(N\#\ssc)$ with respect to the metric $h_\be.$
	\end{fact}
	\subsubsection{Making metric large away from $ N\#\ssc$}This subsubsection is taken from the proof of \cite[Theorem 4.1]{YZj}.
	By the monotonicity formula for stationary varifolds \cite[Section 5.1]{WA}, there is a constant $A$ such that for any stationary mod $2$ current $T$ on $M$ with respect to the metric $h_\be$, we have
	\begin{align}\label{mon}
		\ms_{h_\be}(T|_{B_r(p)})\ge Ar^d,
	\end{align}
	provided $p\in \supp T$ and $r<\operatorname{InjRad} M$. Here by $B_r(p)$ we mean the ball of radius $r$ centered at $p$  in the metric $h_\be$. The symbol  $\operatorname{InjRad} M$ denotes the injectivity radius of $M$. By stationary current we mean that the mass, i.e., area, of $T$ is a critical point with respect to ambient diffeomorphism perturbations that fix $\pd T$. 
	
	By Lemma \ref{tube}, there exists $0<\e''
	<\e'<\e<\operatorname{InjRad}M$ such that \begin{align*}
		\uee\s \ue\s U_\e(N\#\ssc),
	\end{align*}
	with the distance between $\pd U_{\e'}(N\#\ssc)$ and $\pd U_\e(N\#\ssc)$ at least $\e''$, and the distance between $\pd\uee$ and $\pd  \ue$ at least $\e''.$
	
	Then let $\ga$ be a smooth function that equals to  $1$ on $\uee$ and equals to 
	\begin{align}\label{mxm}
		\max\Bigg\{1,\frac{\sqrt[d]{2\ms_{h_\be}(N\#\ssc)}}{\sqrt[d]{A} \e''}\Bigg\},
	\end{align}
	on the complement of $\ue$, and takes value between $1$ and (\ref{mxm}) on $\ue\setminus \uee.$ 
	\begin{defn}\label{defnh}
		Define
		\begin{align*}
			h=\ga^2h_\be.
		\end{align*}
	\end{defn}
	Now let $T$ be an area-minimizing mod $2$ current in $[\Si]$ with respect to metric $h$ on $M.$ There are two possibilities, either $\supp T$ is contained in $U_\e(N\#\ssc)$, or $\supp T$ contains a point in the complement of $U_\e(N\#\ssc)$. The latter cannot happen. To see this, note that $h$ is a constant multiple of $h_\be$ in the complement of $\ue,$ so $T$ restricted to $U_{\e'}(N\#\ssc)\cp$ is also a stationary mod $2$ current with respect to $h_\be$. Thus, one can apply (\ref{mon}) with $p\in U_\e(N\#\ssc)\cp\cap\supp T$ and $r=\e''$ so that $B_r(p)\s U_{\e'}(N\#\ssc)\cp$ to conclude that
	\begin{align*}
		&\ms_{h}(T)\\\ge& \ms_{h}(T|_{B_{\e''}(p)})=\max\Bigg\{1,\frac{\sqrt[d]{2\ms_{h_\be}(N\#\ssc)}}{\sqrt[d]{A} \e''}\Bigg\}^d\ms_{h_\be}(T|_{B_{\e''}(p)})\\\ge&\max\Bigg\{1,\frac{{2\ms_{h_\be}(N\#\ssc)}}{{A} (\e'')^d}\Bigg\} A(\e'')^d\\
		=&\max\{ A(\e'')^d,2\ms_{h_\be}(N\#\ssc)\}\\
		=&\max\{ A(\e'')^d,2\ms_h(N\#\ssc)\}.
	\end{align*} To sum it up, we have the proven the following:
	\begin{fact}\label{mint}
		In the smooth Riemannian metric $h$ (Definition \ref{defnh}) on $M,$ any other area-minimizing mod $2$ current $T$ in $[\Si]$ must have $\supp T\s U_\e(B\#\ssc).$
	\end{fact}
	\subsection{Wrapping up proof of Lemma \ref{nsmin}}
	Let $T$ be an area-minimizing mod $2$ current in $[\Si]$ on $M$ with respect to metric $h$. By Fact \ref{mint}, $T$ is supported in $U_\e(N\#\ssc).$ However, since $U_\e(N\#\ssc)$ is a subset of $M,$ we need to redetermine the relation between $[T]$ and $[N\#\ssc]$ when restricted to $U_\e(N\#\ssc).$
	
	Since we have $H_d(U_\e(N\#\ssc))=\Z/2\Z[N\#\ssc]$ by Lemma \ref{tube}, as $[T]\not=0,$ we deduce that $[T]= [N\#\ssc]$ as homology classes in $H_d(U_\e(N\#\ssc),\Z/2\Z)$.  Since $T$ is a mod $2$ area-minimizing current in $[\Si]$ with respect to $h$, we have
	\begin{align*}
\ms_{h_\be}(N\#\ssc)=\ms_h(N\#\ssc)\ge \ms_h(T)\ge \ms_{h_\be}(T)\ge\ms_{h_\be}(N\#\ssc).
	\end{align*}
In other words, $T$ is mod $2$ area-minimizing with respect to $h_\be$ in $U_\e(N\#\ssc)$. By Fact \ref{uni}, we deduce that $T=N\#\ssc$. To sum it up, we have proved that $N\#\ssc$ is the unique mod $2$ area-minimizing current in $[\Si]$ with respect to the metric $h.$

The bullet in the lemma follows directly from Lemma \ref{singssc} and Assumption  \ref{assumpns} and the fact that we have not changed the metric in a small neighborhood around $\sing\ssc.$ Set
	\begin{defn}\label{defnur}
		For $0<\eta\le\rr,$ we define  $U_\eta(\sing\ssc)$, the radius $\eta$ tubular neighborhood of $\sing\ssc,$ by
		\begin{align*}
			U_\eta(\sing\ssc)=\Ga\m(B_{\eta}^{d-s+c}\times S^s).
		\end{align*}
	\end{defn}
\subsection{Proof of Theorem \ref{thmzhang}}
Follows directly from Lemma \ref{nsmin}.
\section{Proof of Theorem \ref{thmm}}	\label{seccpt}In this section we will finish the proof of Theorem \ref{thmm}.	\subsection{Allard's regularity theorem}\label{secall}
	The following lemma does not need any assumption on $C$ beyond Assumption \ref{assumpc}. 
	\begin{lem}\label{localt}
		For any $\ee>0$, there is an open subset $\Om_h$ containing $h$ in the space of Riemannian metrics, such that any area-minimizing mod $2$ current $T$ in $[\Si]$ with respect to metric $g\in\Om_h$ satisfies 
		\begin{align}\label{cst}
			\supp T\s U_\ee(N\#\ssc),
		\end{align} and $T$ restricted to $$\left( U_\ee(\sing\ssc)\right)\cp$$ is a smooth submanifold with boundary such that  $T$ restricted to $\left( U_\ee(\sing\ssc)\right)\cp$ is the normal bundle exponential map image of a section $S_T$ in $\T^\perp N\#\ssc$, whose $C^4$ norm is at most $\ee$.
	\end{lem}
	\begin{proof}
\cite[Lemma 7.4.1]{ZLns} has the same statement for multiplicity $1$ integral currents. The same strategy works. For the reader's convenience, we reproduce the entire argument again here.

First, we claim that 
\begin{claim}\label{ccst}
	For any $\e'>0,$ there is an open set $\Om_h$ containing $h$ in the space of Riemannian metrics such that any area-minimizing mod $2$ current $T$ in $[\Si]$ with respect to metric $g\in\Om_h$ is of distance at most $\e'$ away to $N\#\ssc$ in the following three senses with respect to  our base metric $h$:\begin{enumerate}
		\item $T$ and $N\#\ssc$ as mod $2$ currents are at most $\e'$ away in the Whitney flat distance in the space of currents,
		\item The induced mass measures of $T$ and $N\#\ssc$ are at most $\e'$ with respect to weak convergence on varifolds,
		\item The supports $\supp T$ and $ N\#\ssc$ are at most $\e'$ away in Hausdorff distance in the space of closed subsets.
	\end{enumerate}
\end{claim}

Note that the above claim already implies (\ref{cst}).

Suppose the above claim is not true. Then we can find a sequence of Riemannian metrics $h_j\to h$ in smooth topology and a sequence of area-minimizing mod $2$ currents $T_j\in[\Si]$ with respect to metric $h_j$ that violates at least one of the three $\e'$ distance conditions.

Let us first prove that $\{T_j\}$ as a sequence of mod $2$ currents must converge to $N\#\ssc$. Let $\{T_{j_k}\}$ be any subsequence of $\{T_j\}$. By the compactness theorem for mod $2$ currents \cite{BWrc}, we can find a subsequence $\{T_{j_{k_m}}\}$ of $\{T_{j_k}\}$ such that $\lim_{m\to\infty} T_{j_{k_m}}$ exists. Since each $T_j$ is area-minimizing and $N\#\ssc$ is the unique area-minimizing representative of $[\Si]$ in $h$ (Lemma \ref{nsmin}) we deduce that $\lim_{m\to\infty}T_{j_{k_m}}= N\#\ssc$ as mod $2$ currents by \cite[Theorem 34.5]{LS}. Now recall the classical fact that if a every subsequence of a sequence has a subsequence that coverges to the the same limit, then the original sequence must converge to the same limit. We deduce that $T_j$ converges to $N\#\ssc$ as mod $2$ currents.

Since $N\#\ssc$ and each $T_j$ are all area-minimizing, \cite[Theorem 34.5]{LS} also implies that $T_j\to N\#\ssc$ as their induced mass measures. Finally $\supp T_j$ converge to $ N\#\ssc$ in Hausdorff distance by upper semi-continuity of densities of stationary varifolds \cite[Corollary 17.8]{LS}. Thus we have arrived at a contradiction.

To prove that $T$ can be expressed as the normal bundle exponential map image of  $N\#\ssc$ with  $\ee$ small norm is a classical application of Allard's regularity theorem \cite[Section 8, Regularity Theorem]{WA}. We will only give a sketch. 

Roughly speaking the core of Allard's regularity theorem \cite[Section 8, Theorem 8.19]{WA} is that volume close to balls implies close to balls as submanifolds. Roughly speaking, if the area of a $d$-dimensional stationary varifold $V$ restricted to a $(d+c)$-dimensional geodesic ball centered at $p$ on $M$ is at most $\kappa$ off from the area of a standard $d$-dimensional ball in $\R^{d+c}$, then the varifold $V$ is $\iota$ close as submanifolds to a $d$-dimensional standard ball in a neighborhood of $p$ in $M$.

Precisely speaking, Allard's regularity theorem implies the following. For any $\iota>0,$ there is $R,\kappa>0,$ such that if we have a $d$-dimensional stationary varifold $V$ in $M$ with metric $h$ and the area of $V$ inside a radius $R$ geodesic ball centered at $p$ on $M$ is at most $\kappa$ off from the volume of a $d$-dimensional radius $R$ standard  ball  in Euclidean space, then in $B_{\frac{1}{2}R}(p)$ on $M$, $V$ is $\iota$ close as a the graph of a function to the exponentiated image of the tangent plane $P$ to $V$ at $p.$ %Moreover, $\iota$ depends only on $R,a,\kappa,h$ and $\e''\to 0$ provided $\frac{\kappa}{R}\to 0$.
(Allard's original theorem is stated with $C^{1,\ai}$ bounds for varifolds with bounded mean curvature in Euclidean space. However, routine elliptic PDE arguments like \cite[Theorem 5.2.15 (7)]{HF}, upgrades  $C^{1,\ai}$ estimates into $C^{k,\ai_k}$ estimate with all $k>0$.)

Now cover $\left(U_\ee(\sing\ssc)\right)\cp$ by finitely many radius $\e''$ balls such that none of the balls in the covering intersect $\sing\left( N\#\ssc\right)$. Since $N\#\ssc$ restricted to $\left(U_\ee(\sing\ssc)\right)\cp$ is a smooth submanifold, for any $\e'>0,$ there is $\e''>0$ such that $N\#\ssc$ restricted to any of the radius $\e''$ balls in the covering has area at most $\e'$ off from that of a radius $\e''$ standard $d$-dimensional ball in $\R^{d+c}$.

Now since the area-minimizing current $T$ is at most $\e'$ away from $N\#\ssc$ as induced mass measures by Claim \ref{ccst}, we obtain that $T$ restricted to each radius $\e''$ ball in the covering has area  at most $2\e'$  off from that of a $d$-dimensional radius $\e''$ standard ball in $\R^{d+c}$. Now we can apply Allard's regularity theorem to finish off the proof, by making $\e'$ small. By making $\e'$ arbitrarily small, we can take $\ee$ arbitrarily small. We are done.
\end{proof}
	\subsection{Veronese $\rpt$ cone products give persistent singular sets}
	Recall from our plan of proof that we add to our area-minimizing representative $N\#\ssc$ in Lemma \ref{nsmin} singular sets that are locally products of $C(\vrpt)$ with spheres in Fact \ref{fctcb}. We will prove the following
	\begin{defn}(IN THIS SECTION ONLY)
		For $d\ge 3,c\ge 2$, set $s=d-3,$ and take $C$ in $\R^{c+3}$ to be $C(\vrpt)$ (Definition \ref{defnvc}) in the slice $\R^5\times\{0\}^{c-2}\cong \VV\times\{0\}^{c-2}.$
	\end{defn}
By \cite[Lemma 2.9]{ZLa} and Theorem \ref{thmc}, $C$ is mod $2$ area-minimizing.	

	Our goal in this section is to achieve the following.
	\begin{lem}\label{cpst}
		The subset $\sing\sfc$ of the area-minimizing mod $2$ current $N\#\si^{d-3}(C)$ obtained in Lemma \ref{nsmin} is a persistent singular set (Definition \ref{psing}), provided  $d\ge 3,c\ge 2.$
	\end{lem}
	The actual proof is straightforward. Roughly speaking, recall Lemma \ref{localt}. In $U_{\ee}(\sing\sfc),$ for $(d-3)$-dimensional Hausdorff measure almost every $p\in \{0\}^{c+3}\times S^{d-3}$, the intersection
	\begin{align*}
		N\#\sfc\cap \pi_{S^{d-3}}\m(p),
	\end{align*}
	gives a mod $2$ current $T_p$ with boundary $\rpt$ and $T_p$ is smooth near its boundary. This implies that $T_p$ must has an interior singular point, since $\rpt$ cannot bound a smooth $3$-manifold. 
	\subsection{Proving Lemma \ref{cpst}}\label{secrpt}
	Recall Lemma \ref{localt}. Take an area-minimizing mod $2$ current $T\in [\Si]$ with respect to a metric $g\in\Om_h.$  Restricted to $U_{2\ee}(\sing\sfc),$ consider the projection
	\begin{align*}
		\pi_{S^{d-3}},
	\end{align*}defined by the orthogonal projection onto the $S^{d-3}$ factor in $B_{2\ee}^{c+3}\times S^{d-3}$. Then by \cite[4.3.6,4.3.13]{HF}, for $(d-3)$ dimensional Hausdorff measure almost every $p\in S^{d-3}$, the slicing of $T|_{U_{2\ee}(\sing\sfc)}$ by $\pi_{S^{d-3}}$,
	\begin{align*}
		T_p=\ri{T,\pi_{S^{d-3}},p}
	\end{align*} is a $3$-dimensional mod $2$ current. Intuitively, the reader can just understand this as saying that the intersection of $T|_{U_{2\ee}(\sing\sfc)}$ with $\pi_{S^{d-3}}\m(p)$ is a $3$-dimensional mod $2$ current 
	\begin{align*}
		T_p,
	\end{align*} 
	for $(d-3)$-dimensional Hausdorff measure almost every $p\in S^{d-3}.$

	Suppose for some $p\in S^{d-3}$, every point $q$ with $$q\in \pi_{S^{d-3}}\m(p)\cap\supp T|_{U_{2\ee}(\sing\sfc)}$$ belongs to the regular set of $T$. Then since the regular set of $T$ is relatively open, we deduce that there is an open neighborhood $U_p$ of $p$ inside $S^{d-3}$ such that 
	\begin{align*}
	\pi_{S^{d-3}}\m(U_p)\cap\supp T|_{U_{2\ee}(\sing\sfc)}
	\end{align*}belongs to the regular set of $T.$ 	 By Sard's Theorem (\cite[Theorem 6.10]{JL}) this implies that for $(d-3)$-dimensional Hausdorff measure almost every $q\in U_p$, the current $T_q$ is a smooth submanifold with boundary. However, this is impossible.
	
To see this, by Lemma \ref{localt}, $T$ is completely smooth in $$ U_{\ee}(\sing\sfc)\cp.$$ By Sard's Theorem (\cite[Theorem 6.10]{JL}) this implies that $T_p$ is smooth in $$U_{\ee}(\sing\sfc)\cp$$ for $(d-3)$-dimensional Hausdorff measure almost every $p\in S^{d-3}.$

However, such $T_p$ necessarily has boundary $$\pd T_p$$ diffeomorphic to $\rpt$. Too see this, note that by Lemma \ref{localt} as $T$ is the normal bundle exponential map image of a section $S_T$ in  $\T^\perp N\#\ssc\cap U_{2\ee}(\sing\sfc)\setminus U_{\ee}(\sing\sfc)$ with $C^\infty$ norm bounded $\ee$. Thus $\pd T_p$ is diffeomorphic to the section $S_T$ restricted to the slice of $\pd(N\#\ssc\cap U_{2\ee}(\sing\sfc))$, which is $\rpt.$

 This implies that for $(d-3)$-dimensional Hausdorff measure almost every  $p\in S^{d-3}$, the current $T_p$ must have a singular set in its interior, as $\rpt$ cannot be the boundary of a compact $3$-manifold. For example this follows as $\rpt$ has mod $2$ Euler characteristic $1$, while surfaces that bound has mod $2$ Euler characteristic $0$ \cite[Lemma 7.1]{RGRef}.  We have arrived at a contradiction.

In other words we have proved that $\pi_{S^{d-3}}\m(p)\cap \sing T$ has  zero dimensional Hausdorff measure at least $1$ for $(d-3)$-dimensional Hausdorff measure almost every $p\in S^{d-3}$. Now integrate this estimate using Eilenberg's inequality \cite[2.10.25]{HF}, we deduce that $\sing T$ has positive $(d-3)$-dimensional Hausdorff measure. We are done.
\section{Discussions}\label{secdis}
\subsection{Comparison with the integral case in \cite{ZLns}}
The reader might wonder why the mod $2$ results in this case are much better than the integral case in \cite{ZLns}. Roughly speaking, we use two kinds of persistent singular sets in \cite{ZLns}, one is the Veronese cone over $\cpt$ times spheres and the other is almost orthogonal transverse intersections, which contributes persistent singular sets of dimensions $(d-5)$ and $(d-c)$, respectively, for $d\ge 5, c\ge 3$ or $3\le c\le d\le 4.$ In the mod $2$, the existence of area-minimizing Veronese $\rpt$ cone significantly sharpens the first kind of singularity.
\subsection{Moderations}
The notion of moderations in Section \ref{secmod} is inspired by Frank Morgan's works \cite{FMcalv,FMeu}. Roughly speaking, Frank Morgan has found that some calibration arguments can be interpreted as the sum of volume of projections onto linear subspaces, thus working even in the case of mod $2$ homology. Upgrade from sums of projections to integral of projections yields our notion of moderations. We expect this new method of proving area-minimizing properties can help solve several other longstanding questions of area-minimizing cones, e.g., the remaining isoparametric focal cones in \cite{YZjis}.`
\subsection{About gluing minimizing currents}
Yongsheng Zhang's works \cite{YZa,YZj} has devised very general ways of gluing calibrations. Zhang's work also applies in the mod $2$ case when there is a Lawlor's vanishing retractions \cite{GL}. Unfortunately, in our case, we have neither calibrations nor Lawlor's retractions, so we have to find a genuinely different procedure for gluing mod $2$ area-minimizing currents. There are two key components in our argument. The first is the squeezing map in Section \ref{secsq}. The squeezing map can be seen as a forced area-non-increasing deformation by altering the metric when there are no Lawlor's area-non-increasing retractions. The second key is the metric shorting Lemma \ref{lemshort}. It can be viewed as a vast generalization of Zhang's ideas of pulling back metrics along retractions \cite{YZa,YZj}.

The squeezing and metric shorting arguments cannot fully replace Zhang's gluing arguments for calibrations, which in some sense controls both global integral and real homology \cite{ZLhs}. For instance, the results in \cite{ZLa2} cannot be proved by using the gluing arguments in this manuscript, as we no longer have the regular neighborhoods available in more complicated singularities, while global calibrations control homology directly.
\subsection{Extensions}
Indeed, if we work in homology with other coefficients, i.e., $R=\R,\Z$ or $\Z/n\Z$ with $n\ge 3,$ our gluing argument requires the cone $aC$ to be $R$-area-minimizing for any $a\in R\setminus\{0\}.$ Thus, analogues of Theorem \ref{thmc} hold for such cones. This happens naturally when $R=\Z/3\Z$ as well, so analogues of Theorem \ref{thmm} hold for $\Z/3\Z$ area-minimizing currents with $d\ge 1,c\ge1$ by using $C=$ the triple junction and the dimension lower bound on the singular set is $(d-1)$.
	\printbibliography
	\includepdf[pages=-]{verifications.pdf}
\end{document}